\documentclass[12pt]{article}
\usepackage{amsmath ,tikz}
\usepackage{amssymb}
\usepackage{amsthm}
\usepackage[all]{xy}
\usepackage{color}
\usepackage[shortlabels]{enumitem}
\usepackage{hyperref}
\usepackage{pgf,tikz}
\usepackage{mathrsfs}
\usepackage{setspace}
\usepackage[letterpaper,left=1.5in,right=1in,top=1in,bottom=1in]{geometry}
\usetikzlibrary{arrows}

\newcommand{\R}{\mathbb{R}} 
\newcommand{\Z}{\mathbb{Z}}   
\newcommand{\C}{\mathbb{C}}   
\newcommand{\Q}{\mathbb{Q}}
\newcommand{\N}{\mathbb{N}}
\newcommand{\He}{\mathbb{H}}

\newcommand{\ch}{\text{ch}}

\newcommand{\mf}{\mathfrak}

\DeclareMathOperator{\spn}{span}
\DeclareMathOperator{\aut}{Aut}

\DeclareMathOperator{\ed}{End}
\DeclareMathOperator{\res}{Res}
\DeclareMathOperator{\id}{id}

\DeclareMathOperator{\im}{Im}
\DeclareMathOperator{\re}{Re}
\DeclareMathOperator{\rk}{rk}

\DeclareMathOperator{\sgn}{sgn}

\DeclareMathOperator{\ad}{ad}

\DeclareMathOperator{\wt}{wt}

\theoremstyle{plain}
\newtheorem{theorem}{Theorem}[section]
\newtheorem{theorem*}{Theorem}
\newtheorem{lemma}[theorem]{Lemma}
\newtheorem{corollary}[theorem]{Corollary}
\newtheorem{proposition}[theorem]{Proposition}

\theoremstyle{definition}
\newtheorem{definition}[theorem]{Definition}
\newtheorem{example}[theorem]{Example}
\newtheorem{remark}[theorem]{Remark}

\title{Placeholder title}

\begin{document}

\pagenumbering{gobble}

\begin{singlespace}
\begin{center}
Three Families of Lie Algebras \break
of Exponential Growth from \break
Vertex Operator Algebras
\end{center}

\vspace{5.5cm}

\begin{center}
A Dissertation \break
Presented to the Faculty of the Graduate School \break
Of \break
Yale University \break
In Candidacy for the Degree of \break
Doctor of Philosophy
\end{center}

\vspace{6cm}

\begin{center}
By \break
Gabriel B. Legros
\end{center}

\vspace{1cm}

\begin{center}
Dissertation Director : Igor. B. Frenkel
\end{center}

\vspace{1cm}

\begin{center}
June 2021
\end{center}

\end{singlespace}

\newpage
\pagenumbering{roman}
\setcounter{page}{1}

\tableofcontents

\newpage

\pagenumbering{arabic}
\setcounter{page}{1}

\section*{Acknowledgements}

\hspace{\parindent}I would like to thank my advisor professor Frenkel.  Without his help and patience, this dissertation would not have been possible.  I have learned so much from professor Frenkel and will be eternally grateful.

Secondly, I would also like to thank the staff of the Yale department of mathematics who supported me and were always friendly and available.  I would also like to thank Yale University and the McKenna foundation for providing my fellowship.

Thirdly, I would like to thank my readers professor Andrew Linshaw and professor Andrew Neitzke for taking the time to read this dissertation.

\newpage
\section{Introduction}

\paragraph{ } In contrast from the theory of finite dimensional Lie algebras, infinite dimensional Lie algebras are poorly understood.  Extensions of semisimple Lie algebras lead to the so-called affine Lie algebras, which themselves can be generalized as Kac-Moody algebras, which again can be generalized as Borcherds-Kac-Moody (BKM) algebras.  As we progress through this sequence, we retain a general structure :  all of these algebras can be described using generators and relations.  The Weyl denominator formula of the finite case becomes the Weyl-Kac-Borcherds denominator formula for BKM-algebras.  On the other hand, we lose explicit realizations of these Lie algebras.  All finite dimensional Lie algebras can be realized as subalgebras of $\mf{gl}(n)$ and the affine Lie algebras can be realized as current algebras.  Beyond that point, finding realizations is significantly more difficult.

Following earlier work from \cite{frenkelkac} (see also \cite{segal}), which introduced vertex operators representations for affine Lie algebras, I.B. Frenkel in \cite{frenkel} constructed representations for any Kac-Moody algebra $g(A)$ with a symmetric matrix $A$ and for double affine Lie algebras.  This particular representation is almost a realization;  the Lie brackets between different elements of $g(A)$ can be explicitly computed hence we no longer need to rely on relations, however we are still left with generators and no explicit vector space.  This representation was achieved in the following way.
Given a lattice $L$, let $h = L \otimes \C$ and define
$$V_L = S(h \otimes t^{-1} \C[t^{-1}]) \otimes \C\{L\}$$
where $\C\{L\}$ is a twisted version of the group algebra $\C[L]$.  Then $V_L$ can be given a $g(A)$-structure via its generators.  For a lattice element $\alpha$, \cite{frenkel} uses the classical vertex operators
$$\alpha \to Y(\alpha,z) \in \ed[V_L][[z,z^{-1}]]$$
which decomposes as $Y(\alpha,z) = \sum\limits_{n \in \Z} Y_n(\alpha) z^{-n-1}$.  If we write $h(-n) = h \otimes t^{-n}$ then operators $Y(h(-n),z)$ were also introduced as well as for products of elements of the form $h(-n)$.  In particular, for generators of $g(A)$, it was shown that
\begin{align*}
e_{\alpha_i} &\rightarrow Y_0(\alpha_i) \\
f_{\alpha_i} &\rightarrow -Y_0(-\alpha_i) \\
h_i &\rightarrow Y_0(h_i(-1)) = h_i \otimes 1
\end{align*}
define a representation of $g(A)$ on $V_L$.  This representation had many advantages over an abstract highest weight module : both real roots and isotropic roots were fully realized, i.e. it is possible to explicitly write down the operators corresponding to these two types of roots using vertex operators.  Furthermore, the no-ghost theorem from \cite{noghost} was used to prove some upper bounds on the multiplicities of Kac-Moody algebras of hyperbolic type of rank $26$, and led to a conjecture on the multiplicities of Kac-Moody algebras of hypebrolic type of any rank.

Later R. E. Borcherds \cite{borcherds0}, \cite{borcherds1} extended the operators defined in \cite{frenkel} to all of $V_L$, with an added identification $\alpha \to e^{\alpha}$, turning $V_L$ from a $g(A)$-module into a structure of its own, now known as a vertex algebra.  Given $u \in V_L$ he defines $Y(u,z) \in \ed[V_L][[z,z^{-1}]]$ where again, 
$$Y(u,z) = \sum\limits_{n \in \Z} u_n z^{-n-1}.$$
Given a second element $v \in V_L$ he then shows that
$$[u,v] = u_0 v$$
is almost a Lie bracket, only missing antisymmetry.  This problem can be resolved with the added structure of a Vertex Operator Algebra (VOA), which adds a Virasoro-structure on $V_L$.  Indeed, if we define
$$L_n = \frac{1}{2} \sum\limits_{i=1}^d \sum\limits_{k \in \Z} : h_i(n-k)h_i(k) :$$
where $d$ is the rank of the lattice $L$, then the operators $L_n$ satisfy the relations of the Virasoro Lie algebra with central charge $d$.  Now the product $[u,v]$ becomes a Lie bracket in the quotient $V/L_{-1}V$.  Furthermore we may identify $g(A)$ as a subalgebra of $V_L$ via the map
\begin{align*}
e_{\alpha_i} &\rightarrow e^{\alpha_i} \\
f_{\alpha_i} &\rightarrow -e^{-\alpha_i} \\
h_i &\rightarrow h_i(-1).
\end{align*}

Since $V_L$ is a module over the Virasoro Lie algebra, we may consider a subspace
$$P^i = \{ v \in V, L_0 v = i v, L_n v = 0 \text{ for } n > 0 \}.$$
The quotient $P^1/L_{-1}P^0$ is then a Lie algebra known as the Lie algebra of physical states, which also contains $g(A)$.

The Fock space $V_L$ plays a particular role in string theory, where it is thought of as a compactified bosonic string.  There is a bilinear form $(.,.)$ on $V_L$ which can be constructed in the following way :
given two elements $\alpha,\beta \in L$ define $(e^\alpha,e^\beta) = \delta_{\alpha+\beta,0}$.  Furthermore, given $\alpha(-n)$ for $n>0$, we define an adjoint operator $\alpha(n)$ which acts on $S(h \otimes t^{-1} \C[t^{-1}])$ via
$$\alpha(n) \cdot \beta(-m) = (\alpha,\beta) \delta_{m+n}$$
$$\alpha(n) \cdot e^{\beta} = 0$$
and extended to all of $V_L$ by derivation.  In this setting, the operators $\alpha(n)$ are known as annihilation operators and the operators $\alpha(-n)$, which act by left multiplication, are known as creation operators.  Then these definitions inductively define a bilinear form $(.,.)$ on all of $V_L$.

It turns out that $(.,.)$ is in some sense an invariant form with respect to $[.,.]$ that contains $L_{-1}P^0$ hence we may factor $P^1$ further by the radical of this bilinear form to obtain a Lie algebra $g_L = P^1/ (.,.)$.  

If $V_L=V_{II_{25,1}}$ is constructed from the unique Lorentzian unimodular even lattice of rank 26, we may use the no-ghost theorem \cite{noghost} to describe this Lie algebra explicitly via the spectrum-generating algebra (see also \cite{Brower}).  More precisely, given $\alpha \neq 0 \in L$ and an isotropic element $c \in L \otimes \R$ satisfying $(\alpha,c) = 1$, let $a$ be an element of $L \otimes \R$ orthogonal to both $\alpha$ and $c$.  Then we may consider the operators
$$A^{a}_m = \res_z Y(a_i(-1) e^{mc},z),$$
which define the spaces
$$T(\alpha,c) = \spn \{ A^{a_1}_{-m_1} A^{a_2}_{-m_2} ... A^{a_k}_{-m_k} e^{\alpha + Mc} \}$$
where $M = m_1 + m_2 + ... + m_k = 1 - \frac{(\alpha,\alpha)}{2}$.  The no-ghost theorem then implies that
$$(g_L)_{\alpha} \simeq T(\alpha,c).$$
In particular, because there are $24$ linearly independent such choices of $a$, we find that the multiplicities of $(g_L)_{\alpha}$ are $p_{24} (1-\frac{(\alpha,\alpha)}{2})$, where $p_{24}(n)$ is the coefficient of $q^n$ in
$$\prod\limits_{n = 1}^\infty \frac{1}{(1 - q^n)^{24}}.$$  This is a very special result among infinite-dimensional Lie algebras.

Borcherds proves that $g_{II_{25,1}}$ is a BKM-algebra, describes all of its roots and provides a denominator formula.  He initially named it the monster Lie algebra, but it is now known as the \emph{fake monster Lie algebra} (the true monster Lie algebra was constructed later using the moonshine module $V^{\natural}$ from \cite{voa}).  Regardless, its definition is straightforward and $II_{25,1}$ possesses many properties that makes it an interesting case study.
 
Indeed, given an isotropic element $\rho \in II_{25,1}$ we may construct a lattice $\rho^{\perp}/\rho$.  This new lattice will be a unimodular, positive-definite lattice of rank $24$.  There are $24$ such lattices : the Niemeier lattices, and all of them are obtained from $II_{25,1}$ in this way.  Borcherds used the Leech lattice specifically in \cite{borcherds1} but we may also attempt to study $g_{II_{25,1}}$ from the perspective of the other $23$ lattices.  Although the denominator formula remains the same, it can be expanded in different ways using the theory of Borcherds products as in \cite{gritsenko}.  In this work we will show how we may find a structure on $g_{II_{25,1}}$ with respect to each Niemeier lattices with the help of the no-ghost theorem, culminating in the following theorem :

\begin{theorem*} (Section \ref{subsectioncharacter1})
Suppose $N$ is a Niemeier lattice other than the Leech lattice and $\rho$ an isotropic vector of $II_{25,1}$ such that $\rho^\perp / \rho = N$.  Let $L_n = \{ \alpha \in II_{25,1} | (\alpha,\rho) = n\}$ and $g_{L_n} = \bigoplus\limits_{\alpha \in L_n} (g_L)_{\alpha}$.  If $\widehat{g}(N) = \widehat{g}(R(N))$ where $R(N)$ are the real roots of $N$, then 
$$g_{L_0} \simeq \widehat{g}(N),$$
as Lie algebras, $g_L = \bigoplus\limits_{n \in \Z} g_{L_n}$ is a graded $\widehat{g}(N)$-module and
$$g_{L_1} \simeq g_{L_{-1}} \simeq V_N$$
linearly.  Furthermore, $(g_{L_{-1}})^* = g_{L_1}$ with respect to the bilinear form $(.,.)$ on $V_L$.  
\end{theorem*}

The construction of the fake monster Lie algebra is easy to generalize to other versions of the no-ghost theorem.  In \cite{borcherds0}, Borcherds briefly explained how we can obtain a Lie bracket from a Fock space constructed from an odd lattice rather than an even lattice.  This seemingly small change requires that we replace the theory of vertex operator algebras with superalgebras (SVOA).  We may also construct SVOAs by starting with the Fock space $V_L$ and tensoring it with an anticommuting algebra $F_d$ of "fermions" (see \cite{friedan}, \cite{feingoldfrenkelries},\cite{tsukada}), where $d$ is the rank of $L$.  From there the Virasoro structure on $V_L$ extends to a super-Virasoro structure on $V_L \otimes F_d$.   We end up with additional operators $G_{n + \frac{1}{2}}$, and we may then modify our definition of $P^i$ to include these operators, i.e.

$$P^i = \{v \in V_L \otimes F_d, L_nv = G_{n-\frac{1}{2}} v= 0, n > 0, L_0v = i v \}.$$

We replace the no-ghost theorem in dimension 26 with the $N=1$ Neveu-Schwarz no-ghost theorem for central charge $15$, in our case a Lorentzian lattice of rank 10 with 10 fermions attached. $P^1$ should be replaced with $P^{1/2}$ and then we may again define a Lie algebra structure.  This was done in \cite{scheit2}, where it is also shown that $g_{NS} = P^{1/2}/(.,.)$ is a BKM-algebra.  Just as in the $N=0$ case, we can find an explicit basis via the spectrum-generating algebra.  This is well-known to physicists (see \cite{browerfriedman}) but in mathematics involve peculiar operators which require special treatment.  In this work, we define operators $A^a_m$ and $B^a_r$ carefully and prove the following theorem,

\begin{theorem*} (Section \ref{subsectionDDFN=1})
Let $\alpha$ be a root of $g_{NS}$.  Then
$$(g_{NS})_{\alpha} = \spn \{ A^{a_{i_1}}_{-m_1} ... A^{a_{i_k}}_{-m_k} B^{a_{j_1}}_{-r_1}...B^{a_{j_l}}_{-r_l} e^{\alpha + Nc} \}$$
for $N = m_1 + ...  m_k + r_1 + ... + r_l = \frac{1}{2} - \frac{(\alpha,\alpha)}{2}$.
\end{theorem*}

Finally, we may then consider what happens if we add a $N=2$ structure to this vertex operator superalgebra.  Is it also possible to define two Lie algebras, with the help of the equivalent no-ghost theorem in this setting?  In this work we give a positive answer to these questions and construct a Lie algebra of physical states $\widetilde{P}^{0,0}$ based on the lattice $II_{2,2}$ and a smaller Lie algebra $g^{(2)}_{NS}$ parametrized by the light-cone
$$\{ \alpha \in II_{2,2}, (\alpha,\alpha) = 0\}.$$
The Lie algebra $g^{(2)}_{NS}$ was first found in \cite{kugel}, with techniques that rely on heavy use of BRST cohomology.  In this work we instead take a more traditional approach and rely more heavily on the equivalent of the no-ghost theorem in $N=2$.  Decompose $II_{2,2}$ as $L^+ \oplus L^-$ and an element $\lambda \in II_{2,2}$ as $\lambda^+ + \lambda^-$.  We describe the structure of this Lie algebra precisely, which results in the following theorem :

\begin{theorem*} (Section \ref{subsectionN=2})
If $L = II_{2,2} = L^+ \oplus L^-$, the structure of $g^{(2)}_{NS}$ is given by
$$g^{(2)}_{NS} = \C e^0 \oplus A_0 \oplus A_{\infty} \oplus \bigoplus\limits_{r \in \Q^\times} A_r$$
where $A_r = \spn\{e^{m u_r + n v_r}\}$ for $r =p/q \in \Q^{\times}$ is the Lie algebra generated by some elements $e^{\pm u_r}, e^{\pm v_r}$ with Lie bracket 
$$[e^{m u_r + n v_r}, e^{k u_r + l v_r}] = (ml - nk) pq e^{(m+k) u_q + (n+l) v_q},$$
each of these subalgebras commute with each other and
$$A_0 = \spn \{ e^\lambda , \lambda^+ = 0, \lambda^- \neq 0 \}, A_{\infty} = \spn \{ e^{\lambda}, \lambda^- = 0, \lambda^+ \neq 0 \},$$
are abelian.
\end{theorem*}

Furthermore, if we think of elements of $II_{2,2}$ as $2 \times 2$ matrices, there is a natural $SL(2,\Z)$ action on the light-cone which can be lifted :
\begin{theorem*} (Section \ref{subsectionN=2})
Each left action of $SL(2,\Z)$ on $II_{2,2}$ lifts to a Lie algebra isomorphism in $g^{(2)}_{NS}$ which preserves each $A_r, r \in \Q \cup \infty$.  Each right action of $SL(2,\Z)$ on $II_{2,2}$ induces vector space isomorphisms $A_r \simeq A_{r'}$ for $r,r' \in \Q$.
\end{theorem*}

This dissertation is organized as follows.

In chapter 2, we review the basic definitions of BKM algebras as well as the Borcherds-Weyl-Kac denominator formula and how it applies to untwisted affine Lie algebras.

In chapter 3, we review the abstract notion of Vertex Operator algebras, including $N=1$ and $N=2$ SVOAs as well as basic properties such as tensor products and modules.  The dual module from \cite{frenkelhuanglepowsky} which provides important properties for $(.,.)$ is also reviewed.

In chapter 4, we summarize the notions of BRST cohomology introduced in \cite{feigin},\cite{frenkelgarlandzuckerman}, but using the modern adaptations found in \cite{lianzuckerman1}, \cite{lianzuckerman2}.  These notions establish alternate methods to study Lie algebra of physical states.  We do not need these techniques for our new results, but they are included due to their significance for the theory as a whole.  These techniques also provide an additional structure in $N=1$ which extends the Lie algebra to a Lie superalgebra as done in \cite{scheit} and in an unpublished paper by Lian, Moore and Zuckerman \cite{lianmoorezuckerman}.

In chapter 5, we define vertex operator algebras $V_L$ based on an even lattice $L$ and the resulting Lie algebras $P^1/L_{-1} P^0$ and $g_L = P^1 / (.,.)$ from the no-ghost theorem.  We also describe the spectrum-generating algebra completely.

In chapter 6, we tensor $V_L$ with the SVOA $F_d$ on $d$ fermions constructed in \cite{friedan}, \cite{feingoldfrenkelries} and \cite{tsukada}.  We recover the Lie algebra of physical states $P^{1/2}/G_{-1/2} P^0$ of \cite{scheit2} as well as its quotient $g_{NS}$, and we describe the spectrum-generating algebra for the quotient, yielding an explicit basis.

In chapter 7, we define the Lie algebra of physical states $\widetilde{P}^{0,0}$ in $N=2$ and describe the structure of its quotient $g^{(2)}_{NS}$ completely.  We lift the $SL(2,\Z)$ action on $II_{2,2}$ to a left and right action of $g^{(2)}_{NS}$ and prove a lower bound for the multiplicities of $\widetilde{P}^{0,0}$.

In chapter 8, we review the theory of Borcherds products from \cite{borcherds3} and important theorems from the theory of Jacobi forms.  These results are important for chapter 9.

In chapter 9, we study the structure of $g_L$ for $L=II_{25,1}$ further, first reviewing the results from \cite{borcherds1} then extending to any decomposition $II_{25,1} = N \oplus II_{1,1}$ where $N$ is a Niemeier lattice.  Finally, we compare this algebraic structure to the expansion of the Weyl-Kac-Borcherds denominator formula of $g_L$ with respect to $N$ as described in \cite{gritsenko}. 

\section{Kac-Moody algebras and Borcherds-Kac-Moody algebras} Most of the information in this section can be found in \cite{kac} and \cite{wakimoto}.

Let $A = (a_{ij})_{1 \leq i,j \leq n}$ be an $n \times n$ real matrix.  We say $A$ is a BKM-matrix if $A$ satisfies the following conditions :
\begin{enumerate}
\item $a_{ii} = 2$ or $a_{ii} \leq 0$.
\item $i \neq j \Rightarrow a_{ij} \leq 0$.
\item $a_{ij} = 0 \Leftrightarrow a_{ji} = 0$.
\item If $a_{ii} = 2$, then $a_{ij} \in \Z$ for all $j$.
\end{enumerate}
Write $I^{re} = \{ 1 \leq i \leq n | a_{ii} = 2 \}$ and $I^{im} = \{ 1 \leq i \leq n | a_{ii} \leq 0 \}$.  If $I^{re} = \{ 1 ,..., n \}$, then $A$ is known as a generalized Cartan matrix.  We may associate to a BKM-matrix a diagram, known as \emph{Dynkin diagram} in the following way :
\begin{enumerate}
\item For each row of the matrix is a vertex, which we denote by $\alpha_i$.
\item If $i \neq j$ and $a_{ij} \neq 0$, there is an edge between $\alpha_i$ and $\alpha_j$ of multiplicity $(-a_{ij},-a_{ji})$.
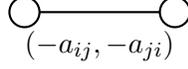
\begin{center}
  \begin{tikzpicture}[scale=.4]
    \foreach \x in {0,...,1}
    \draw[xshift=\x cm,thick] (4*\x cm,0) circle (.5cm);
    \draw[thick] (0.5 cm,0) -- +(4 cm,0);
    \draw (2.5,-0.25) node[anchor=north]  {\small{$(-a_{ij},-a_{ji})$}};
  \end{tikzpicture}
\end{center}
\item Multiplicities of the type $(2,1)$, $(1,2)$, $(3,1)$ and $(1,3)$ may be replaced with double or triple edges and arrows pointing in the direction of the smaller of the two numbers $a_{ij}$, $a_{ji}$.
\end{enumerate}

If the Dynkin diagram of $A$ is the Dynkin diagram of a finite-dimensional semisimple Lie algebra (resp. an affine Lie algebra), we say it is of \emph{finite} (resp. \emph{affine}) type.

We say that $A$ is \emph{indecomposable} if the Dynkin diagram of $A$ is connected.  Equivalently, $A$ is \emph{decomposable} if there exists a permutation of the rows and columns of $A$ such that $A$ can be written as a block matrix
$$\begin{bmatrix}
A_1 & &  & 0 \\
0 & A_2 &  & 0 \\
& & \ddots \\
0 & &  & A_k
\end{bmatrix}$$
for some $k \geq 2$.

We say that $A$ is \emph{symmetrizable} if there exists a diagonal $n \times n$ matrix $D$ such that $B = DA$ is a symmetric matrix.  All of the BKM matrices we consider in this work will be symmetric.

The Generalized Kac-Moody algebra $g(A)$ associated to $A$, or Borcherds-Kac-Moody (BKM) algebra, is the Lie algebra generated by elements $h_i$, for $1 \leq i \leq 2n - \rk A = l$ , and $e_i, f_i$, for $1 \leq i \leq n$ with the following properties :
\begin{enumerate}
\item $h = \bigoplus\limits_{i=1}^l \C h_i$ is an abelian Lie algebra.
\item $[h_i,e_j] = a_{ij} e_j$ and $[h_i,f_j] = - a_{ij} f_j$.
\item $[e_i, f_j] = \delta_{ij} h_i$
\item If $a_{ii} = 2$, $(\ad e_i)^{1 - a_{ij}} e_j = 0$ and $(\ad f_i)^{1 - a_{ij}} f_j = 0$.
\item If $a_{ii},a_{jj} \leq 0$ and $a_{ij} = 0$ then $[e_i,e_j] = [f_i,f_j] = 0$.
\end{enumerate}

If $A$ is a Cartan matrix of finite (respectively affine) type, then $g(A)$ is a finite dimensional simple Lie algebra (respectively an affine Lie algebra).  

Let $\Pi = \{ \alpha_i | i = 1,...,n \}$ and $Q = \spn_\Z \Pi$.  We write
$$\Pi^{re} = \{ \alpha_i \in \Pi | a_{ii} = 2 \}$$
and
$$\Pi^{im} = \{ \alpha_i \in \Pi | a_{ii} \leq 0 \}.$$
Elements of the former are known as \emph{real simple roots} while elements of the latter are known as \emph{imaginary simple roots}.

If $A$ is symmetrizable, then we may define a bilinear form on $Q$ by
$$(\alpha_i,\alpha_j) = b_{ij}$$
for $B = (b_{ij})_{1 \leq i,j \leq n}$.

The space $Q$ has an important group action.  For $i \in I^{\re}$ and $\lambda \in Q$, write
$$r_i(\lambda) = \lambda - (\alpha_i,\lambda) \alpha_i.$$
It is easy to check that $r_i(\lambda)^2 = \id$.  We form the group
$$W = \langle r_i | i \in I^{re} \rangle$$
and we refer to this group as the \emph{Weyl group}.  The Weyl group also induces isomorphisms in $g(A)$.

The Lie algebra $g(A)$ has a very special isomorphism which we refer to as the \emph{Cartan involution}.  It is defined by
$$\theta(e_i) = -f_i, \theta(f_i) = -e_i, \theta(h_i) = -h_i$$
and extended to all of $g(A)$ in the obvious way.

The Lie algebra $g(A)$ also has a decomposition
$$g(A) = \bigoplus\limits_{\alpha \in Q} g_\alpha$$
where $g_{\alpha} = \{x \in g(A) | [h,x] = \alpha(h)x \text { for all } h \in h \}$.

Let $\Delta = \{ \alpha \in Q | g_{\alpha} \neq 0 \} \setminus \{ 0 \}$.  Then it turns out that
$$\Delta = \Delta^{+} \cup \Delta^{-}$$
where
$$\Delta^{+} = \left\lbrace \sum\limits_{i=1}^n a_i \alpha_i | a_i \geq 0 \right\rbrace \setminus \{ 0 \}$$
and $\Delta^{-} = - \Delta^{+}$.  Then we may rewrite our decomposition as
$$g(A) = \bigoplus\limits_{\alpha \in \Delta^{+}} g_{-\alpha} \oplus g_0 \oplus \bigoplus\limits_{\alpha \in \Delta^{+}} g_{\alpha}.$$

As a result, we also define the Lie subalgebras
$$n^{-} = \bigoplus\limits_{\alpha \in \Delta^{+}} g_{-\alpha}$$
and
$$n^{+} = \bigoplus\limits_{\alpha \in \Delta^+} g_{\alpha}.$$

Let $g$ be a semi-simple Lie algebra which is a direct sum of Lie algebras of type $A_n, D_n$ or $E_n$.  Let $\Delta_{fin}$ be the root system of $g$.  We will consider the affine Lie algebra 
$$\widehat{g} = g \otimes \C [t, t^{-1}] \oplus \C c \oplus \C d.$$
defined in such a way that
\begin{align*}
[x \otimes t^m,y \otimes t^n] &= [x,y] \otimes t^{m+n} + (x,y) m \delta_{m,-n} c \\
[c, x \otimes t^m] &= [c,d] = 0 \\
[d,x \otimes t^m] &= m x \otimes t^{m}
\end{align*}

Define $\delta \in h^*$ by $\delta(d) = 1$.  Then clearly,
$$[d,x \otimes t^m] = m\delta(d) x \otimes t^m.$$
In particular, the set
$$\Delta = \{ n\delta + \alpha | \alpha \in \Delta_{fin} \cup \{0\}\}$$
describes the roots of $\widehat{g}$.  If $\gamma$ is a highest root for $g$, then $\Delta$ can be rewritten as
$$\Delta = \{n\alpha_0 + n\gamma + \alpha | \alpha \in \Delta_{fin} \cup \{0\}\}$$
where $\alpha_0 = \delta - \gamma$.
This description has the advantage that $n\delta + \alpha$ is a positive root whenever $n > 1$ and a negative root whenever $n < 1$.  Choose $e_{\gamma} \in g_{-\gamma}$ and $f_{\gamma} \in g_{\gamma}$ such that $e_{\gamma},f_{\gamma},[e_{\gamma}, f_{\gamma}]$ form an $sl_2$-triple.  Set
$$e_0 = t \otimes e_{\gamma}$$
$$f_0 = t^{-1} \otimes f_{\gamma}$$
Then we can compute that $h_0 := [e_0,f_0] = \frac{2}{(\gamma,\gamma)}c + [e_\gamma,f_\gamma]$.  Write $\alpha_1,...,\alpha_n$ and $h_i,...,h_n$ for the simple roots and coroots of $g$.  Then the sets
$$\Pi = \{ \delta - \gamma, \alpha_1,...,\alpha_n\}$$
and
$$\Pi^{\vee} = \{ \frac{2}{(\gamma,\gamma)}c + [e_{\gamma},f_{\gamma}], h_1,...,h_n\}$$
define an $n+1 \times n+1$ matrix $A = (\alpha_j(h_i))_{i,j=0}^n$.  In particular,

\begin{theorem}
The Lie algebra $\widehat{g}$ is isomorphic to the Kac-Moody algebra defined by $g(A)$.
\end{theorem}

Let $\lambda \in h^*$.  A highest weight module $V_{\lambda}$ for $g(A)$ is a $g(A)$-module if and only there exists $v \in V_{\lambda}$ such that the following conditions are satisfied :
\begin{enumerate}
\item $U(n^+) v = 0$,
\item $h \cdot v = \lambda(h)v$, for all $h \in h$.
\item $U(n^-)v = V$.
\end{enumerate}

We can study the highest weight modules for a Kac-Moody algebra in a similar way to semisimple Lie algebras.  In particular, we have an analogous character formula.

\begin{theorem}\emph{(Weyl-Kac character formula)}
Suppose $A$ is a generalized Cartan matrix.  Let $\rho \in h^*$ such that $\rho(h_i) = 1$ for all $i = 1,...,n$ and $V(\lambda)$ an irreducible highest weight representation for $g(A)$ of dominant integral highest weight $\lambda$.  Then
$$\ch V(\lambda) = \frac{ \sum\limits_{w \in W} (-1)^{l(w)} e^{w(\lambda + \rho)} }{e^\rho \prod\limits_{\alpha \in \Delta^+} (1 - e^{-\alpha})^{\dim g_\alpha}}$$
\end{theorem}

\begin{corollary}\emph{(Weyl-Kac denominator formula)}
If $\lambda = 0$, the character formula simplifies to
$$e^{\rho} \prod\limits_{\alpha \in \Delta^+}(1 - e^{\alpha})^{\dim g_\alpha} = \sum\limits_{w \in W} (-1)^{l(w)} e^{w(\rho)} $$
\end{corollary}

There is also a denominator formula for BKM algebras :

\begin{theorem}\emph{(Weyl-Kac-Borcherds denominator formula)}
Suppose $A$ is a BKM matrix.   Set
$$S_0:= \sum\limits_{\substack{k \in \Z_{\geq 0} \\ i_1,...,i_k \in I^{im} }} (-1)^{k} e^{\rho - (\alpha_{i_1} + ... + \alpha_{i_k})}$$
where the sum is only over $i_1,...,i_k \in I^{im}$ satisfying $\alpha_1,...,\alpha_k$ pairwise orthogonal and distinct.  Then
$$e^{\rho} \prod\limits_{\alpha \in \Delta^{+}} (1 - e^{\alpha})^{\dim g_{\alpha}} = \sum\limits_{w \in W} \epsilon(w) w(S_0)$$
\end{theorem}

Next, we apply the denominator formula to untwisted affine Lie algebras.  In this case, we have
$$\Delta^+ =  \{n\delta + \alpha | \alpha \in \Delta^+_{fin}, n \geq 0\} \cup\{n\delta - \alpha | \alpha \in \Delta^+_{fin}, n \geq 1\} \cup \{ n \delta | n \geq 1 \}.$$ We also know that $\dim g_{\alpha} = 1$ if $(\alpha,\alpha) = 2$ and $d$ if $(\alpha,\alpha) = 0$.  Therefore, the left side of the denominator formula for an untwisted affine Lie algebra can be rewritten as
$$e^\rho \left( \prod\limits_{m \in \Z_{\geq 1}}(1 - e^{m \delta})^d \right) \left( \prod\limits_{n \in \Z_{\geq 1}} \prod\limits_{\alpha \in \Delta^+_{fin}} (1 - e^{n \delta - \alpha})\right) \left( \prod\limits_{n \in \Z_{\geq 1}} \prod\limits_{\alpha \in \Delta^+_{fin}} (1 - e^{n \delta + \alpha})\right).$$
We set $q = e^{\delta} = q^{2 \pi i \tau}$ and $p = e^{(\alpha,z)}$ and specialize the denominator formula.  After simplification, we obtain
\begin{equation}
e^\rho \prod\limits_{m \geq 1}(1 - q^m)^d \prod\limits_{\alpha \in \Delta_{fin}^+} \prod\limits_{n \geq 1} (1 - q^{n-1} p) (1 - q^n p^{-1}).
\end{equation}

Generalized Kac-Moody algebras also have an alternate description which is satisfied by some of the Lie algebras which appear in this document (see \cite{borcherds1a}).
\begin{theorem} \label{generalizedkacmoodycharacterization}
A Lie algebra $G$ is the quotient of a generalized Kac-Moody algebra by a subspace of its center if and only if it satisfies the following three conditions :
\begin{enumerate}
\item $G$ can be $\Z-$graded as $G = \bigoplus\limits_{i \in \Z} G_i$ where each $G_i$ is a sum of finite-dimensional $G_0$-modules.
\item $G_0 \subseteq [G,G]$.
\item $G$ has an involution $w : G_i \to G_{-i}$ which acts as $-1$ on $G_0$.
\item $G$ has an invariant bilinear form $(.,.)$, invariant under $w$ such that $G_i$ is orthogonal to $G_j$ whenever $i \neq -j$, and such that $-(g,w(g)) > 0$ if $g \neq 0 \in G_i$ such that $i \neq 0$.
\end{enumerate}

\end{theorem} 

\section{Vertex Operator Algebras}

Vertex Operator algebras as appear here were developed in detail in \cite{voa} based on earlier works referenced within, and were generalized into the more abstract notion of vertex algebra by Borcherds in \cite{borcherds0}.  We will rely on these notions extensively in the modern form of \cite{lepowsky}.

\subsection{Definitions}
In this section, we define vertex operator algebras as described in \cite{lepowsky} and describe some of their basic properties.  First, define the $\delta$-function
$$\delta(z) = \sum\limits_{n \in \Z} z^n.$$
Its most important property is that for any $f(z) \in \C[z,z^{-1}]$, we have
$$f(z)\delta(z) = f(1)\delta(z).$$

\begin{definition}
A \emph{vertex algebra} is a vector space $V$ along with a map
\begin{align*}
Y(.,z)  : V &\to (\ed V)[[z,z^{-1}]] \\
v &\to Y(v,z) = \sum\limits_{n \in \Z} v_n z^{-n-1}
\end{align*}
and a distinguished element $1 \in V$ (known as \emph{vacuum vector}) such that
$$u_n v = 0$$ for $n$ sufficiently large,
$$Y(1,z) = \id,$$
$$Y(v,z)1 \in V[[z]] \text{ and } \lim\limits_{z \to 0} Y(v,z)1 = v.$$  Furthermore we also have the \emph{Jacobi identity}
\begin{align*}
& z_0^{-1} \delta\left( \frac{z_1 - z_2}{z_0} \right) Y(u,z_1)Y(v,z_2) - z_0^{-1} \delta \left(\frac{z_2 - z_1}{-z_0} \right) Y(v,z_2)Y(u,z_1)\\
&= z_2^{-1} \delta \left( \frac{z_1 - z_0}{z_2} \right) Y(Y(u,z_0)v,z_2)
\end{align*}
where all binomial expansions of the form $(x_1 + x_2)^k$ for $k \in \Z$ are understood to be in nonnegative powers of $x_2$.
\end{definition}

\begin{definition}
A \emph{vertex operator algebra} (VOA for short) is a vertex algebra $(V,Y,1)$ along with a $\Z$-grading $V = \coprod\limits_{n \in \Z} V_n$ and a distinguished vector $\omega \in V_2$ such that $1 \in V_0$ and
$$Y(\omega,z) = \sum\limits_{n \in \Z} L_n z^{-n-2},$$
where the operators $L_n$ satisfy the Virasoro commutation relations
$$[L_m,L_n] = (m-n) L_{m+n} + \frac{m^3-m}{12} \delta_{m+n,0} c_V$$ for some $c_V \in \Q$ which we will usually denote only by $c$.  Furthermore,
$$L_0 v = nv \text{ if } v \in V_n$$
and
$$\frac{d}{dz} Y(v,z) = Y(L_{-1} v,z).$$
\end{definition}
We are also interested in an analogous definition in the superalgebra case which can be found in \cite{voas} :
\begin{definition}
A vertex operator superalgebra (SVOA for short) is a $\frac{1}{2} \Z$-graded vector space $V$ with analogous axioms to that of a vertex operator algebra with the exception that the Jacobi identity is replaced by
\begin{align*}
& z_0^{-1} \delta\left( \frac{z_1 - z_2}{z_0} \right) Y(u,z_1)Y(v,z_2) - (-1)^{\overline{u} \overline{v}}z_0^{-1} \delta \left(\frac{z_2 - z_1}{-z_0} \right) Y(v,z_2)Y(u,z_1)\\
&= z_2^{-1} \delta \left( \frac{z_1 - z_0}{z_2} \right) Y(Y(u,z_0)v,z_2)
\end{align*}
where $\overline{u}, \overline{v}$ are the degrees (odd or even) of $u,v$.  More precisely, $V = V_{\overline{0}} \oplus V_{\overline{1}}$ with
$$V_{\overline{0}} = \bigoplus\limits_{n \in \Z} V_n \text{ and }V_{\overline{1}} = \bigoplus\limits_{n \in \frac{1}{2} + \Z} V_n.$$
\end{definition}
Secondly, there is often a third distinguished element $\tau \in V_{3/2}$ which gives rise to the definition of $N=1$ SVOA.
\begin{definition}
An $N=1$ SVOA is a SVOA along with a distinguished element $\tau \in V_{3/2}$, known as Neveu-Schwarz element, such that
$$Y(\tau,z) = \sum\limits_{n \in \Z} \tau(n) z^{-n-1} = \sum\limits_{n \in \Z} G_{n+ \frac{1}{2}} z^{-n-2}.$$
Furthermore,
$$\omega = \frac{1}{2}G_{-\frac{1}{2}} \tau$$ and we have the supercommutation relations
$$[G_{m},L_n] = (m - \frac{n}{2}) G_{m+n},$$
$$\{ G_m, G_n \} = 2 L_{m+n} + \frac{1}{3}(m^2 - \frac{1}{4}) \delta_{m+n} c$$
for the same number $c$.  The Lie superalgebra generated by these elements and relations is known as the Neveu-Schwarz \emph{Super Virasoro Algebra.}
\end{definition}

Finally, there is a third notion of interest to us.

\begin{definition}
An $N = 2$ SVOA is a SVOA along with distinguished elements $\tau^+, \tau^-, j$ such that
$$Y(\tau^{\pm},z) = \sum\limits_{n \in \Z} \tau^{\pm} (n) z^{-n-1} = \sum\limits_{n \in \Z} G^{\pm}_{n+\frac{1}{2}} z^{-n-2},$$
$$Y(j,z) = \sum\limits_{n \in \Z} j(n) z^{-n-1} = \sum\limits_{n \in \Z} J_n z^{-n-2}$$
where $\tau = \tau^+ + \tau^-$ is a Neveu-Schwarz element and we have the additional commutation relations :
\begin{align*}
& [L_m,J_n] = -nJ_{m+n} \\
& \{ G_r^+, G_s^- \} = L_{r+s} + \frac{1}{2} (r-s) J_{r+s} + \frac{c}{6} (r^2 - \frac{1}{4}) \delta_{r+s}  \\
& \{ G_r^+, G_s^+ \} = \{ G_r^-, G_s^- \} = 0 \\
&[L_m, G_r^{\pm}] = (\frac{m}{2} - r) G_{r+m}^{\pm} \\
& [J_m,G_r^{\pm}] = \pm G_{m+r}^{\pm}
\end{align*}
The Lie superalgebra generated by these elements with these relations is known as the \emph{N = 2 superconformal algebra.}
\end{definition}

VOAs and SVOAs share many properties.  We do not list them all, and instead only list the ones that we will need.  To begin, we have :
$$Y(u,z) = e^{z L_{-1}} u \text{ (Creation property)},$$
$$Y(u,z)v = (-1)^{\overline{u} \overline{v}} e^{z L_{-1}} Y(v,-z)u \text{ (Skew-symmetry)}.$$

From the Jacobi identity, taking residues at $z_0$ on both sides of the equation, we obtain
$$\{ Y(u,z_1),Y(v,z_2) \} = \res_{z_0} z_2^{-1} \delta \left( \frac{z_1 - z_0}{z_2} \right) Y(Y(u,z_0)v,z_2).$$
Observe that because $(z_1-z_0)^k$ expands in nonnegative powers of $z_0$, we do not need to consider the regular part of $Y(\omega,z_0)$ in any computation involving the right hand side of this formula.  We rewrite this bracket as
\begin{equation} \label{commutatorformula}
\{Y(u,z_1),Y(v,z_2)\} = \res_{z_0} z_2^{-1} \left( e^{-z_0 \frac{\partial}{\partial z_1}} \delta(z_1/z_2) \right) Y(Y(u,z_0)v,z_2),
\end{equation}
and we refer to this expression as the \emph{commutator formula}.  An important corollary of the commutator formula is the following :

\begin{corollary} \label{commutatorcorollary}
Suppose $u,v \in V$.  Then
$$[u_0,Y(v,z)] = Y(u_0 v,z)$$
\end{corollary}
\begin{proof}
Replace $z_2$ with $z$ in the commutator formula.  We must compute
$$\res_{z_0} \res_{z_1} z^{-1}\left( e^{-z_0 \frac{\partial}{\partial z_1}} \delta(z_1/z) \right) Y(Y(u,z_0)v,z).$$
It is enough to see that the only way $z_1^{-1}$ can appear in this expression is via the constant term of the exponential $e^{-z_0 \frac{\partial}{\partial z_1}}$.  In other words, this expression simplifies to
$$\res_{z_0} Y(Y(u,z_0)v,z) = Y(u_0 v,z)$$
as desired.
\end{proof}

As a result of this corollary, we see that $N = 1$ SVOA's have the extra property that :
$$\{ G_{- \frac{1}{2}}, Y(u,z) \} = Y(G_{- \frac{1}{2}}u,z).$$

We may also use the commutator formula to obtain an additional important property.

\begin{proposition}
Suppose $u \in V_{(m)}$ and $v \in V_{(n)}$.  Then $u_k v \in V_{(m+n-k-1)}$.
\end{proposition}
\begin{proof}
See paragraph before Remark 3.1.25 in \cite{lepowsky}.
\end{proof}

\subsection{Tensor products, modules and homomorphisms}

In general it is not so simple to construct $N=1$ SVOAs directly.  In fact, one may argue it is not so simple to construct VOAs themselves.  We will need to spend a great deal of time and ink to define the VOAs and SVOAs which interest us.  As a result, it will be useful to know how to create new examples from old ones.  Therefore, we describe here some elementary knowledge of tensor product of SVOAs.  Suppose $V_1,...,V_k$ are SVOAs.  Then consider
$$V = V_1 \otimes ... \otimes V_k.$$
We define $Y(v,z)$ for $v  = v_1 \otimes ... \otimes v_k \in V$ by
$$Y(v_1 \otimes ... \otimes v_k,z) = Y(v_1,z) \otimes ... Y(v_k,z).$$
The vacuum vector is then
$$1 = 1_1 \otimes ... \otimes 1_k.$$
The $\Z$-grading on $V$ is defined by
$$(V)_n = \sum\limits_{n_1 + ... n_k = n} (V_1)_{n_1} \otimes ... \otimes ({V_k})_{n_k}$$
This grading also separates $V$ into odd and even parts in the obvious way. If $\omega_1,...,\omega_k$ are conformal vectors for $V_1,...,V_k$, we define
$$\omega = \omega_1 \otimes 1 \otimes ... \otimes 1 + 1 \otimes \omega_2 \otimes ... \otimes 1 + ... + 1 \otimes 1 \otimes ... \otimes \omega_k.$$
With respect to this new $\omega$, we have
$$L_n = L_n \otimes ... \otimes 1 + ... + 1 \otimes ... \otimes L_n$$
for $n \in \Z$.  It is easy to see that $L_0$ is compatible with the new $\Z$-grading.  We have now established all the ingredients necessary to state the following proposition :

\begin{proposition}
The vector space $V$ with vacuum vector $1$ and conformal vector $\omega$ described above is an SVOA.
\end{proposition}

\begin{proof}
See Proposition 3.12.5 and Proposition 3.12.8 in \cite{lepowsky} and adapt it to super vertex operator algebras.
\end{proof}

We will need to consider as well the notion of modules and homomorphisms of SVOAs.  We define briefly these notions here for SVOAs.  Observe that the definitions for VOAs can be obtained by restricting to the even part of an SVOA.

\begin{definition}
Suppose $V_1,V_2$ are SVOAs.  We say $f : V_1 \to V_2$ is a homomorphism of SVOAs if $f$ is a linear map such that
$$f(Y(u,z)v) = Y(f(u,z))f(v)$$
and
$$f(1) = 1.$$
\end{definition}

As usual, given a definition of homomorphism, the definition of isomorphism is obvious.

\begin{definition}
Let $V$ be an SVOA and suppose $W$ is a $\frac{1}{2} \Z$-graded vector space.  We say $W$ equipped with a map
$$Y_W(.,z) : V \to (\ed W)[[z,z^{-1}]]$$
$$v \to Y_W(v,z) = \sum\limits_{n \in \Z} v_n z^{-n-1}$$
is a $V$-module if 
$$u_n w = 0$$
for $n$ sufficiently large,
$$Y_W(1,z) = \id$$ and
\begin{align*}
& z_0^{-1} \delta\left( \frac{z_1 - z_2}{z_0} \right) Y_W(u,z_1)Y_W(v,z_2) - \\ & (-1)^{\overline{u} \overline{v}}z_0^{-1} \delta \left(\frac{z_2 - z_1}{-z_0} \right) Y_W(v,z_2)Y_W(u,z_1)\\
&= z_2^{-1} \delta \left( \frac{z_1 - z_0}{z_2} \right) Y_W(Y(u,z_0)v,z_2)
\end{align*}
for $u,v \in V$.  Observe that both $Y_W$ and $Y$ appear on the right hand side.  One writes
$$W = \coprod\limits_{n \in \frac{1}{2}\Z} W_{(n)}.$$
\end{definition}

\begin{remark}
Any SVOA $V$ is also itself a $V$-module with the obvious action.
\end{remark}

\begin{remark}
Given a $V$-module $W$ and an isomorphism $f$ of $V$, we may twist the action of $V$ on $W$ by this isomorphism, i.e.
$$(Y_W)_{f}(v,z)w = Y_W(f(v),z)w$$
for $v \in V$, $w \in W$, and we obtain a new $V$-module structure on $W$.  We will use this notion.
\end{remark}

\begin{definition}
Suppose $W_1, W_2$ are $V$-modules.  We say $f : W_1 \to W_2$ is a homomorphism of $V$-modules if $f$ is linear and 
$$f(Y_{W_1}(u,z)w) = Y_{W_2}(u,z)f(w).$$
for $u \in V, w \in W_2$.
\end{definition}

\subsection{Dual module and adjoint vertex operators} \label{dualmodule}
This section is heavily based on \cite{frenkelhuanglepowsky}, but modified for SVOAs.  The proofs are essentially the same, but because the formula is slightly different computations where applicable should be done again.

Suppose $W$ is a $V$-module, where $V$ is an SVOA.  Then we have
$$W = \coprod\limits_{n \in \frac{1}{2}\Z} W_{(n)}.$$
In this work, we will assume that each $W_{(n)}$ decomposes further as a direct sum of finite-dimensional spaces, i.e.
$$W_{(n)} = \coprod\limits_{i \in I_n} W_{(n),i}$$
  We define the graded dual of $W$ as
$$W' = \coprod\limits_{n \in \frac{1}{2}\Z} \coprod\limits_{i \in I_n} W^*_{(n),i}.$$
Because the dimensions of each $W_{(n),i}$ is finite, we have
$$\dim W^*_{(n),i} = \dim W_{(n),i} < \infty.$$
The main objective of this section is to define a $V$-module structure on $W'$.  We define \emph{adjoint vertex operators} $Y'(v,z)$ for $v \in V_{(n + \frac{1}{2})} \oplus V_{(n)}$ by
$$\langle Y'(v,z)w',w \rangle = (-1)^n \langle w', Y(e^{z L_1} z^{-2 L_0}v,z^{-1})w \rangle. $$
This formula differs from \cite{frenkelhuanglepowsky}, but can be found in \cite{duncan}.  Write
$$\overline{L_0}(v) = n \text{ for } v \in V_{(n + \frac{1}{2})} \oplus V_{(n)}.$$

The main theorem is the following :

\begin{theorem} (\emph{\cite{frenkelhuanglepowsky}, Theorem 5.2.1}, \cite{duncan}, Proposition 2.5)
The module $W'$ with operators $Y'$ is a $V$-module.
\end{theorem}
In \cite{frenkelhuanglepowsky}, much effort is spent on proving that
$$\frac{d}{dz} Y'(v,z) = Y'(L_{-1}v,z),$$
however this follows from the Jacobi identity and the other axioms (see Proposition 4.1.3 in \cite{lepowsky}).  Therefore, it is enough to prove Jacobi identity.  In other words, we must prove
\begin{align*}
& \langle z_0^{-1} \delta\left( \frac{z_1 - z_2}{z_0} \right) Y'(v_1,z_1)Y'(v_2,z_2) w',w \rangle \\
- & (-1)^{\overline{v_1}\overline{v_2}} \langle z_0^{-1} \delta\left( \frac{z_2 - z_1}{-z_0} \right) Y'(v_2,z_2)Y'(v_1,z_1) w', w\rangle \\
&= \langle z_2^{-1} \delta\left( \frac{z_1 - z_0}{z_2} \right) Y'(Y(v_1,z_0)v_2,z_2) w', w \rangle.
\end{align*}
From there, one follows the computations on page 49 of \cite{frenkelhuanglepowsky}, inserting the term $(-1)^{\overline{v_1}\overline{v_2}}$ where appropriate and replacing $L_0$ with $\overline{L_0}$.  The end result is that one must show
\begin{align*}
& (-1)^{\overline{v_1}\overline{v_2}}z_1^{-1} \delta \left(\frac{z_2+z_0}{z_1} \right) Y(e^{z_2 L_1} (-1)^{\overline{L_0}} z_2^{-2L_0})Y(v_1,z_0)v_2,z_2^{-1})
\\ &= z_1^{-1} \delta \left(\frac{z_2+z_0}{z_1} \right) Y(Y(e^{z_1 L_1} (-1)^{\overline{L_0}} z_2^{-2L_0}) v_1, -z_0/z_1 z_2)  \cdot \\ & \cdot e^{z_2 L_1} (-1)^{\overline{L_0}} z_2^{-2L_0}) v_2,z_2^{-1}).
\end{align*}
or equivalently, we can show the conjugation formula
\begin{align*}
& e^{z L_1} (-1)^{\overline{L_0}} z^{2L_0} Y(v,z_0) (-1)^{\overline{L_0}} z_2^{-2L_0} e^{-z L_1} \\
&= (-1)^{\overline{v_1}\overline{v_2}} Y(e^{(z+z_0)L_1} (-1)^{\overline{L_0}}(z+z_0)^{-2L_0}v, -z_0/(z+z_0)z).
\end{align*}
We prove this conjugation formula using Lemma 5.2.3 in \cite{frenkelhuanglepowsky} and the following additional lemma :
\begin{lemma}
Suppose $v \in V_{(m)}$ and $w \in V_{(n)}$ are homogeneous elements.  Then the following conjugation formula holds :
$$(-1)^{\overline{L_0}} Y(v,z_0) (-1)^{\overline{L_0}} w = (-1)^{\overline{v}\overline{w}} Y((-1)^{\overline{L_0}}v,-z_0) w$$
\end{lemma}
\begin{proof}
Observe that $v_k w \in V_{n + m - k - 1}$ hence by a direct computation,
$$(-1)^{\overline{L_0}} v_k (-1)^{\overline{L_0}} w = (-1)^{\overline{v}\overline{w}} (-1)^{-k-1} ((-1)^{\overline{L_0}}v)_k w$$
where the term $(-1)^{\overline{v} \overline{w}}$ appears because
$$\overline{L_0}v + \overline{L_0}w + 1 = \overline{L_0}(v+w)$$
if and only if $m,n \in \Z + \frac{1}{2}$.  But then, the result of the Lemma follows from computing the coefficients of the formal power series in the assertion.
\end{proof}

\section{Semi-infinite cohomology} \label{sectioncohomology}
The results of this section were first developed in \cite{feigin},\cite{frenkelgarlandzuckerman} and expanded in \cite{lian} and \cite{lianzuckerman1}, \cite{lianzuckerman2}.  We do not need the theory of semi-infinite cohomology for any of our new results, however it offers a different perspective on the different Lie algebras that will be presented.

\subsection{Definitions}

In this section, we describe the classical case as it first appeared in \cite{frenkelgarlandzuckerman}, but using the definitions from \cite{lian}.  Let $\mf{g}$ be a $\frac{1}{2} \Z$-graded Lie superalgebra.  Then
$\mf{g} = \bigoplus\limits_{i \in \frac{1}{2} \Z} \mf{g}_i.$ and assume that each $\mf{g}_i$ is finite-dimensional.  We let $q_B, q_F \in \frac{1}{2} \Z$, but to be consistent with \cite{lianzuckerman2}, we will use $-q_B, -q_F$ instead of $q_B,q_F$ themselves in the following definition :  
$$\mf{n}_{+} = \bigoplus\limits_{i > -q_F} (\mf{g}_i)_{even} \oplus \bigoplus_{i > -q_B} (\mf{g}_i)_{odd} \text{ and } \mf{n}_{-} = \bigoplus\limits_{i \leq -q_F} (\mf{g}_i)_{even} \oplus \bigoplus_{i \leq -q_B} (\mf{g}_i)_{odd}.$$  
Choose a basis $\{ e_{i,j} \}_{i \in \frac{1}{2}\Z}$, $\{f_{i,j} \}_{i \in \frac{1}{2}\Z}$ of $\mf{g}_{even}$ and $\mf{g}_{odd}$, where for each index $i$, the index $j$ runs over a finite set.  We define a dual basis $\{ e'_{i,j} \}_{i \in \frac{1}{2}\Z}$, $\{ f'_{i,j} \}_{i \in \frac{1}{2}\Z}$ such that $\langle e'_{i,j} ,e_{k,l} \rangle = \langle f'_{i,j} f_{k,l} \rangle = \delta_{i,k} \delta_{j,l}$.  Define 
$$\mf{g'} = \bigoplus\limits_{i,j \in \frac{1}{2}\Z} \C e'_{i,j} \oplus \C f'_{i,j} \text{ and } \mf{g'}_i = \bigoplus\limits_{j \in \frac{1}{2}\Z} \C e'_{i,j} \oplus \C f'_{i,j}.$$ Then we let 
$$\mf{n'}_{+} = \bigoplus\limits_{i > -q_F} (\mf{g'}_i)_{even} \oplus \bigoplus_{i > -q_B} (\mf{g'}_i)_{odd} \text { and } \mf{n'}_{-} = \bigoplus\limits_{i \leq -q_F} (\mf{g'}_i)_{even} \oplus \bigoplus_{i \leq -q_B} (\mf{g'}_i)_{odd}.$$ 
We consider
$$\wedge^{\frac{\infty}{2}} \mf{g} = \bigwedge (n'_+ \oplus n_-)_{even},$$
$$\vee^{\frac{\infty}{2}} \mf{g} = \bigvee(n'_+ \oplus n_-)_{odd}.$$
These spaces are essentially polynomial algebras of commuting and anticommuting variables in the basis elements.  Explicitly, $\wedge^{\frac{\infty}{2}} \mf{g}$ is the linear span of the exterior products
$$\omega_0 = e'_{i_1} \wedge ... \wedge e'_{i_n} \wedge e_{j_1} \wedge ... \wedge e_{j_m},$$
and $\vee^{\frac{\infty}{2}} \mf{g}$ is the linear span of the symmetric products
$$\omega_1 = f'_{i_1} \vee ... \vee f'_{i_n} \vee f_{j_1} \vee ... \vee e_{j_m}.$$
Let
$$\Omega^{\frac{\infty}{2}} \mf{g} = \wedge^{\frac{\infty}{2}} \mf{g} \otimes \vee^{\frac{\infty}{2}} \mf{g}.$$
Write an element
$$\omega = x_1 \wedge ... \wedge x_n \otimes y_1 \vee ... \vee y_m$$
where $x_i \in \mf{g}_{even} \oplus \mf{g'}_{even}$ and $y_j \in \mf{g}_{odd} \oplus \mf{g'}_{odd}$.  Then we define $\langle .,. \rangle$ on $\mf{g} \oplus \mf{g'}$ by
$\langle \mf{g}, \mf{g} \rangle = \langle \mf{g'}, \mf{g'} \rangle = 0$ and $\langle a, x \rangle = \langle x, a \rangle$ for $a \in \mf{g}$ and $x \in \mf{g'}$.

We now define on the element $\omega$, for $u \in (\mf{n'}_+)_{even}, x \in (\mf{n'}_-)_{even}, y \in (\mf{n'}_+)_{odd}, z \in (\mf{n'}_-)_{odd}, a \in (\mf{n}_+)_{even} , b \in (\mf{n}_{-})_{even}, c \in (\mf{n}_+)_{odd} , d \in (\mf{n}_{-})_{odd}$,
\begin{align*}
\epsilon(u) \omega &= u \wedge \omega \\
\epsilon(x) \omega &= \sum\limits_{k=1}^{m} (-1)^{k-1} \langle x, x_k \rangle x_1 \wedge ... \hat{x_k} ... \wedge x_n \otimes y_1 \vee ... \vee y_m \\
\epsilon(y) \omega &= \omega \vee y \\
\epsilon(z) \omega &= \sum\limits_{k=1}^m \langle z,y_k \rangle x_1 \wedge ... \wedge x_n \otimes y_1 \vee ... \hat{y}_k ... \vee y_m
\end{align*}
\begin{align*}
\iota(a) \omega &= \sum\limits_{k=1}^n (-1)^{k-1} \langle a, x_k \rangle x_1 \wedge ... \hat{x_k} ... \wedge x_n \otimes y_1 \vee ... \vee y_m \\
\iota(b) \omega &= b \wedge \omega \\
\iota(c) \omega &= -\sum\limits_{k=1}^m \langle c,y_k \rangle x_1 \wedge ... \wedge x_n \otimes y_1 \vee ... \hat{y}_k ... \wedge y_m \\
\iota(d) \omega &= \omega \vee d.
\end{align*}

There is a normal ordering
   $$:\epsilon(e'_i) \iota(e_j) : = \left\{
     \begin{array}{lr}
       - \iota(e_j) \epsilon(e'_i) & , i > -q_F \\
       \epsilon(e'_i) \iota(e_j) & , i \leq -q_F
     \end{array}
   \right.$$
   $$:\epsilon(f'_i) \iota(f_j) : = \left\{
     \begin{array}{lr}
       \iota(f_j) \epsilon(f'_i) & , i > -q_B \\
       \epsilon(f'_i) \iota(f_j) & , i \leq -q_B
     \end{array}
   \right.$$
Then we may consider for $x \in \mf{g}$,
$$\tilde{\rho}(x) = \sum\limits_{i \in \Z} : \epsilon(e'_i) \iota([e_i,x]): - \sum\limits_{i \in \Z} : \epsilon(f'_i) \iota([f_i,x]) :$$
It turns out that $\tilde{\rho}(x)$ is not a representation of $\mf{g}$, but rather lifts to a representation of a central extension of $\mf{g}$.  In our applications, $\mf{g}$ will have no non-trivial central extensions and if we write
$$\gamma(x,y) = [\tilde{\rho}(x),\tilde{\rho}(y)] - \tilde{\rho}[x,y]$$
then there will be an element $\beta \in \mf{g}'$ such that
$$\gamma(x,y) = \langle \beta, [x,y] \rangle,$$
a trivial cocycle and then we may replace $\tilde{\rho}(x)$ with 
$$\rho(x) = \tilde{\rho}(x) + \langle \beta,x \rangle,$$
which now defines a representation of $\mf{g}$.

Next, we consider the \emph{ghost number operator},
$$U = \sum\limits_{i \in \Z} : \epsilon(e'_i) \iota(e_i) : - \sum\limits_{i \in \Z} : \epsilon(f'_i) \iota(f_i) :$$
an operator which is diagonalizable on $\Omega^{\frac{\infty}{2}} \mf{g}$ with integer eigenvalues, which are commonly known as \emph{ghost numbers}.  In fact, we have the commutation relations
$$[U,\epsilon(x)] = \epsilon(x),$$
$$[U,\iota(y)] = - \iota(y).$$
Therefore, we may decompose $\Omega^{\frac{\infty}{2}} \mf{g}$ as a direct sum of its eigenspaces for the operator $U$, which we denote by $\Omega^{\frac{\infty}{2}+n} \mf{g}$.

Now let $(V,\pi)$ be a $\Z$-graded $\mf{g}$-module such that $\pi(\mf{g}_m) \cdot V_n \subseteq V_{m+n}$ and $\pi(\mf{n}_+) \cdot v$ is finite-dimensional for any vector $v \in V$.  Then we let
$$C^{\frac{\infty}{2}+*}(\mf{g},V) = V \otimes \Omega^{\frac{\infty}{2}+*} \mf{g}$$

We define an operator $d : C^{\frac{\infty}{2}+n}(\mf{g},V) \to C^{\frac{\infty}{2}+n+1}(\mf{g},V)$ by
\begin{align*}
d = &\epsilon(\beta) + \sum\limits_{i \in \Z} \pi(e_i) \epsilon(e'_i) + \sum\limits_{i \in \Z} \pi(f_i) \epsilon(f'_i) - \frac{1}{2} \sum\limits_{i,j \in \Z} : \iota([e_i,e_j]) \epsilon(e'_i) \epsilon(e'_j): \\
& + \sum\limits_{i,j \in \Z} \iota([e_i,f_j]) \epsilon(e'_i) \epsilon(f'_j) - \frac{1}{2} \sum\limits_{i,j \in \Z} : \iota(\{ f_i,f_j \}) \epsilon(f'_i) \epsilon(f'_j):
\end{align*}
with the following very important property :
\begin{lemma} (\cite{lian}, Lemma 2.1)
$d^2 = 0$.
\end{lemma}
Therefore, we have a chain complex $(C^{\frac{\infty}{2}+*}(\mf{g},V), d)$ and we denote its cohomology groups by
$$H^{\frac{\infty}{2}+n}(\mf{g},V).$$
Next let $\theta(x) = \pi(x) + \rho(x)$.  It is clear that $\theta(x)$ is a representation of $\mf{g}$.  Also, let $\mf{h} \subseteq \mf{g}$ be a subalgebra of $\mf{g}_0$.  Then we define the relative subcomplex
$$C^{\frac{\infty}{2}+*}(\mf{g},\mf{h},V) = \{ v \in C^{\frac{\infty}{2}+*}(\mf{g},V), \iota(x)v = \theta(x)v = 0, x \in \mf{h} \}.$$
It can be checked that $d$ restricts to a map $$d : C^{\frac{\infty}{2}+n}(\mf{g},\mf{h},V) \to C^{\frac{\infty}{2}+n+1}(\mf{g},\mf{h},V),$$ and therefore that $C^{\frac{\infty}{2}+*}(\mf{g},\mf{h},V)$ is a subcomplex with relative cohomology groups
$$H^{\frac{\infty}{2}+*}(\mf{g},\mf{h},V).$$

\subsection{The Virasoro Lie algebra} \label{BRSTVir}

Let $\mf{g}$ be the Virasoro Lie algebra.  Then the odd part of $\mf{g}$ is zero and much of the definitions simplify.  We fix $q_F = 2$.

On the space $\Lambda^{\frac{\infty}{2}} \mf{g}$ we may then define $c(n) = \epsilon(L'_{-n})$ and $b(n) = \iota(L_n)$.  Observe that on the vacuum vector $1$, we have
$$b(n)1 = 0, n > -2$$
$$c(n)1 = 0, n \geq 2$$ 
More generally, we write the fields
$$c(z) = \sum\limits_{n \in \Z} c(n) z^{-n+1},\ b(z) = \sum\limits_{n \in \Z} b(n) z^{-n-2}.$$
Under this description, we then have on the vacuum, creation operators for non-positive powers of $z$ and annihilation operators for positive powers of $z$.
Then define
$$L^{\Lambda}(z)=\sum\limits_{n \in \Z} \rho(L_n) z^{-n-2}.$$
There are many results from \cite{frenkelgarlandzuckerman} adapted in this setting :
\begin{proposition}
On $\Omega^{\frac{\infty}{2}} \mf{g}$, $\rho(c) = -26$ and $\rho(L_m) = \sum\limits_{n \in \Z} (n-m) :c_{-n}b_{m+n}:$.  Furthermore on $(C^{\frac{\infty}{2}+*}(\mf{g},V), d)$, we have
$$L^{\Lambda}(z) = :2 \partial b(z) c(z) + b(z) \partial c(z):$$
\end{proposition}

Therefore, let $V$ be a VOA with central charge $26$ and conformal element $\omega^L$.  Then let $L^V(z) = Y(\omega^L,z) = \sum\limits_{n \in \Z} L^V_n z^{-n-2}$.  
The operator
$$L(z) = L^V(z) + L^{\Lambda}(z)$$
acts in an obvious way on $(C^{\frac{\infty}{2}+*}(\mf{g},V), d)$ with central charge $c=0$.  Finally, if we write
$$J(z) = (L^V(z) + \frac{1}{2} L^{\Lambda}(z))c(z).$$
Then $J_0$ recovers the operator $d$ (\cite{frenkelgarlandzuckerman}, Proposition 2.2)
We may also write $d$ as
\begin{equation}\label{BRSTformula}
d = \sum\limits_{i \in \Z} \pi(L_n) c_{-n} + \sum\limits_{m < n} (m-n) : b_{m+n} c_{-m} c_{-n}:
\end{equation}

\subsection{The super Virasoro algebra} \label{BRSTSVir}
This section is a review of \cite{lianzuckerman1}, which we begin here and finish later after we have discussed how to construct the module $V$.  Let $\mf{g}_0$ and $\mf{g}_{1/2}$ be the Ramond and Neveu-Schwarz super-Virasoro algebras, respectively.  Let $\kappa = 0, 1/2$, depending on which of the two we will be considering, i.e. the commutation relations
\begin{align*}
[L_m,L_n] &= (m-n)L_{m+n} + \frac{c}{12} (m^3-2 \kappa m) \delta_{m+n} \\
[L_m,G_{n+\kappa}] &= (\frac{1}{2}m - n - \kappa) G_{m+n+\kappa} \\
\{G_{m+\kappa},G_{n-\kappa} \} &= 2 L_{m+n} + \frac{c}{3} (m^2 + 2 \kappa m) \delta_{m+n}
\end{align*}
where $c$ is a central element.  We choose $q_F = 2$ again, preferring to be consistent with \cite{lianzuckerman2} rather than \cite{lianzuckerman1}, where they chose $q_F = 0$.  The choice of $q_B$ is more complicated.  For $\kappa = \frac{1}{2}$ we choose $q_B = \frac{1}{2}$ but for $\kappa = 0$ we will consider both $q_B = 0$ and $q_B = 1$.  We denote the corresponding space $\Omega^{\frac{\infty}{2}} \mf{g}$ by $\Omega_{q_B}^{\frac{\infty}{2}} \mf{g}$ to account for this additional dependence on $q_B$.

We define as before, $b(n) = \iota(L_n)$, $c(n) = \epsilon(L'_{-n})$ but now also $\beta(r) = \iota(G_r)$ and $\gamma(r) = \epsilon(G'_{-r})$.  In this case we observe that
$$\beta(r)1 = 0, r + \kappa > - q_B$$
$$\gamma(r)1 = 0, r + \kappa \geq q_B .$$
More generally, we obtain the fields $b(z),c(z)$ as before and the fields
$$\beta(z) = \sum\limits_{n \in \Z} \beta(n+\kappa) z^{-n - \kappa - \frac{3}{2}}, \gamma(z) = \sum\limits_{n \in \Z} \gamma(n + \kappa) z^{-n - \kappa + \frac{1}{2}}.$$
The representation $\rho$ is now very different, because of the extra terms.  Indeed,
\begin{align*}
\rho(L_m) = &\sum\limits_{n \in \Z} (n-m) : c(-n)b(m+n) : + \\ &\sum\limits_{n \in \Z} (\frac{1}{2}m - n - \kappa) : \gamma(-n-\kappa) \beta(m+n+\kappa): - \kappa \delta_m,
\end{align*}
$$$$
$$\rho(G_{m+\kappa}) = -2 \sum\limits_{n \in \Z} b(-n) \gamma(m+n-\kappa) : + \sum\limits_{n \in \Z} (\frac{1}{2} - \kappa - m) c(-n) \beta(m+n+\kappa),$$
and $\rho(c) = -15$.

Let $V_\kappa$ be a module for the super Virasoro algebra of central charge 15 for each of $\kappa = 0, \frac{1}{2}.$  Define
\begin{align*}
C_{\kappa}^{\frac{\infty}{2}+*}(\mf{g_\kappa},V_\kappa) &= V_{0} \otimes (\Omega_{0}^{\frac{\infty}{2}} \mf{g}_{0} \oplus \Omega_{1}^{\frac{\infty}{2}} \mf{g}_0) & \text{ for } \kappa &= 0  \\
&= V_\frac{1}{2} \otimes \Omega_{\frac{1}{2}}^{\frac{\infty}{2}} \mf{g}_0 & \text{ for } \kappa &= \frac{1}{2}
\end{align*}

It can be checked that $(C^*_{q_B},d)$ is a chain complex with $d : C^{r,s}_{q_B} \to C^{r+1,s}_{q_B}$.  We will return to our study of semi-infinite cohomology once we know more about these modules $V_0, V_{\frac{1}{2}}$.

\section{The N=0 Lie algebra of physical states} \label{sectionoghost}

\subsection{Vertex operator algebras on even lattices} \label{subsectionphysicalstate}
Let $L$ be an even Lorentzian lattice of rank $d \geq 2$, i.e. a lattice of signature $(d-1,1)$ with bilinear form $(.,.)$.  Write $h = L \otimes \C$, $\widehat{h} = h \otimes \C[t,t^{-1}]$ and $\widehat{h}^{-} = h \otimes t^{-1} \C[t^{-1}]$.

We extend $L$ by a central extension
$$1 \rightarrow \Z_2 \rightarrow \hat{L} \rightarrow L \rightarrow 1.$$
and with group operation in $\hat{L}$ given by
$$(\alpha,a) + (\beta,b) = (\alpha+\beta, \epsilon(\alpha,\beta)ab)$$
where $\alpha,\beta \in L$, $a,b \in \{ \pm 1 \}$ and $\epsilon : L \times L \rightarrow \{ \pm 1 \}$ is a cocycle with the following properties :
\begin{align*}
\epsilon(\alpha,\beta) \epsilon(\alpha+\beta,\gamma) & = \epsilon(\alpha,\beta+\gamma) \epsilon(\beta,\gamma) \\
\epsilon(\alpha,\beta) &= (-1)^{(\alpha,\beta)} \epsilon(\beta,\alpha) \\
\epsilon(\alpha,-\alpha) &= (-1)^{\frac{1}{2}(\alpha,\alpha)} \\
\epsilon(0,\alpha) &= \epsilon(\alpha,0) =  1
\end{align*}
The construction of such a cocycle is described in \cite{voa}.  We fix a section
$$e : L \to \hat{L},$$
$$\alpha \to e_{\alpha}$$
satisfying $e_0 = 1$.  Denote by $\kappa$ the image of a generator for $\Z_2$ in $\hat{L}$.

Next we consider the Fock space
$$V_L = S(\widehat{h}^{-}) \otimes \C\{L\}$$
where $\C\{L\}$ is the induced module
$$\C\{L\} = \C[\hat{L}]/(1 + e^{\kappa})\C[\hat{L}]$$
from the group algebra of $\hat{L}$.  Note that although $\C\{L\}$ and $\C[L]$ are not isomorphic as group algebras, they are still isomorphic linearly via the isomorphism
$$e^{\alpha} \to e^{e_{\alpha}},$$ which allows us to abuse notation and write an element of $\C\{L\}$ as $e^\alpha$ for $\alpha \in L$ (and not $\hat{L}$). 

In this sense, $\C\{L\}$ is the "twisted" group algebra of $L$ with product
$$e^{\alpha}e^{\beta} = \epsilon(\alpha,\beta)e^{\alpha+\beta}.$$
Then using the cocycle conditions, we have
\begin{align*}
e^{\alpha} \cdot e^{\beta} \cdot e^{\gamma} &= \epsilon(\alpha,\beta) e^{\alpha + \beta} \cdot e^{\gamma} \\
&= (-1)^{(\alpha,\beta)} \epsilon(\beta,\alpha) e^{\alpha + \beta} \cdot e^{\gamma} \\
&= (-1)^{(\alpha,\beta)} e^{\beta} \cdot e^{\alpha} \cdot e^{\gamma}
\end{align*}
or in other words,
$$e^{\alpha} e^{\beta} = (-1)^{(\alpha,\beta)} e^{\beta} e^{\alpha}$$
as operators on $V_L$.

We will write $\alpha(n) \in S(\widehat{h}^{-})$ as a shorthand for $\alpha \otimes t^n$, where $\alpha \in h$.  Then, for $n \in \Z$, we define operators on $V_L$ by
\begin{align*}
\alpha(n) \cdot v \otimes e^{\beta} = n \frac{\partial v}{\partial \alpha(-n)} \otimes e^{\beta} & \text{ for } n > 0 \\
\alpha(n) \cdot v \otimes e^{\beta} = \alpha(n) v \otimes e^{\beta} & \text{ for } n < 0 \\
\alpha(0) \cdot v \otimes e^{\beta} = (\alpha, \beta) v \otimes e^{\beta}
\end{align*}
where $\frac{\partial x(m)}{\partial y(n)} = (x,y) \delta_{m,n}$ and is extended to all of $S(\widehat{h}^-)$ linearly and by derivation.

For a product of elements $\alpha_i(m_i)$ where $m_i \in \Z$ we denote by $: \prod \alpha_i(m_i) :$ the same product but in ascending order of the $m_i$'s.  This is known as \emph{normal ordering}.

It will be useful to consider power series of the form
$$\alpha(z) = \sum\limits_{n \in \Z} \alpha(n) z^{-n-1},$$
and their commutators, which is an easy computation.

\begin{lemma} \label{commutatorbosonicfields}
Suppose $\alpha,\beta \in h$.  Then
$$[\alpha(z_1),\beta(z_2)] = -(\alpha,\beta) z_2^{-1} \frac{\partial}{\partial z_1} \delta(z_1/z_2).$$
\end{lemma}

If $h_i$, $i=1,...,d$ is an orthonormal basis for $L \otimes \C$, we have a distinguished element
$$\omega = \frac{1}{2} \sum\limits_{i=1}^{d} h_i(-1)^2 \in V_L.$$

For a given element $e^\alpha$ of $V_L$, we may define vertex operators in \break $\ed(V_L)[[z,z^{-1}]]$ by
$$Y(e^\alpha,z) = \exp\left( \sum\limits_{n > 0} \frac{\alpha(-n)}{n} z^n \right) e^\alpha z^{\alpha(0)} \exp\left( - \sum\limits_{n > 0} \frac{\alpha(n)}{n} z^{-n} \right).$$

For an arbitrary element $v = \prod\limits_{i=1}^N s_i(-n_i) \otimes e^\alpha$, we define, again in \break $\ed(V_L)[[z,z^{-1}]]$, the operator
$$Y(v,z) = : Y(e^\alpha,z) \prod\limits_{i=1}^N \frac{1}{(n_i-1)!} \left(\frac{d}{dz}\right)^{n_i-1} s_i(z):$$

If we write $Y(\omega,z) = \sum\limits_{n \in \Z} L_n z^{-n-2} = \sum\limits_{n \in Z} \omega_n z^{-n-1}$, we have
$$L_n = \frac{1}{2} \sum\limits_{j \in \Z} \sum\limits_{i=1}^d : h_i (-j) h_i (j+n) :$$
and we can easily check that for an element $v = \prod\limits_{i=1}^N s_i(-n_i) \otimes e^\alpha \in V_L$,

$$L_0(v) = \left( \sum\limits_{i=1}^N n_i + \frac{(\alpha,\alpha)}{2} \right) v.$$
We will refer to the $L_0$-eigenvalue of an homogeneous element $v$ as \emph{weight}.

\begin{theorem} \label{latticevertextheorem}
The vector space $V_L$ along with operators $Y(.,z)$ and conformal element $\omega$ is a Vertex Operator Algebra.
\end{theorem}
\begin{proof}
We do not prove this theorem.  For a proof, see \cite{voa} or \cite{lepowsky}.  However, we do compute the commutator $[Y(\omega,z_1), Y(\omega,z_2)]$ as preparation for similar computations in a later section, because we will be using the same techniques.

Using the commutator formula, equation (\ref{commutatorformula}), we obtain
\begin{equation} \label{equation1}
[Y(\omega,z_1),Y(\omega,z_2)] = \res_{z_0} z_2^{-1} \left( e^{-z_0 \frac{\partial}{\partial z_1}} \delta(z_1/z_2) \right) Y(Y(\omega,z_0)\omega,z_2).
\end{equation}
It is easy to compute
\begin{align*}
\omega_0 \omega &= L_{-1} \omega \\
\omega_1 \omega &= L_0 \omega = 2 \omega \\
\omega_2 \omega &= L_1 \omega = 0 \\
\omega_3 \omega &= L_2 \omega = \frac{d}{2} \\
\omega_n \omega &= L_n \omega = 0 \text{ for } n \geq 4
\end{align*}
Therefore the singular part of $Y(\omega,z_0) \omega$ is
$$Y^{-}(\omega,z_0)\omega := L_{-1} \omega z_0^{-1} + 2 \omega z_0^{-2} + \frac{d}{2} z_0^{-4}$$
hence computing residues in equation (\ref{equation1}) and using $Y(L_{-1} \omega,z_2) = \frac{d}{dz_2} Y(\omega,z)$, we find
\begin{align*}
[Y(\omega,z_1),Y(\omega,z_2)] &= \left( \frac{d}{d z_{2}} Y(\omega,z_2) \right) z_2^{-1} \delta(z_1/z_2) - 2 Y(\omega,z_2) z_2^{-1} \frac{\partial}{\partial z_1} \delta(z_1/z_2) \\ &-\frac{d}{12} z_2^{-1} \left( \frac{\partial}{\partial z_1} \right)^3 \delta(z_1/z_2).
\end{align*}
From there, computing the commutation relations of $L_m,L_n$ is a matter of equating coefficients in this result. \end{proof}
It will be useful to know the dimension of graded spaces $V_L = \bigoplus\limits_{\alpha \in L} (V_L)_{\alpha}$.  These are infinite dimensional vector spaces, but if we grade again by weight we obtain the following :

\begin{proposition} \label{voadimension}
Let $((V_L)_{\alpha})_n$ be the subspace of $(V_L)_{\alpha}$ of elements of weight $n$.  Then
$$\dim ((V_L)_{\alpha})_n = p_l\left(n - \frac{(\alpha,\alpha)}{2} \right)$$
where $p_l(m)$ is the number of integer partitions of $m$ in $l$ colors.
\end{proposition}

\begin{proof}
It is enough to write an explicit basis for $((V_L)_{\alpha})_n$ and the result is obvious.
\end{proof}

There is quite a bit of structure that can be found in $V_L$.  Indeed, write $Y(v,z) = \sum\limits_{n \in \Z} Y_n(v) z^{-n-1}$. 
For $\alpha, \beta \in L$ and $h \in L \otimes \C$, we have
\begin{align*}
[h(m), Y_n(e^{\alpha})] &= (h,\alpha) Y_{n+m}(e^{\alpha}) \\
[Y_n(e^{\alpha}), Y_m(e^{\beta})] &= 0 \text{ if }(\alpha,\beta) \geq 0 \\
[Y_n(e^{\alpha}), Y_m(e^{\beta})] &= \epsilon(\alpha,\beta) Y_{m+n}(e^{\alpha+\beta}) \text{ if }(\alpha,\beta) = -1 \\
[Y_n(e^{\alpha}),Y_m(e^{-\alpha})] &= -(\alpha(m+n) + n \delta_{m,-n}) \text{ if } (\alpha,\alpha) = 2
\end{align*}
which leads to the following theorem from \cite{frenkel}, based on earlier work in \cite{frenkelkac} and \cite{segal}  :
\begin{theorem} \label{vertexrepresentation}
Suppose $A = (a_{ij})_{1 \leq i,j \leq n}$ is a generalized Cartan matrix such that for $i \neq j$, we have $a_{ij} = 0$ or $a_{ij} = -1$.  Then the map $\pi$ defined by 
\begin{align*}
e_{\alpha_i} &\rightarrow Y_0(e^{\alpha_i}) \\
f_{\alpha_i} &\rightarrow Y_0(-e^{-\alpha_i}) \\
h_i &\rightarrow Y_0(h_i(-1)) = h_i(0)
\end{align*}
defines a representation $\pi$ of $g(A)$.  Furthermore, if $\alpha$ is a real root of $g(A)$ then if we write $x_{\alpha} \in g_{\alpha}$, $\pi(x_\alpha)$ is a multiple of $Y_0(e^{\alpha})$.
\end{theorem}

\begin{proof}
The previous formulas shows that the elements $Y_0(e^{\alpha_i}), Y_0(e^{-\alpha_i})$ and $h_i(0)$ form a Lie algebra.  Furthermore, this Lie algebra satisfies the same relations as that of a Kac-Moody algebra.  Therefore $\pi$ defines an homomorphism of Lie algebras, hence a representation.

The second statement of the theorem can be proven directly, but we will provide a simple proof later.
\end{proof}
\begin{remark}
Observe that in the representation defined above, the only elements $v \in g(A)$ such that $\pi(v) = 0$ originate from elements $c \in h$ which are central in $g(A)$.  To see this, decompose $A$ into its indecomposable components $A_i$ and observe that $g(A_i)/Z(g(A_i))$ is simple. 
\end{remark}

There is a special version of this theorem for affine Lie algebras which can also be generalized to double affine Lie algebras, and the proof is essentially the same, also from \cite{frenkel} :
\begin{theorem}
Suppose $g$ is a finite-dimensional semisimple Lie algebra $g$ which is a direct sum of Lie algebras of type $A_n, D_n$ or $E_n$.  Then the map $\pi$ defined by
\begin{align*}
e_{\alpha_i} \otimes t^n &\rightarrow Y_n(e^{\alpha_i}) \\
f_{\alpha_i} \otimes t^n &\rightarrow Y_n(-e^{-\alpha_i}) \\
h_i \otimes t^n &\rightarrow h_i(n)
\end{align*}
defines a level 1 representation of the affine Lie algebra $\tilde{g}$.
\end{theorem}
As a corollary, the theorem was applied to hyperbolic Kac-Moody algebras \cite{frenkel} : 
\begin{corollary}
Suppose $g$ is a finite-dimensional simple Lie algebra of type $A_n, D_n$ or $E_n$ or a hyperbolic Kac-Moody algebra of type $A_n^{++}, D_n^{++}, E_n^{++}$, then $\pi$ defines a faithful representation of $g$. 
\end{corollary}

\subsection{The bilinear form $(.,.)$} \label{subsectionbilinearform}

The goal of this section is to construct a bilinear form $(.,.)$ which is compatible with the definition of vertex operators.  First, the vertex operator algebra $V_L$ has an important involution.  For $\alpha \in L$, define
$$\theta(\alpha) = - \alpha.$$
Then $\theta$ can be lifted to be an involution of $\hat{L}$ such that $\overline{\theta} = \theta$.  We lift $\theta$ to an involution of $\C\{L\}$.  The section $\alpha \to e_{\alpha}$ can be chosen in such a way that
$$\theta(e_{\alpha}) = e_{-\alpha}$$
is satisfied.  Therefore, $\theta$ also lifts to an involution of $V_L$ where
$$\theta(e^{\alpha}) = e^{-\alpha} = \epsilon(\alpha,-\alpha)(e^{\alpha})^{-1}$$
and
$$\theta(\alpha(n)) = -\alpha(n).$$
Perhaps most importantly is the following proposition :
\begin{proposition}
The involution $\theta$ on $V_L$ is an automorphism of $V_L$ as a Vertex Operator Algebra.
\end{proposition}

Suppose $u \in V_L$.  There is a way to think of $u$ as an element of $(V_L)'$, the graded dual of $V_L$.  To do so, define as operators on $V_L$,
$$(e^{\alpha})^* = \epsilon(-\alpha,\alpha)e^{-\alpha},$$
$$\alpha(n)^* = -\alpha(n)$$
$$(xy)^* = y^* x^* \text{ for } x,y \in V_L.$$
Then in general, we define
$$\langle u, v \rangle = \text{Coefficient of } e^0 \text{ in } \res_{t} t^{-1} u^* v \in \C.$$
This defines a map $\phi : V_L \to (V_L)'$, although one needs to show that $\phi$ is well-defined.  Indeed, we have
\begin{proposition}
The map $\phi : V_L \to (V_L)'$ is a well-defined linear map.
\end{proposition}

\begin{proof}
The only non-trivial part is to verify that
$$\phi(\epsilon (\alpha,\beta) e^{\alpha + \beta}) = \phi(e^{\alpha} e^{\beta}) = \phi(e^{\beta}) \phi(e^{\alpha}).$$
Observe that $\phi(e^{\gamma})e^{\delta}$ is zero unless $\gamma = \delta$.  Therefore, it is enough to check equality when applied to $e^{\alpha+\beta}$.
On the left-hand side, we have
$$\phi(\epsilon(\alpha,\beta) e^{\alpha+\beta})e^{\alpha+\beta} = \epsilon(\alpha,\beta)$$
On the right-hand side, we have
\begin{align*}
\phi(e^{\beta})\phi(e^{\alpha})e^{\alpha+\beta} &= \epsilon(-\alpha,\alpha) \epsilon(-\beta,\beta) \epsilon(-\alpha,\alpha+\beta) \epsilon(-\beta,\beta) \\ &= \epsilon(-\alpha,\alpha) \epsilon(-\alpha,\alpha+\beta)
\end{align*}
However, by the cocycle conditions, we have for $x,y,z \in L$,
$$\epsilon(x,y)\epsilon(x+y,z) = \epsilon(x,y+z)\epsilon(y,z).$$
Replacing $x$ with $-\alpha$, $y$ with $\alpha$ and $z$ with $\beta$ we find
$$\epsilon(-\alpha,\alpha)\epsilon(0,\beta) = \epsilon(-\alpha,\alpha+\beta)\epsilon(\alpha,\beta).$$
which proves the proposition.
\end{proof}

We know from Section \ref{dualmodule} that $(V_L)'$ is also a $V_L$-module.  We twist the action of $V_L$ on $(V_L)'$ by $\theta$, i.e.
$$\langle Y_{\theta}'(v,z)u, w \rangle = \langle Y'(\theta v,z) u, w \rangle.$$
The following lemmas will be necessary to see that the map $\phi$ and $Y'(\theta v,z)$ are essentially the same.

\begin{lemma} \label{adjointlemma1} Suppose $\alpha \in L \otimes \C$.  Then
$\alpha(z)^* = z^{-2} \alpha(z^{-1})$.
\end{lemma}
\begin{proof}
Write
\begin{align*}
\alpha(z)^* &= \left( \sum\limits_{n \in Z} \alpha(n) z^{-n-1} \right)^* \\
&= \sum\limits_{n \in \Z} \alpha(-n) z^{-n-1} \\
&= \sum\limits_{n \in \Z} \alpha(n) z^{n-1} \\
&= z^{-2} \sum\limits_{n \in \Z} \alpha(n) z^{n+1} \\
&= z^{-2} \sum\limits_{n \in \Z} \alpha(n) (z^{-1})^{-n-1}
\end{align*}
as desired.
\end{proof}

\begin{lemma}Suppose $1 \otimes e^{\alpha} \in V_L$ such that $(\alpha,\alpha) = k \in 2\Z$.  Then as operators on $V_L[z,z^{-1}]$,
$Y(e^{\alpha},z)^* = (-1)^{k/2}(z)^{-k} Y(e^{\alpha},z^{-1})$.  In particular, $$(e^{\alpha})_n^* = (-1)^{k/2}(e^{-\alpha})_{-n-2+k}.$$
\end{lemma}
\begin{proof}
Assuming we show that $$Y(1 \otimes e^{\alpha},z)^* = z^{-k} Y(1 \otimes e^{-\alpha},z^{-1})$$ as operators in $V_L[z,z^{-1}]$, then we will have
$$Y(1 \otimes e^{\alpha},z)^* = \sum\limits_{n \in \Z} (1 \otimes e^{\alpha})_n^* z^{-n-1} = \sum\limits_{n \in \Z} (1 \otimes e^{-\alpha})_n z^{n+1-k}$$
and then $(1 \otimes e^{\alpha})_n^* = (1 \otimes e^{-\alpha})_{-n}$ by comparing coefficients as in the previous lemma.

To prove the assertion, note that clearly
$$\exp\left( \sum\limits_{n > 0} \frac{\alpha(-n)}{n} z^n \right)^* = \exp\left( \sum\limits_{n > 0} \frac{\alpha(n)}{n} (z^{-1})^{-n} \right),$$
$$\exp\left( - \sum\limits_{n > 0} \frac{\alpha(n)}{n} z^{-n} \right)^*=\exp\left( - \sum\limits_{n > 0} \frac{\alpha(n)}{n} (z^{-1})^{n} \right)$$
and
$$(e^\alpha z^{\alpha(0)})^* = \epsilon(\alpha,-\alpha) z^{\alpha(0)} e^{-\alpha} = (-1)^{\frac{1}{2}(\alpha,\alpha)} z^{-k} e^{-\alpha} z^{\alpha(0)}.$$
Therefore,
\begin{align*}
Y(e^\alpha,z)^* &= \exp\left( - \sum\limits_{n > 0} \frac{\alpha(n)}{n} (z^{-1})^{-n} \right) z^{-k} e^{-\alpha} z^{\alpha(0)} \exp\left( \sum\limits_{n > 0} \frac{\alpha(n)}{n} (z^{-1})^n \right) \\ &= (-1)^{k/2}(z)^{-k} Y(e^{-\alpha},z^{-1}) 
\end{align*}
as desired.
\end{proof}

From these two lemmas we now prove :

\begin{theorem}
The map $\phi : V_L \to (V_L)'$ agrees with $Y'(\theta v,z)$, i.e.
$$\phi(Y(v,z)u)w = \langle Y(v,z)u,w \rangle = \langle u, Y'(\theta v,z)w \rangle$$
for $u,v,w \in V_L$.
\end{theorem}
\begin{proof}
We have computed $\alpha(z)^*$ and $Y(e^{\alpha},z)^*$.  On the other hand,
$$\alpha(z) = Y(\alpha(-1),z)$$
hence
\begin{align*}
\langle Y'_{\theta}(\alpha(-1),z)u,w \rangle &= \langle u, Y(e^{z L_1} (-z^2)^{L_0} \theta(\alpha(-1)),z^{-1})w \rangle \\ &= \langle u, Y(z^{-2} \alpha(-1),z^{-1})w \rangle
\\ &= \langle u, z^{-2} \alpha(z) w \rangle.
\end{align*}
Secondly, suppose $(\alpha,\alpha) = k \in 2\Z$.  Then
\begin{align*}
\langle Y'_{\theta}(e^{\alpha},z)u,w \rangle &= \langle u, Y(e^{z L_1} (-z^2)^{L_0} \theta(e^{\alpha}),z^{-1})w \rangle  \\ &= \langle u, (-z^2)^{-k/2} Y(e^{-\alpha},z^{-1})w \rangle
\end{align*}
Because every element of $V_L$ is of the form $u^1_{n_1}...u^l_{n_l} e^{\alpha}$, we see that $\phi$ is compatible with the structure of $(V_L)'$ as a $V_L$-module.
\end{proof}

As a culmination of all of our efforts, we now see that we may define a bilinear form $(.,.)_{V_L} = (.,.)$ on $V_L$ by defining the two adjoint operators $\alpha(n)^* = \alpha(-n)$ and $(e^{\alpha})^* = \epsilon(\alpha,-\alpha) e^{-\alpha}$ and then extending to all of $V_L$ as before.  It is then clear that for $v \in V_L$, we have
$$(Y(v,z)u,w) = (u, Y(e^{zL_1}(-z^2)^{-L_0}\theta(v),z^{-1})w)$$

It is important to observe that this bilinear form is symmetric and also that
$$((V_L)_{\alpha},(V_L)_{\beta}) \neq 0 \text{ only if } \alpha = \beta.$$

This bilinear form has many properties which will be important to us.

\begin{proposition} \label{formpositivedefinite}
The bilinear form $(.,.)$ on $V_L$ is positive-definite if and only if the lattice $L$ is positive definite.
\end{proposition}
\begin{proof}
If $L$ is not positive definite, there exists $\alpha \in L$ such that $(\alpha,\alpha) \leq 0$.  Then
$$(\alpha(-1),\alpha(-1)) = (1,\alpha(1)\alpha(-1)) = (1, (\alpha,\alpha)) \leq 0.$$
This proves the "only if" part of the proposition.  To prove the "if" part of the proposition, assume $L$ is positive-definite.  Then let $h_1,...,h_d$ be an orthonormal basis for $L \otimes \R$.  Then it is easy to check that
$$(h_{i_1}(-j_1)...h_{i_r}(-j_r) e^{\alpha}, h_{k_1}(-l_1)...h_{k_s}(-l_s) e^{\beta})$$
is non-zero if and only if $r = s$ and there is a permutation $\sigma$ between the indexes $i,k$ and $j,l$.  Furthermore it is clear by induction on $r$ that if this expression is non-zero it must be positive.

Therefore, we may rescale the elements $h_{i_1}(-j_1)...h_{i_r}(-j_r) e^{\alpha}$ in such a way that they form an orthonormal basis for $S(\widehat{h}^{-}) \otimes_{\R} \R[L]$.  In particular, $(.,.)$ will also be positive-definite on $V_L$.
\end{proof}

\begin{lemma} \label{adjointvirasoro}
Suppose $n \in \Z$.  The adjoint of the operator $L_n$ on $V_L$ is $L_{-n}$.
\end{lemma}
\begin{proof}
Recall that
$$L_n = \frac{1}{2} \sum\limits_{i=1}^d \sum\limits_{k \in \Z} : h_i(n-k)h_i(k) :$$
Observe that the adjoint of a normally ordered expression is again normally ordered.  Therefore,
$$:h_i(n-k)h_i(k):^* = :h_i(k-n)h_i(-k):$$
Replacing $k$ by $-k$ and summing over all $k \in \Z$ we recover $L_{-n}$.
\end{proof}

Next, suppose $v \in V_L$ satisfies $L_0 v = v$ and $L_n v = 0$ for $n>0$.  Then we can see that
$$\langle Y(v,z)u, w \rangle = \langle u, Y(e^{z L_1} (-z^{2})^{-L_0}\theta(v),z^{-1})= \langle u, Y(-z^{-2} \theta(v),z^{-1}) \rangle.$$
As usual, if we write
$$Y(v,z) = \sum\limits_{n \in \Z} v_n z^{-n-1}$$
Then we have
$$Y(v,z)^* = \sum\limits_{n \in \Z} v_n^* z^{-n-1}.$$
On the other hand, we have
$$z^{-2} Y(\theta(v),z^{-1}) = \sum\limits_{n \in \Z} -\theta(v)_n z^{n-1}$$
and comparing the coefficients for $z^{-1}$ yields the following :

\begin{corollary} \label{adjointcorollary}
Suppose $v \in V_L$ such that $L_0 v = v$ and $L_n v = 0$ for $n>0$.  Then $(v_0)^* = -\theta(v)_0$.
\end{corollary}

We have now completed the collection of all information needed about the bilinear form $(.,.)$.  We are almost ready to define a Lie algebra structure on a subset of $V_L$. It will be important to know some information about the Lie bracket of Virasoro operators with $Y(v,z)$.  The following proposition is proven in \cite{voa}, and we provide the proof again here because we will adapt it in a later section.
\begin{lemma} \label{highestweightcommutation}
Suppose $v \in V_L$ satisfies $L_0 v = hv$ and $L_n v = 0$ for all $n > 0$.  Then
$$[L_m, v_n] = (h(m+1) - m - n - 1) v_{m+n}.$$
\end{lemma}
\begin{proof}
Using the commutator formula,
$$[Y(\omega,z_1),Y(v,z)] = \res_{z_0} z^{-1} \left( e^{-z_0 \frac{\partial}{\partial z_1}} \delta(z_1/z) \right) Y(Y(\omega,z_0)v,z)).$$
As mentioned earlier, we need only consider the singular part of $Y(\omega,z_0)$ in the right hand side, allowing us to rewrite this commutator as
\begin{align*}
[Y(\omega,z_1),Y(v,z)] &= z^{-1}Y(L_{-1}v,z) \delta(z_1/z) - z^{-1} Y(L_0 v,z) \frac{\partial}{\partial z_1} \delta(z_1/z) \\
&+ \res_{z_0} z^{-1} \sum\limits_{n > 0} Y(L_n v,z) z_0^{-n-2} \left( e^{-z_0 \frac{\partial}{\partial z_1}} \delta(z_1/z) \right)
\end{align*}
In general this expression is unwieldy but using the conditions of the propositions, it reduces to
\begin{align*}
[Y(\omega,z_1),Y(v,z)] &= z^{-1} Y(L_{-1}v,z) \delta(z_1/z) - z^{-1}Y(L_0 v,z) \frac{\partial}{\partial z_1} \delta(z_1/z) \\
&= z^{-1} \left(\frac{d}{d z} Y(v,z) \right) \delta(z_1/z) - h z^{-1} Y(v,z) \frac{\partial}{\partial z_1} \delta(z_1/z).
\end{align*}
The coefficient of $z_1^{-m-2}$ in this expression yields
$$[L_m,Y(v,z)] = \left( z^{m+1} \frac{d}{dz} + h(m+1) z^m \right) Y(v,z)$$
and the coefficient of $z^{-n-1}$ in this expression results in
$$[L_m,v_n] = (h(m+1) - m - n - 1) v_{m+n},$$
as desired.
\end{proof}

\subsection{No-ghost theorem and the Lie algebra of physical states}
Define the subspace
$$P^i = \{ v \in V_L | L_0(v) = iv \text{ and } L_n(v) = 0 \text{ for all } n > 0 \}.$$
Elements of $P^i$ are know as \emph{lowest weight vectors} for the Virasoro Lie algebra.  Note that if $v \in P^i$, $L_0 L_m v = (i-m) L_m v$ and then, in particular, $L_{-1} P^0 \subseteq P^1$.  We refer to elements of $P^1$ as \emph{physical states}.

\begin{theorem} \label{liealgebraofphysicalstates}
Consider the set $\widetilde{P^1} = P^1/L_{-1} P^0$.  Let $u,v \in \widetilde{P^1}$.  Then the bracket
$$[u,v] = u_0 v$$ defines a Lie algebra structure on $\widetilde{P^1}$.
\end{theorem}

\begin{proof}
It is clear that $[u,v]$ is bilinear.  To prove that this bracket is well-defined, we need to check that $[u,v] \in P^1$ and that if $v \in L_{-1} P^0$ then $[u,v] = 0$.  To see these, note that from the previous lemma,
$$[L_{m}, x_n] = (i(m+1)-m-n-1) x_{m+n}$$
for all $m,n \in \Z$ and $x \in P^i$.  This identity implies that $[L_{m},x_0] = 0$ for all $x \in P^1$.  Then for $u, v \in P^1$ we have
$$L_0 u_0v = u_0 L_0v = u_0v$$
and
$$L_n u_0 v = u_0 L_n v = 0$$ for $n > 0$.  This proves that $[u,v] \in P^1$.
Next suppose that $v \in L_{-1} P^0$.  Then $v_0 = 0$ because
$$Y(L_{-1}v,z) = \frac{d}{dz} Y(v,z)$$
hence $[u,v] = 0$.

Next we prove antisymmetry.  Because $V_L$ is a vertex operator algebra, we have by skew-symmetry the identity
$$Y(u,z)v = e^{z L_{-1}} Y(v,-z) u.$$
In particular, looking at coefficients of $z^{-1}$, we obtain
$$u_0v = \res_z Y(u,z) v = - v_0 u + \sum\limits_{i \geq 1} \frac{(-1)^{i+1}}{i!} L^i_{-1} (v_i u).$$
Therefore, $[u,v] = -[v,u] + L_{-1}x$ for some $x$.  It is clear that $L_{-1} x \in P^1$ but we must show that $x \in P^0$.  First observe that $L_{-1} y = 0$ if and only if $y \in \C \otimes e^0$, which is clearly already a subset of $P^0$, and that each $L_n$ preserves $S(\hat{h}^{-}) \otimes e^{\alpha}$ and that $L_{-1}$ is injective on each of these subspaces as long as $\alpha \neq 0$.  Therefore, we can assume $x \in S(\hat{h}^{-}) \otimes e^{\alpha}$ for some $\alpha \neq 0$.  In this case, observe that
$$L_{-1} x = L_0 L_{-1} x = L_{-1} L_0 x + L_{-1} x$$ hence $L_0 x = 0$.  Then assume by induction that $L_n x = 0$ and we see that
$$0 = L_{n+1} L_{-1} x = L_{-1} L_{n+1} x + (n+2) L_{n} x = L_{-1} L_{n+1} x$$
which proves that $L_{n+1} x = 0$ because $L_{-1}$ is injective, hence $x \in P^0$ by induction.  In particular, we conclude that $[u,v] = -[v,u]$ in the quotient $P^1/L_{-1}P^0$.

Finally we prove Jacobi identity for Lie algebras. Using Corollary \ref{commutatorcorollary}, we have
$$u_0(v_0 w) - v_0(u_0 w) = (u_0 v)_{0} w$$
which is equivalent to
$$[u,[v,w]] - [v,[u,w]] = [[u,v],w].$$
which proves Jacobi identity for Lie algebras.
\end{proof}

\begin{remark}\label{vertexrepresentationinclusion}
Observe that the representation $\pi$ defined in Theorem \ref{vertexrepresentation} in fact defines a Lie algebra homomorphism $\pi \rightarrow P^1$.  This provides a proof for the second part of the theorem because the only elements of $P^1$ corresponding to real roots are multiples of $1 \otimes e^{\alpha}$.
\end{remark}

Going back to the bilinear form $(.,.)$ we find the following.

\begin{proposition}
Suppose, $x,y,z \in \widetilde{P^1}$.  Then
$$([x,y],z) = (x,[\theta(y),z]).$$
Consequently, its radical is an ideal of $\widetilde{P^1}$.
\end{proposition}

\begin{proof}
Using Corollary \ref{adjointcorollary}, we obtain
$$([x,y],z) = -(y_0 x,z) = (x,\theta(y)_0 z) = (x,[\theta(y),z])$$
which proves the proposition.
\end{proof}

From now on, we may refer to elements in the radical of $(.,.)$ as \emph{null states}.  We wish to provide a basis for $\widetilde{P^1}$.  To do so we first define a family of elements known as transversal states, introduced by Del Giudice, Di Vecchia and Fubini (DDF) in \cite{ddf}.  We will refer to this construction as \emph{DDF construction} and the corresponding operators as \emph{DDF operators}.  The algebra spanned by these operators is known as the \emph{spectrum-generating algebra}.

\begin{lemma} \label{isotropicc}
Suppose $\alpha$ is a non-zero element of $L$.  There exists $c \in h = L \otimes \R$ such that $(c,\alpha) = 1$ and $(c,c) = 0$.
\end{lemma}
\begin{proof}
Because $L$ is Lorentzian, $L \otimes \R = \R^{d-1,1}$ where an element $(a_1,...,a_d)$ of $\R^{d-1,1}$ has norm $a_1^2+...+a_{d-1}^2-a_d^2$.  Write $\alpha = (\alpha_1,...,\alpha_d)$.  Because $\alpha \neq 0$, there exists $l$ such that $\alpha_l \neq 0$.  If $\alpha_l \neq \alpha_d$, choose $c = \frac{1}{\alpha_l - \alpha_d} (0,...,1,...,0,1)$.  If all $\alpha_l = \alpha_d$ choose $c = \frac{1}{2 \alpha_d} (0,...,1,...,0,-1)$.
\end{proof}

We now fix $\alpha \in L$.  Write
$$S(\alpha) = \{ v \otimes e^{\alpha} \in V_L\},$$
$$V_i(\alpha) = \{ u=v \otimes e^{\alpha} | L_0u = iu \}$$
and
$$P^i(\alpha) = \{ v \otimes e^{\alpha} \in P^i \}.$$

By the lemma, there exists $c \in h$ such that $(c,\alpha) = 1$ and $(c,c) = 0$.  For $a \in h$ such that $(a,c) = 0$ and $m \in \Z$, consider very special operators
$$A_m^a = \res_z Y(a(-1) e^{mc},z).$$
These operators do not in general act on $V_L$ because $c$ is not an element of $L$.  If however we replace $L$ with $L \otimes \Q$ then $A_m^a$ will map $S(\alpha)$ to $S(\alpha+mc)$.  This will be enough for our purposes, and we do not need any additional structure, including the structure of vertex operator algebras, although we imitate it.  All identities can be proven directly without the need for the Jacobi identity.  We list some properties of these operators.

\begin{proposition}\label{ddfproperties} Suppose $a,a_1,a_2 \in L \otimes \R$ such that $(a,c) = (a_1,c) = (a_2,c) = 0$ and $\alpha \in L$ and $m,m_1,m_2 \in \Z$.  Then
\begin{enumerate}[label=\alph*)]
\item $[A^{a_1}_{m_1},A^{a_2}_{m_2}] = m_1(a_1,a_2) \delta_{m_1+m_2,0} c(0)$
\item $[L_n,A^{a}_m] = 0$
\item $[c(m),A^a_m] = 0$.
\item $A^a_m e^\alpha = 0$ for all $m > 0$
\item $(A^a_m)^* = A^a_{-m}$
\end{enumerate}
\end{proposition}
\begin{proof}
To prove part a), we compute the bracket directly without the need for normal ordering, using Lemma \ref{commutatorbosonicfields}.  Indeed,
\begin{align*}
[A_{m_1}^{a_1},A_{m_2}^{a_2}] &= \res_{z_1} \res_{z_2} [a_1(z_1),a_2(z_2)] Y(e^{m_1 c},z_1) Y(e^{m_2 c},z_2) \\
&= - \res_{z_1} \res_{z_2} (a_1,a_2) z_2^{-1} \frac{\partial}{\partial z_1} \left( \delta(z_1/z_2)\right) Y(e^{m_1 c},z_1) Y(e^{m_2 c},z_2) \\
&= \res_{z_1} \res_{z_2} (a_1,a_2) z_2^{-1} \delta(z_1/z_2) \frac{\partial}{\partial z_1} \left( Y(e^{m_1 c},z_1) Y(e^{m_2 c},z_2) \right) \\
&= \res_{z} (a_1,a_2) z^{-1} m_1 c(z) Y(e^{(m_1 + m_2) c},z)
\end{align*}
which is $0$ unless $m_1 + m_2 = 0$, and then residues result in
$$[A_{m_1}^{a_1},A_{m_2}^{a_2}] = m_1 (a_1,a_2) \delta_{m_1+m_2} c(0),$$
as desired.

To prove part b), it is enough to observe that $a(-1) e^{mc}$ is a lowest weight vector for the Virasoro Lie algebra of weight $1$ then apply Lemma \ref{highestweightcommutation}.  

Part c) is obvious and part d) follows from a simple computation.  Part e) follows from the fact that $a(-1) e^{mc} \in P^1$ and then use Corollary \ref{adjointcorollary}.
\end{proof}

Consider the subspace $T_M(\alpha,c)$ of $S(\alpha)$ spanned by elements $$A_{-m_1}^{a_1} A_{-m_2}^{a_2} ... A_{-m_k}^{a_k} e^{\alpha + M c}$$ for $a_1,...,a_k \in h$ satisfying $(a_i,c) = 0$ and $m_1 + m_2 + ... + m_k = M$. If $M$ is chosen so that $(\alpha,\alpha) = 2 - 2M$, then we write $T_M(\alpha,c) = T(\alpha,c)$.  Observe that $T(\alpha,c)$ is a subspace of $P^1(\alpha)$.  

Let $G(\alpha,c)$ be the subspace of $S(\alpha)$ spanned by elements 
\begin{equation} \label{Galpha}
L_{-1}^{\lambda_1} ... L_{-n}^{\lambda_n}  c(-1)^{\mu_1}... c(-m)^{\mu_m} v
\end{equation}
such that $v \in T_M(\alpha,c)$ for some $M$ and $\lambda_i, \mu_i \in \Z_+, \sum \lambda_i + \mu_i > 0$. Then we have the following lemma, as a first step in the proof of the no-ghost theorem.

\begin{lemma} \label{independentlemma} Suppose $\alpha \in L$.  Then
\begin{enumerate}[label=\alph*)]
\item The space $T(\alpha,c)$ is positive definite.
\item $S(\alpha)$ is generated by $T(\alpha,c)$ and $G(\alpha,c)$ and the generators of $T(\alpha,c)$ and $G(\alpha,c)$ are all linearly independent.
\end{enumerate}

\end{lemma}

\begin{proof}
The proof is based on \cite{noghost}, to which we add some details.  Part a) follows from Proposition \ref{ddfproperties} part a) and e).  We show that $S(\alpha) = T(\alpha,c) + G(\alpha,c)$.  It is enough to check that the generators of $T(\alpha,c)$ and $G(\alpha,c)$ are linearly independent, because then they will form a basis by counting dimension.

Suppose $x = \sum\limits_{i \in I} x_i$ is a linear combination in the generators in equation (\ref{Galpha}).  By Proposition \ref{ddfproperties} a), b), c) and d), any term not containing an operator $A^a_m$ may be canceled by applying $A^a_{-m}$.  Repeating this process as many times as necessary, we see that generators $$A_{-m_1}^{a_1} A_{-m_2}^{a_2} ... A_{-m_k}^{a_k} e^{\alpha + M c}$$
are linearly independent.  Therefore, we can assume that all $x_i$ terms of the form $$L_{-1}^{\lambda_1} ... L_{-n}^{\lambda_n} c(-1)^{\mu_1}... c(-m)^{\mu_m} v$$ share the same $v \in T(\alpha,c)$.  Equivalently, we can assume $v = 1 \otimes e^{\alpha+Mc}$ because $A^a_m$ commutes with all $L_{-i}$'s and $c(-i)$'s.  Therefore, we are reduced to proving that elements
$$x_i = L_{-1}^{\lambda_1} ... L_{-n}^{\lambda_n} c(-1)^{\mu_1}... c(-m)^{\mu_m} e^{\alpha+Mc}$$
are linearly independent.  Write $\alpha' = \alpha+Mc$.  Obviously, $\alpha'$ and $c$ are linearly independent hence so are different monomials in $\alpha'(-i),c(-j)$.  Note that $L(-i) e^{\alpha'} = \alpha'(-i) e^{\alpha'}$.  Then $x_i$ will contain a monomial of the form
$$\alpha'(-1)^{\lambda_1}...\alpha'(-n)^{\lambda_n} c(-1)^{\mu_1}...c(-m)^{\mu_m}$$
which cannot appear in any other $x_i$.  Therefore they are linearly independent.
\end{proof}

This Lemma also offers us an alternate description of $T(\alpha,c)$.
\begin{corollary} \label{transervedefinitionalternative} 
Suppose $v \in P^1(\alpha)$. Then $v \in T(\alpha,c)$ if and only if $c(n)v = 0$ for all $n > 0$.
\end{corollary}

\begin{proof} \label{alternatetransverse}
By the lemma, it is clear that if $v \in T(\alpha,c)$ then $c(n)v = 0$ for all $n > 0$ by Proposition \ref{ddfproperties} part c).  Conversely, suppose $c(n)v = 0$ for all $n > 0$.  Without loss of generality, assume $v \neq 0$.  Write $v$ as a linear combination of terms of the form
$$L_{-1}^{\lambda_1} ... L_{-n}^{\lambda_n}  c(-1)^{\mu_1}... c(-m)^{\mu_m} v_i.$$
By applying $c(n)^{\lambda_n}...c(1)^{\lambda_1}$ to this expression and organizing the terms in such a way that the exponents $\lambda_1,...,\lambda_n$ are maximal while the exponents $\mu_1,...,\mu_m$ are minimal, we can see that only the terms containing the sequence $L_{-1}^{\lambda_1} ... L_{-n}^{\lambda_n}$ will have the term $c(-1)^{\mu_1}...c(-m)^{\mu_m} v_i$.  Therefore we can have no such term hence all $\lambda_1 = ... = \lambda_n = 0$.

To see why all $\mu_1 = ... = \mu_m = 0$ apply $L_m$ and use the fact that $(c,\alpha) = 1$.
\end{proof}

\begin{lemma}
For fixed $\alpha,c$, let $N^s$ denote all elements of $G(\alpha,c)$ of weight $s$ such that any of the $\lambda_i$ is non-zero.  Then the operators $L_1, \widetilde{L}_2 = L_2 + \frac{3}{2} L_1^2$ are well-defined maps $N^1 \to N^0$ and $N^1 \to N^{-1}$ if and only if $d = 26$.
\end{lemma}

\begin{proof}
Write $y = L_{-n} x$ such that $x \in V_L$ and $L_0 y = sy$.  Suppose $y = L_{-n} x \in N^1$.  Then we can write
$$y = L_{-1} x_1 + L_{-2} x_2$$
with $L_0 x_1 = 0$ and $L_0 x_2 = -x_2$.  This decomposition is justified because $L_{-n}$ can be written as a polynomial in $L_{-1}$ and $L_{-2}$ in the universal enveloping algebra of the Virasoro Lie algebra.  Notice that
\begin{align*}
L_1 (L_{-1} x_1) &= L_{-1} L_1 x_1 + 2 L_0 x_1 \\
&= L_{-1} L_1 x_1 \in N^0\\ 
L_2 (L_{-1} x_1) &= L_{-1} L_2 L_{-1} x_1 + 3 L_1 x_1 \\
L_1^2 (L_{-1} x_1) &= L_1 (L_{-1} L_1 x_1) \\
&= L_{-1} L_1^2 x_1 + 2 L_0 L_1 x_1 \\
&= L_{-1} (...) - 2 L_1 x_1
\end{align*}
It is clear from these computations that $L_2 + \frac{3}{2}L_1^2(L_{-1} x) \in N^{-1}$.  Observe that this is the reason why we require the operator $\widetilde{L}_{2}$.  Indeed, the operator $L_2$ is not a map $N^1 \to N^{-1}$ in general.  Next, we also compute
\begin{align*}
L_1 (L_{-2} x_2) &= L_{-2} L_1 x_2 + 3 L_{-1} x_2 \in N^{0}\\ 
L_2 (L_{-2} x_2) &= L_{-2} L_2 x_2 + 4 L_0 x_2 + \frac{d}{2} x_2 \\
&= L_{-2} (...) - 4 x_2 + \frac{d}{2} x_2 \\
L_1^2 (L_{-2} x_2) &= L_1 (L_{-2} L_1 x_2 + 3 L_{-1} x_2) \\
&= L_{-2} L_1^2 x_2 + 3 L_{-1} L_1 x_2 + 3 L_{-1} L_1 x_2 + 6 L_0 x_2 \\
&= L_{-1}(...) + L_{-2} (...) - 6 x_2 
\end{align*}
Then $(L_2 + \frac{3}{2}L_1^2)(L_{-2} x_2) \in N^{-1}$ if and only if $-13 + \frac{d}{2} = 0$ or equivalently, $d = 26$.
\end{proof}

We can now prove no-ghost theorem which first appeared in \cite{noghost}, but first used in this context in \cite{frenkel}, and then in \cite{borcherds1}.

\begin{theorem} \label{noghost}
Suppose $u \in P^1(\alpha)$ and that $\dim L = 26$.  Then $u = v + w$ where $v \in T(\alpha,c)$ and $w$ lies in $P^1$ and is a linear combination of terms $L_{-n} x$ for some $x \in V_L$.
\end{theorem}
\begin{proof}

Write $K^s$ to be the subspace of $V_s(\alpha)$ composed of elements of the form
$$A = \sum\limits_{i \in I} a_i  c(-1)^{\mu^i_1}... c(-m)^{\mu^i_{m}} B$$ for some $B \in T(\alpha,c)$, with the condition $L_0(A) = sA$.  It can be easily checked that $L_1, L_2$ are maps $K^s \rightarrow K^{s-1}$ and $K^s \rightarrow K^{s-2}$, respectively.

Next, suppose $u \in P^1(\alpha)$.  By the lemma, we can write $u = v + w$ where $w \in N^1$ and $v \in K^l$.  By assumption, $L(1)u = L(2)u = 0$.  Therefore, we must also have $L_1(v) = L_2(v) = 0$ and $L_1(w) = L_2(w) = 0$ because $L_1, L_2$ act independently on $N^1, K^1$ by the previous lemma.  In particular, $L_n(v) = L_n(w) = 0$ for all $n \geq 1$ because $L_1$ and $L_2$ generate all $L_n, n \geq 1$.

Therefore, we can conclude that $w$ satisfies the conditions of the theorem.  It remains to show that $v \in T(\alpha,c)$.  Write
$$v = \sum\limits_{i \in I} a_i  c(-1)^{\mu^i_1}... c(-m)^{\mu^i_{m}} v'$$
for some $v' \in T(\alpha,c)$.  It is clear that $L_n v' = 0$ by Proposition \ref{ddfproperties}.  To see why all $\mu^i_j$ must be $0$, observe that if $k$ is the largest number such that $\mu^i_k$ is non-zero, then by applying $L_k^{\mu^i_{k}} v$ to $v$ we obtain a linear combination of linearly independent vectors with the term originally $c(-k)^{\mu^i_k}$, now become $c(0)^{\mu^i_k}$ (recall as well that $(c,\alpha) = 1$), which sums up to $0$ (because $L_k(v) = 0$).  This shows that all the coefficients must be $0$ and hence that such a $k$ does not exist.
\end{proof}

We may now use the no-ghost theorem to describe the structure of the quotient of $P^1$ by the radical of $(.,.)$ more explicitly.

\begin{corollary} \label{dimensionfakemonster}
Suppose $L$ has rank 26.  Then $\widetilde{P^1} = P^1/L_{-1}P^0$ is a positive semidefinite space and the quotient $g_L$ of $\widetilde{P^1}$ by the radical of $(.,.)$ is a Lie algebra such that $$\dim g_{\alpha} = p_{24}\left( 1 - \frac{(\alpha,\alpha)}{2} \right).$$
\end{corollary}
\begin{proof}
By the theorem, we have shown that $P^1(\alpha) = T(\alpha,c) \oplus D_1$ where $D_1$ is a subspace of $P^1$ of elements of the form $L_{-n}x$ for some $x \in V_L$.  Observe that if $u \in P^1$,
$$(L_{-n}x, u) = (x, L_n u) = 0$$
and therefore that $D_1$ lies in the radical of the bilinear form $(.,.)$.  In particular, this shows that $P^1$ is positive semi-definite.  Therefore, so is $g_L$ and it is clear that $(g_L)_{\alpha} = T(\alpha,c)$, which has the dimension stated in the theorem by construction.
\end{proof}

For $L=II_{25,1}$, the Lie algebra $g_L$ is known as the \emph{fake monster Lie algebra}. We will return to study $g_L$ in more detail in the final chapter.

\subsection{Longitudinal vertex operators} \label{sectionlongitudinal}

In the previous section, we defined
$$A_m^a = \res_z Y(a(-1)e^{mc},z)$$
for $(c,\alpha) = 1$ and $(a,\alpha)=(a,c) = 0$.  It is then natural to ask about the operators $A_m^\alpha$ and $A_m^c$.  It is easy to see that $A_m^c = 0$ hence it can be discarded.  For the other operator, following \cite{gebert}, observe that
$$L_1 \cdot \alpha(-1)e^{mc} = m e^{mc}$$
hence we would obtain a third term from the commutator formula.  Adjusting the proof of Lemma \ref{highestweightcommutation} in this case, we obtain an extra term :
\begin{align*}
[L_n,Y(\alpha(-1)e^{mc},z) = & \left( z^{n+1} \frac{d}{dz} + (n+1)z^n \right) Y(\alpha(-1)e^{mc},z) \\&+ \frac{1}{2}mn(n+1) Y(e^{mc},z)z^{n-1}
\end{align*}
In turn, this implies that $[L_n,A^\alpha_m] \neq 0$.  Therefore, any application of $A^\alpha_m$ on a physical state will not in general be a physical state again.  To remedy this problem, \cite{Brower} introduces an extra term in the following way : let
$$c^{\times}(z) = z c(z) - 1 = \sum\limits_{n \neq 0}c(n)z^{-n}.$$
The second equality only makes sense as operators on physical states in $S(\alpha - Nc)$, for some $N$, because $(c,\alpha) = 1$.  Regardless, we may now introduce
$$\log(zc(z)) = \log(1+c^{\times}(z)) = \sum\limits_{i \geq 1} \frac{(-1)^{i+1}}{i} (c^{\times}(z))^i.$$
The coefficients of $z^k$ of this expression are well-defined operators on $S(\alpha-Nc)$.

\begin{proposition}
On $\bigoplus\limits_{N \in \Z} S(\alpha-Nc)$, we have the following commutators :
$$[L_n,\log(1+c^\times(z))] = z^{n+1} \frac{d}{dz} \log(1+ c^\times(z)) + nz^n$$
and consequently,
$$[L_n, \frac{d}{dz} \log(1 + c^\times(z))] = \left( z^{n+1} \frac{d}{dz} + (n+1)z^n \right) \frac{d}{dz} \log(1 + c^\times (z)) + n^2 z^{n-1}.$$
\end{proposition}
\begin{proof}
Using the convention that $c(0) = 1$, we evaluate
\begin{align*}
[L_n, c^{\times}(z)] &= [L_n,\sum\limits_{m \in \Z^\times} c(m) z^{-m}] \\
&= \sum\limits_{m \in \Z^\times} -m c(m+n) z^{-m} \\
&= z^n \sum\limits_{m \in \Z^\times} - m c(m+n) z^{-m-n} \\
&= z^n \left( z \frac{d}{dz} + n \right) c^{\times}(z) + n z^n.
\end{align*}
Next, we have using the chain rule, the fact that $(c,c) = 0$ and the cancellation of terms in the partial sum (justified by the fact that only finitely many terms will ever be non-zero when used as an operator),
\begin{align*}
[L_n, \log(1 + c^{\times}(z)] &= \sum\limits_{i \geq 1} (-1)^{i+1} (c^{\times} (z))^{i-1} [L_n, c^{\times}(z)] \\
&= \sum\limits_{i \geq 1} (-1)^{i+1} (c^{\times}(z))^{i-1} \left( z^{n} \left( z \frac{d}{dz}+n \right) c^{\times}(z) + n z^n \right) \\
&= n z^n  + \sum\limits_{i \geq 1} (-1)^{i+1} (c^{\times}(z))^{i-1} \left( z^{n+1} \frac{d}{dz} c^{\times}(z) \right) \\
&= n z^n + z^{n+1} \frac{d}{dz} \log(1 + c^{\times}(z))
\end{align*}
as desired.  The second assertion follows from the first by differentiating both sides of the equality.
\end{proof}
From there, let
$$c'(z) = \frac{d}{dz} \log(1 + c^{\times}(z)) - z^{-1}.$$
Then we find
$$[L_n, c'(z)] = \left( z^{n+1} \frac{d}{dz} + (n+1)z^n \right) c'(z) + n(n+1)z^{n-1},$$
and also, using the product rule, the fact that $(c,c) = 0$ and Lemma \ref{highestweightcommutation},
\begin{align*}
[L_n, c'(z) Y(e^{mk},z)] = & \left( z^{n+1} \frac{d}{dz} + (n+1)z^n \right) c'(z) Y(e^{mk},z) \\ &+ n(n+1)z^{n-1} Y(e^{mk},z).
\end{align*}
Now, we define
$$\mathcal{Y}(\alpha(-1)e^{mk},z) = Y(\alpha(-1)e^{mk},z) - \frac{m}{2} c'(z) Y(e^{mk},z).$$
It is then clear that
$$[L_n,\mathcal{Y}(\alpha(-1)e^{mk},z)] = \left(z^{n+1} \frac{d}{dz} + (n+1)z^n \right) \mathcal{Y}(\alpha(-1)e^{mk},z).$$
And in particular,
$$[L_m, \res_z \mathcal{Y}(\alpha(-1)e^{mk},z)] = 0.$$
Therefore, the operators
$$\mathfrak{L}_m := -\res_z \mathcal{Y}(\alpha(-1)e^{mk},z)$$
may be used to generate new physical states.  It turns out that they also have some very interesting commutation relations :
\begin{theorem}
On $\bigoplus\limits_{N \in \Z} S(\alpha-Nc)$, we have
$$[\mathfrak{L}_m,\mathfrak{L}_n] = (m-n) \mathfrak{L}_{m+n} + 2 (m^3-m) \delta_{m+n} c(0)$$
and
$$[\mathfrak{L}_m,A^a_n] = -n A^a_{m+n}$$
\end{theorem}
\begin{proof}
See \cite{gebert}, Section 3.2.
\end{proof}
As a result, the operators $\mathfrak{L}_m$'s behave as a Virasoro Lie algebra of central charge 24 (regardless of the rank of the lattice $L$) which preserve DDF states.  The result is then that given a base state $e^{\alpha}$, one can apply the operators $A^a_m$ and $\mathfrak{L}_m$ and obtain linearly independent elements in $P^1$.  However, we may also compute directly : 
$$\mathfrak{L}_{-1} e^{\alpha} \propto L_{-1} e^{\alpha - c}$$
hence by applying the operator $\mathfrak{L}_{-1}$ on any element of $P^1$ we obtain an element of $L_{-1}P^0$!  These results together with counting the dimensions of the respective spaces proves the following theorem :
\begin{theorem}
The Lie algebra $P^1 / L_{-1} P^0$ is spanned by elements of the form
$$A^{a_{i_1}}_{-m_1} ... A^{a_{i_l}}_{-m_l} \mathfrak{L}_{-j_1} ... \mathfrak{L}_{-j_k} e^{\alpha}$$
where $i_1,...,i_l \in \{1,...,d-2\}$, $m_1,...,m_l \geq 1$ and $j_1,...,j_k \geq 2$.
\end{theorem}

\subsection{Semi-infinite cohomology, continued}
So far we have shown how we may construct a Lie algebra with known root multiplicities from a vertex operator algebra on an even lattice and in particular the no-ghost theorem resulted in Corollary \ref{dimensionfakemonster}, a Lie algebra with multiplicities of $p_{24}(1 - (\alpha,\alpha)/2)$.  There are other ways to construct this Lie algebra, and in this section we describe how the problem can be approached from semi-infinite cohomology.  The vector space itself was found in \cite{feigin},\cite{frenkelgarlandzuckerman} and the Lie bracket was recovered in this context in \cite{lianzuckerman2}.  This section is also a continuation of subsection \ref{BRSTVir}.

Let $L = II_{25,1}$ and $\mf{g}$ be the Virasoro Lie algebra and $\mf{h} = \C L_0$.  Then $V_L$ is obviously a $\mf{g}$-module hence we may consider the chain complexes $C^{\frac{\infty}{2}+*}(\mf{g},\mf{h},V_L)$ and cohomology groups $H^{\frac{\infty}{2}+*}(\mf{g},\mf{h},V_L)$.  Furthermore, they inherit the $L$-grading on $V_L$, i.e.

$$H^{\frac{\infty}{2}+*}(\mf{g},\mf{h},V_L) = \bigoplus\limits_{\alpha \in L} H^{\frac{\infty}{2}+*}(\mf{g},\mf{h},V_L)[\alpha].$$
First, there is a vanishing theorem, adjusted in this setting.
\begin{theorem} (\cite{frenkelgarlandzuckerman}, Theorem 1.12)
$H^{\frac{\infty}{2}+m}(\mf{g},\mf{h},V_L) = 0$ unless $m=1$.
\end{theorem}
Then using the Euler-Poincar\'e principle, it was shown that
\begin{theorem} (\cite{frenkelgarlandzuckerman}, Theorem 2.7)
$$\dim H^{\frac{\infty}{2}+1}(\mf{g},\mf{h},V_L)[\alpha] = p_{24} \left( 1 - \frac{(\alpha,\alpha)}{2} \right).$$
\end{theorem}
Therefore we see that a priori, the dimensions of the cohomology groups match with that of $T(\alpha,c)$.  This is not a coincidence.  Indeed, given $v \in P^1$, define $T : V \to C^{\frac{\infty}{2}+1}(\mf{g},\mf{h},V_L)$ by
$$T(v) = v \otimes c(1) = v \otimes L'_{-1}$$

\begin{proposition}
For any $v \in P^1$, $dT(v) = 0$.
\end{proposition}
\begin{proof}
We use equation \ref{BRSTformula}, rewriting it as
$$d = \sum\limits_{i \in \Z} \pi(L_n) \epsilon(L'_n) + \sum\limits_{m < n} (m-n) : \iota(L_{m+n}) \epsilon(L'_m) \epsilon(L'_n):$$
It is clear that $\pi(L_n) \epsilon(L'_n) T(v) = 0$ if $n > 0$ and it is easy to verify using our definition of $\epsilon$ that $\pi(L_n) \epsilon(L'_n) T(v) = 0$ if $n < 0$.  Therefore the only non-zero term from $\pi(L_n) \epsilon(L'_n)$ is
$$\pi(L_0) \epsilon(L'_0) T(v) = v \otimes L'_0 \wedge L'_{-1}.$$

Next we examine the terms that are created from the second term.  Because of the normal ordering, the only way to obtain a non-zero expression is if all of the terms are creation operators or if $m+n = -1$.  In the first case, we can only have three creation operators if $m=n=-1$, in which case the term will be $0$ in the exterior algebra.  In the second case, if $m+n=-1$, then because $m<n$ we must have $m < 0$, but also $m \geq -1$ else $\epsilon(L'_m)$ is a annihilation operator.  In this case $m = -1$ and $n=0$ and we obtain the term $v \otimes L'_{-1} \wedge L'_{0}$ which cancels out with the first term.
\end{proof}

This inclusion is not just a simple linear map.  Indeed,
\begin{theorem}(\cite{frenkelgarlandzuckerman}, Theorem 2.8)
The map $T : P^1 \to C^{\frac{\infty}{2}+1}(\mf{g},\mf{h},V_L)$ factors to a (unitary) isomorphism
$$P^1 / rad(.,.) \simeq H^{\frac{\infty}{2}+1}(\mf{g},\mf{h},V_L).$$
\end{theorem}

Therefore, there is at least a linear isomorphism between the two, in which the dimension $26$ is of critical importance.  However, it was shown in \cite{lianzuckerman2} that this linear isomorphism is in fact a Lie algebra isomorphism.  Indeed, for two elements $u,v \in H^{\frac{\infty}{2}+*}(\mf{g},V_L)$ (note : not in relative cohomology) with homogeneous ghost numbers $|u|$ and $|v|$, they defined
$$\{u,v \} = (-1)^{|u|} (b_{-1} u)_0 v,$$
for which they proved many properties and we list a few here (\cite{lianzuckerman2}, Theorem 2.2) :
\begin{theorem}
On the cohomology $H^{\frac{\infty}{2} + *}$, we have
\begin{enumerate}[label=\alph*)]
\item $\{u,v\} = - (-1)^{(|u|-1)(|v|-1)} \{v,u\}$
\item $(-1)^{(|u|-1)(|t|-1)} \{u,\{v,t\}\} +(-1)^{(|t|-1)(|v|-1)} \{t,\{u,v\}\} $\newline$+(-1)^{(|v|-1)(|u|-1)} \{v,\{t,v\}\} = 0$
\item $\{.,.\} : H^{\frac{\infty}{2} + p} \times H^{\frac{\infty}{2} + q} \to H^{\frac{\infty}{2} + p+q-1}$
\end{enumerate}
\end{theorem}

From there, we see that $H^{\frac{\infty}{2}+1}$ is a Lie algebra.  In our case of interest, we see clearly that $b_{-1}$ cancels out the term created from $c_1$ and therefore that $T$ is indeed a Lie algebra homomorphism.  Most importantly, the results of \cite{lianzuckerman2} indicate that there is a bigger structure to study, that of a Gerstenhaber algebra.

\section{The N=1 Lie algebra of physical states}

\subsection{Fermionic construction}

There is a second version of the no-ghost theorem, for the Neveu-Schwarz algebra.  This leads to an analogous Lie algebra.  We describe it in this section.  The definition of the Fock space and its vertex operators appeared in physics in \cite{friedan}, but in this section is based on the mathematics literature from \cite{feingoldfrenkelries}, \cite{tsukada}.  Let $A$ be a complex vector space of dimension $d=2l$ and fix a basis $a_i, a_i^*$, $i=1,...,l$ of $A$.  For brevity, we will also write $a_{l+i} = a^*_i$ and $a^*_{l+i} = a_i$ for $1 \leq i \leq l$. Then define a bilinear form on $A$ as follows :
$$ (a_i,a_j) = (a_i^*, a_j^*) = 0$$
$$(a_i, a_j^*) = \delta_{i,j}$$
for all $i,j = 1,...,l$.  Then the vector space
$$\widehat{A} = A \otimes \C[t,t^{-1}]t^{\frac{1}{2}} \oplus \C c$$
is a Lie superalgebra with even part $\widehat{A}_0 = \C c$, $\widehat{A}_1 = A \otimes \C[t,t^{-1}]t^{\frac{1}{2}}$ and superbracket defined by
$$\{ a \otimes t^r, b \otimes t^s\} = (a,b) \delta_{r+s,0} c$$
$$[c,x] = 0 \text{ for all } x \in \widehat{A}$$
As before, we use the notation $a(r) = a \otimes t^r$, however this time $r \in \Z + \frac{1}{2}$.  Similarly, we write
$$a(z) = \sum\limits_{r \in \Z + \frac{1}{2}} a(r) z^{-r - \frac{1}{2}}.$$
Observe that the exponents of $z$ are still integers.
Furthermore, $\widehat{A}^{-} = A \otimes \C[t^{-1}]t^{\frac{1}{2}}$.  The Fock space in this case is given by
$$F_d = \Lambda(\widehat{A}^{-})$$
understood as the exterior algebra of $\widehat{A}^{-}$.  $V$ has a superalgebra structure given by
$$(F_d)_0 = \sum\limits_{n \in 2 \Z} \Lambda^n \widehat{A}^{-}$$
$$(F_d)_1 = \sum\limits_{n \in 2 \Z + 1} \Lambda^n \widehat{A}^{-}$$
with product being the usual product in the exterior algebra.  For a homogeneous element $v \in V$, let $\overline{v} = 0$ if $v \in (F_d)_0$ or $1$ if $v \in (F_d)_1$.  Then it is clear that for homogeneous elements $u,v \in F_d$, we have
$$u \wedge v = (-1)^{\overline{u} \overline{v}} v \wedge u.$$
There is also a symmetric bilinear form on $\widehat{A}^{-}$ defined by
$$(a(-r), b(-s)) = (a,b)\delta_{r,s}$$
and extended to all of $F_d$ by
$$(u_1 \wedge ... \wedge u_m, v_1 \wedge ... \wedge v_n ) = \delta_{m,n} \sum\limits_{\sigma \in S_m} \sgn (\sigma) (u_1, v_{\sigma(1)})...(u_m, v_{\sigma(m)}).$$

The action of $\widehat{A}$ on $F_d$ is given by
\begin{align*}
a(-r) \cdot v &= a(-r) \wedge v \\
a(r) \cdot v &= \frac{\partial}{\partial a(-r)} v \\
c \cdot v &= v
\end{align*}

for $a \in A, v \in F_d, r \in \Z_{+} + \frac{1}{2}$.
The derivation is understood as a super derivation, i.e., it follows the rules
\begin{align*}
\frac{\partial 1}{\partial a(-r)} &= 0 \\
\frac{\partial b(-r)}{\partial a(-s)} &= (a,b) \delta_{r,s} \\
\frac{\partial}{\partial a(-r)} u \wedge v &= \left( \frac{\partial}{\partial a(-r)} u \right) \wedge v + (-1)^{\overline{u}} u \wedge \left( \frac{\partial}{\partial a(-r)} v \right)
\end{align*}

We may adapt some of our previous lemmas in this settings :
\begin{lemma} \label{commutatorfermionicfields}
Suppose $a,b \in A$.  Then
$$\{ a(z_1), b(z_2) \} = (a,b) z_2^{-1} \delta(z_1/z_2).$$
\end{lemma}

\begin{lemma} \label{adjointfermion}
With respect to $(.,.)$ the adjoint of $a(z)$ is $z^{-1} a(z^{-1})$.
\end{lemma}
\begin{proof}
The proof is almost the same as the proof of Lemma \ref{adjointlemma1}.
\end{proof}

Because of the supercommutativity, we need to be more careful when defining our normal ordering.  The normal ordering on $\widehat{A}$ is given by
$$: a_1(n_1)...a_k(n_k) : = \sgn(\sigma) a_1(n_{\sigma(1)})... a_k(n_{\sigma(k)})$$
for $a_1,...,a_k \in A$, and where $\sigma$ is chosen such that $n_{\sigma(1)} \leq n_{\sigma(2)} ... \leq n_{\sigma(k)}$.  Observe that we can replace this condition with the condition that all negative $n_i$'s are on the left and all positive $n_i$'s are on the right (or in other words, that the creation operators appear on the left of the annihilation operators) and the definition does not depend on the choice of $\sigma$ satisfying either of these two conditions.  However, unlike the previous normal ordering, $\sgn(\sigma)$ plays an important role.

Define
$$\wt(a_1(-n_1) \wedge ... \wedge a_k(-n_k)) = n_1 + ... + n_k$$
which decomposes $F_d$ as
$$F_d = \sum\limits_{\substack{n \in \frac{1}{2}\Z}} (F_d)_n.$$
Furthermore, define $D : F_d \to F_d$ by
\begin{align} \label{operatorsuperD}
D \cdot 1 &= 0 \\ 
D(a(-r)) &= (r+\frac{1}{2})a(-r-1)
\end{align}
and acting as derivation on $F_d$, i.e.,
\begin{equation} \label{operatorsuperD3}
D(u \wedge v) = Du \wedge v + u \wedge Dv.
\end{equation}
Finally, define
$$Y(v,z) : F_d \to F_d[[z,z^{-1}]]$$
by
\begin{align*}
Y(1,z) &= \id \\
Y\left(a \left(-\frac{1}{2}\right) \right) &= a(z) \\
Y(u \wedge v,z) &= : Y(u,z) Y(v,z) : \\
Y(Dv,z) &= \frac{d}{dz} Y(v,z).
\end{align*}

These rules are sufficient to define $Y(.,z)$ on the entirety $F_d$.
\begin{proposition}
Suppose $b_1,...,b_k \in A$ and $n_1,...,n_k \in \Z_{\geq 0}$.  Then
\begin{align*}
Y(b_1(-n_1 - \frac{1}{2}) \wedge ... \wedge b_k(-n_k - \frac{1}{2},z) \\
= \frac{1}{n_1!...n_k!} : \left( \frac{d}{dz} \right)^{n_1} (b_1(z)) ... \left( \frac{d}{dz} \right)^{n_k} (b_k(z)):
\end{align*}
\end{proposition}

Recall that $a_{l+i} = a^*_i$ and define the distinguished element
$$\omega = \frac{1}{2} \sum\limits_{i=1}^{2l} a_i\left( -\frac{3}{2} \right) a_i^* \left( - \frac{1}{2} \right) = \frac{1}{2} \sum\limits_{i=1}^l a_i\left( -\frac{3}{2} \right) a_i^* \left( - \frac{1}{2} \right) + a_i^* \left( -\frac{3}{2} \right) a_i \left( - \frac{1}{2} \right)$$

\begin{proposition} \label{n=1structure}
The vector space $V$ with operator $Y(.,z)$ and distinguished element $\omega$ is a Vertex Operator superalgebra.
\end{proposition}
\begin{proof}
To prove Jacobi identity, observe that it follows from weak supercommutativity and weak associativity (see for example \cite{lepowsky}), both of which are shown in \cite{tsukada}.  Here we verify that the conformal element $\omega$ generates the Virasoro Lie algebra.
Observe that
$$Y(\omega,z) = \frac{1}{2} \sum\limits_{i=1}^{2l} : \frac{d}{dz} Y\left(a_i \left( - \frac{3}{2} \right),z\right)  Y\left( a_i\left(- \frac{1}{2}\right),z\right) :$$
Write $Y(\omega,z) = \sum\limits_{n \in \Z} L_{n} z^{-n-2}$.  First observe that
$$\frac{d}{dz} a(z) = \sum\limits_{n \in \Z+\frac{1}{2}} (-n - \frac{1}{2}) a(n) z^{-n - \frac{3}{2}}.$$
Therefore,
$$L_k = \frac{1}{2} \sum\limits_{i=1}^{2l} \sum\limits_{n \in \Z + \frac{1}{2}} (n-k - \frac{1}{2}) : a_i(k-n) a_i^*(n): $$
In particular,
$$L_0 = \sum\limits_{i=1}^{2l} \sum\limits_{n \in \Z_{\geq 0} + \frac{1}{2}} n a_i(-n) a_i^*(n)$$
and it is clear from this formula that $L_0$ acts as the degree operator on $F_d$.  Similarly,
$$L_{-1} = \sum\limits_{i=1}^{2l} \sum\limits_{n \in \Z_{\geq 0} + \frac{1}{2}} (n + \frac{1}{2}) a_i(-n - 1) a_i^*(n)$$
and it is clear from this formula that $L_{-1}$ acts on $F_d$ as the operator $D$ defined by equations (\ref{operatorsuperD})-(\ref{operatorsuperD3}).

The proof that the components of $\omega$ satisfy the relations of the Virasoro Lie algebra may be done as in Theorem \ref{latticevertextheorem}.  We must check the relations
\begin{align*}
\omega_0 \omega &= L_{-1} \omega \\
\omega_1 \omega &= L_0 \omega = 2 \omega \\
\omega_2 \omega &= L_1 \omega = 0 \\
\omega_3 \omega &= L_2 \omega = \frac{l}{2} 1\\
\omega_n \omega &= L_n \omega = 0 \text{ for } n \geq 4.
\end{align*}
We have just done the first two of these.  To verify that $L_2 \omega = \frac{l}{2} 1$, consider
$$L_{2} = \frac{1}{2} \sum\limits_{i=1}^{2l} \sum\limits_{n \in \Z} (n - 5/2)  a_i(-n + 2) a_i^*(n).$$
When applied to $\omega$, the only nontrivial terms appear from $n = \frac{1}{2}$ and $n = \frac{3}{2}$.  Of course, if $n$ is one of these two numbers then $2-n$ is the other one.  Therefore, we consider the action of the terms
$$-\frac{1}{2} a_i\left(\frac{3}{2}\right) a^*_i\left(\frac{1}{2}\right)$$ on $\omega$, for $i=1,...,2l$.  Because these terms are super derivations, the result will be $\frac{1}{4}$ which results in $L_2 \omega = \frac{l}{2} 1$ when summing over all $2l$ of them.  The other conditions are easy to check.
\end{proof}
We have found in Proposition \ref{voadimension} the dimensions of the graded spaces $(V_L)_{\alpha}$ of a given weight.  We can accomplish something similar for $F_d$.

\begin{proposition} \label{svoadimension}
Write $c(n)$ as the dimension of $(F_d)_n$, the space of elements of $F_d$ of weight $n$.  Then $c(n)$ is the coefficient of $q^n$ in the expansion of
$$\prod\limits_{i \in \Z_{\geq 0} + \frac{1}{2}}(1+q^i)^{2l},$$
\end{proposition}
\begin{proof}
It is clear that a basis for $(F_d)_n$ is given by
$$a_{i_1}(-j_1) \wedge ... \wedge a_{i_k}(-j_k)$$
such that $j_1 + ... j_k = n$ and where $j_k \leq j_{k-1} \leq ... \leq j_1$ and $j_{r} = j_{r+1}$ implies that $a_{i_r} < a_{i_k}$.  Equivalently, these basis elements can be described as partitions of $n$ into half-integers of the form $\frac{1}{2} + \Z_{\geq 0}$ with $2l$ colors, and the number of such partitions is evidently obtained from the desired generating function.
\end{proof}

\subsection{Construction of N = 1 and N=2 SVOAs}

We will now combine the construction of the vertex operator algebra $V_L$ from an even lattice $L$ of rank $d \in 2 \Z$ with the construction of the vertex operator superalgebra $F_d$ constructed from $A$ spanned by the basis $a_1,...,a_l,a^*_1,...,a^*_l$.  The objective of this section is to construct a $N=1$ SVOA based on the lattice $L$.  However, this construction also results in an $N=2$ SVOA with some minor modifications.  Therefore, we introduce these minor modifications right away because we will need the $N=2$ construction later.

Indeed, we let
$$V_{NS} = V_L \otimes F_d$$
where the tensor product is understood as in \cite{lepowsky}, Proposition 3.12.5.  Observe that $V_{NS}$ is again a vertex operator superalgebra with even part
$$(V_{NS})_0 = V_L \otimes (F_d)_0$$
odd part
$$(V_{NS})_1 = V_L \otimes (F_d)_1,$$
and conformal element the sum
$$\omega = \frac{1}{2} \sum\limits_{i=1}^{d} h_i(-1)^2 + \frac{1}{2} \sum\limits_{i=1}^{2l} a_i\left( -\frac{3}{2} \right) a_i^* \left( - \frac{1}{2} \right)$$
where the $h_i$'s form an orthogonal basis for $L \otimes \C$.  The Virasoro algebra generated by $\omega$ will then have central charge equal to the sum of two original charges, i.e. $d + \frac{d}{2} = \frac{3d}{2}$.  We now let
$$h_i^+ = \sum\limits_{i=1}^l \frac{1}{\sqrt{2}} (h_i + i h_{l+i})$$
and
$$h_i^- = \sum\limits_{i=1}^l \frac{1}{\sqrt{2}} (h_i - i h_{l+i}).$$
We can easily check that $(h_i^+,h_i^+) = (h_i^-, h_i^-) = 0$ and 
$(h_i^+,h_j^-) = \delta_{i,j}$.  Observe that if the lattice $L$ has signature $(l,l)$ then the orthogonal elements $h_i$'s can be chosen so that $h_i^+,h_i^-$ lie in $L \otimes \R$.  Then, we define
$$\tau^+ = \sum\limits_{i=1}^l h_i^+(-1) a_i (-\frac{1}{2}),$$
$$\tau^- = \sum\limits_{i=1}^l h_i^-(-1) a^*_i (-\frac{1}{2}),$$
$$j = \sum\limits_{i=1}^l a_i (-\frac{1}{2}) \wedge a_i^* (-\frac{1}{2}).$$

resulting in the following proposition (see \cite{qiu}, \cite{kazama}) :

\begin{proposition}
$V_{NS}$ is an $N=2$ SVOA with distinguished elements $\omega, \tau^+, \tau^-, j$.  Therefore, it is also an $N=1$ SVOA with distinguished elements $\omega, \tau = \tau^+ + \tau^-$.
\end{proposition}
\begin{proof}
Define operators $G^{\pm}_{r}$ by
$$Y(\tau^{\pm},z) = \sum\limits_{n \in \Z} \tau_n z^{-n-1} = \sum\limits_{n \in \Z} G^{\pm}_{n + \frac{1}{2}} z^{-n-2}$$
and operators $J_n$ by
$$Y(j,z) = \sum\limits_{n \in \Z} j_n z^{-n-1} = \sum\limits_{n \in \Z} J_n z^{-n-1}.$$
Observe that $j_n = J_n$.  Then we may evaluate,
$$G^+_{n + \frac{1}{2}} = \sum\limits_{i=1}^l \sum\limits_{k \in \Z} h^+_i(k) a_i(n + \frac{1}{2} - k),$$
$$G^-_{n + \frac{1}{2}} = \sum\limits_{i=1}^l \sum\limits_{k \in \Z} h^-_i(k) a^*_i(n + \frac{1}{2} - k)$$
and
$$j_n = \sum\limits_{i=1}^l \sum\limits_{k \in \Z} : a_i (k+ \frac{1}{2}) \wedge a_i^* (n-k- \frac{1}{2}) :$$

From there, it is easy to compute the following table, which can also be found in \cite{scheit} :
\begin{align*}
\omega_0 j &= L_{-1} j & \tau^+_n \tau^+ &= \tau^-_n \tau^- = 0 \text{ for } n \geq 0 \\
\omega_1 j &= L_0 j = j & \tau^+_0 \tau^-&= G^+_{-\frac{1}{2}} \tau^- = \omega + \frac{1}{2} L_{-1} j\\
\omega_n j &= 0 \text{ for } n \geq 2 & \tau^+_1 \tau^- &= G^+_{\frac{1}{2}} \tau^- =  j \\
\omega_0 \tau^{\pm} &= L_{-1} \tau^{\pm} & \tau^+_2 \tau^- &= l
\end{align*}
\begin{align*}
\omega_1 \tau^{\pm} &= L_0 \tau^{\pm} = \frac{3}{2} \tau^{\pm} & \tau^+_n \tau^- &= 0 \text{ for } n \geq 0 \\
\omega_n \tau^{\pm} &= 0 \text{ for } n \geq 2 & j_0 \tau^{\pm} &= \pm \tau^{\pm} \\
j_0 j &= 0 & j_n \tau^{\pm} &= 0 \text{ for } n \geq 0 \\
j_1 j &= l & j_n j &= 0 \text{ for } n \geq 2.
\end{align*}
Using the techniques introduced in Theorem \ref{latticevertextheorem} and Proposition \ref{n=1structure}, we can recover the commutation relations for the $N=2$ superconformal algebra.  Because the process is fairly repetitive, we only recover $\{G^+_r, G^-_s \}$ here, since it is the most difficult one.

Using the supercommutator formula, we compute
\begin{align*}
& \{ Y(\tau^+,z_1), Y(\tau^-,z_2)\} = \res_{z_0} z_2^{-1} \left(e^{- z_0 \frac{\partial}{\partial z_1}} \delta(z_1/z_2) \right) Y(Y(\tau^+,z_0)\tau^-,z_2) \\
&= \res_{z_0} z_2^{-1} \left(e^{- z_0 \frac{\partial}{\partial z_1}} \delta(z_1/z_2) \right) ( Y(\omega + \frac{1}{2} L_{-1} j,z_2)z_0^{-1} + Y(j,z_2)z_0^{-2} + l z_0^{-3}Y(1,z_2)) \\
&= Y(\omega,z_2) z_2^{-1} \delta(z_1/z_2) + \frac{z_2^{-1}}{2} \delta(z_1/z_2)  \frac{\partial}{\partial z_2} Y(j,z_2) \\
&- z_2^{-1} \frac{\partial}{\partial z_1} \delta(z_1/z_2) Y(j,z_2) + \frac{l}{2} z_2^{-1} \left( \frac{\partial}{\partial z_1} \right)^2 \delta(z_1/z_2)
\end{align*}
Taking the coefficients of $z_1^{-m-2}$ on both sides, we find
\begin{align*}
\{ G^+_{m+\frac{1}{2}}, Y(\tau^-,z_2) \} &= Y(\omega,z_2) z_2^{m+1} + \frac{z_2^{m+1}}{2} \frac{\partial}{\partial z_2} Y(j,z_2) \\ &+ (m+1) z_2^m Y(j,z_2) + \frac{l}{2} m(m+1) z_2^{m-1}.
\end{align*}
Then, taking the coefficients of $z_2^{-n-2}$ on both sides, we find
\begin{align*}
\{G^+_{m+\frac{1}{2}}, G^-_{n + \frac{1}{2}} \} &= L_{m+n+1} + \frac{-m-n-2}{2} J_{m+n+1} \\ &+ (m+1) J_{m+n+1} + \frac{l}{2} m(m+1) \delta_{m+n+1,0} \\
&= L_{m+n+1} + \frac{1}{2} (m-n) J_{m+n+1} + \frac{l}{2} m(m+1) \delta_{m+n+1,0}.
\end{align*}
Letting $r = m+\frac{1}{2}$ and $s = n + \frac{1}{2}$ and observing that $l = \frac{c}{3}$ we then have
$$\{G^+_{r}, G^-_{s} \} = L_{r+s} + \frac{1}{2} (r-s) J_{r+s} + \frac{c}{6} (r^2 - \frac{1}{4}) \delta_{r+s,0},$$
as desired.
\end{proof}

\subsection{No-ghost theorem in the Neveu-Schwarz model}
For the remainder of this section we will only be interested in the $N=1$ structure of $V_{NS}$.  Define for $i=1,...,l$,
$$\tilde{h}_i = \frac{1}{\sqrt{2}} (a_i + a_i^*)$$
and for $i=l+1,...,d$,
$$\tilde{h}_i = \frac{i}{\sqrt{2}} (a_i - a_i^*).$$
It is clear that $(\tilde{h}_i,\tilde{h}_j) = \delta_{i,j}$.
Given $x = \sum\limits_{i=1}^d x_i h_i$ we define $\tilde{x} = \sum\limits_{i=1}^d x_i \tilde{h}_i$.  The following is a generalization of the commutator $[L_k,\alpha(n)] = -n \alpha(k+n)$.

\begin{lemma} \label{bosonfermionbridge}
Suppose $n \in \Z_{\geq 0}$, $r,s \in \frac{1}{2} + \Z$ and $x \in L$.  Then $$[G_{r},x(-n)] = n \tilde{x}(n + r).$$  Furthermore, $$[G_{r},\tilde{x}(s)] = x(r+s).$$
\end{lemma}

To state the no-ghost theorem in this setting, we first modify the definition of $P^i$ from section \ref{subsectionphysicalstate}, replacing lowest weight vectors for Virasoro with lowest weight vectors for super Virasoro, i.e.
$$P^i = \{ v \in V_{NS} | L_0v = iv, G_{n - \frac{1}{2}}v = 0 \text{ for all } n > 0\}.$$
Observe that these conditions also imply that $L_n v = 0$ for all $n > 0$.

Again fix $\alpha \neq 0 \in L$.  Then there exists $c \in L \otimes \R$ such that $(c,\alpha) = 1$ and $(c,c) = 0$ by Lemma \ref{isotropicc}.  Define again
$$S(\alpha) = \{v \otimes e^{\alpha} \in V_{NS} \}$$
$$V_s(\alpha) = \{ u=v \otimes e^{\alpha} \in V_{NS},L_0 u = su \}$$ and
$$P^i(\alpha) = \{ v \otimes e^{\alpha} \in P^i \}.$$
Then we replace $T(\alpha,c)$ with
$$T(\alpha,c) = \{ v \in V_{NS}(\alpha) | G_{-\frac{1}{2}+n } v = c(n)v = \tilde{c}(n - \frac{1}{2})v = 0 \text { for all } n > 0 \},$$
where the definition is motivated by Corollary \ref{transervedefinitionalternative}.  We generalize Lemma \ref{independentlemma} in the following way :
\begin{lemma}
Suppose $\lambda_1,...,\lambda_n,\mu_1,...,\mu_m \in \Z_{\geq 0}$ and $\epsilon_1,...,\epsilon_k$, $\delta_1,...,\delta_l \in \{ 0, 1 \}$.  Then for different values of $\lambda_i,\mu_i,\epsilon_i,\delta_i$, the spaces
$$G_{- \frac{1}{2}}^{\epsilon_1}...G_{-k + \frac{1}{2}}^{\epsilon_k}L_{-1}^{\lambda_1}...L_{-n}^{\lambda_n} \tilde{c} \left(- \frac{1}{2}\right)^{\delta_1}...\tilde{c} \left(-l + \frac{1}{2}\right)^{\delta_l} c_{-1}^{\mu_1}...c_{-m}^{\mu_m} T(\alpha,c)$$
are independent and their direct sum is the entirety of $S(\alpha)$.
\end{lemma}

\begin{proof}
Choose linearly independent vectors $x_1,...,x_{d-2}$ in $L \otimes \R$ which are orthogonal to $\alpha,c$ and let $y_i = \tilde{x_i}$.  Then every element of $T(\alpha,c)$ can be written as a linear combination of linearly independent elements of the form
$$T = \prod\limits_{i=1}^{d-2} x_i(-1)^{a^1_i}...x_i(-k_i)^{a^{k_i}_i} y_i( - \frac{1}{2})^{b^1_i}...y_i(-l_i + \frac{1}{2})^{b^{l_i}_i} e^{\alpha}.$$
with $a_i^j \in \Z$ and $b_i^j \in \{0,1 \}$.  Observe that these monomials do not contain any $\alpha,\tilde{\alpha},c,\tilde{c}$ and therefore we may add some terms of the form $\alpha(-i),\tilde{\alpha}(-\frac{1}{2}-i)$ or $c(-i), \tilde{c}(-\frac{1}{2} -i)$ and we will obtain again linearly independent monomials.  

But then, observe that $L_{-i}^{\lambda_i} e^{\alpha}$ generates a monomial of the form $\alpha(-i) e^{\alpha}$.  Furthermore, $G_{-i + \frac{1}{2}} e^{\alpha}$ generates a monomial of the form $\tilde{\alpha}(-i + \frac{1}{2}) e^{\alpha}$.  Therefore, the choices of values $\lambda_i,\mu_i,\epsilon_i,\delta_i$ generate (among other terms) when applied to $T$, a monomial of the form
\begin{align*}
&\tilde{\alpha}(-\frac{1}{2})^{\epsilon_1} ... \tilde{\alpha}(-k + \frac{1}{2})^{\epsilon_k} \alpha(-1)^{\lambda_1}... \alpha(-n)^{\lambda_n} \cdot \\ &\cdot \tilde{c} \left(- \frac{1}{2}\right)^{\delta_1}... \tilde{c} \left(-l + \frac{1}{2}\right)^{\delta_l} c_{-1}^{\mu_1}...c_{-m}^{\mu_m} T
\end{align*}

From there it is enough to order the $\lambda_i,\mu_i,\epsilon_i,\delta_i,a_i^j,b_i^j$ with some lexicographic ordering and apply a maximality argument to see why they must yield linearly independent elements.

The second part of the statement follows by using linear independence of the monomials $T$ previously described and counting the dimension of all generated monomials.  Equivalently, one can write an isomorphism of vector spaces sending $\alpha(-i)$ to $L_{-i}$ and $\tilde{\alpha}(-l+\frac{1}{2})$ to $G_{-l + \frac{1}{2}}$, and this linear map is obviously injective (by linear independence) and surjective by definition.
\end{proof}

\begin{lemma}
For fixed $\alpha,c$, let $N^s$ denote all elements of $V_s(\alpha)$ such that any of the $\lambda_i,\epsilon_i$ is non-zero.  Then the operators $G_{\frac{1}{2}}$ and $G_{\frac{3}{2}} + 2 L_1 G_{\frac{1}{2}}$ are well-defined maps $N^{\frac{1}{2}} \to N^0$ and $N^{\frac{1}{2}} \to N^{-1}$ if and only if $d = 10$.
\end{lemma}
\begin{proof}
Write $y = G_{-n}x$ such that $x \in V_L$ and $L_0 y = \frac{1}{2}y$.  Then we can write
$$y = G_{-\frac{1}{2}} x_1 + G_{- \frac{3}{2}} x_2$$
with $L_0 x_1 = 0$ and $L_0 x_2 = -x_2$.  Once again, this decomposition is justified because $G_{-\frac{1}{2}}$ and $G_{- \frac{3}{2}}$ generate the lower part of the super Virasoro algebra.  Therefore, we compute
\begin{align*}
G_{\frac{1}{2}}(G_{-\frac{1}{2}} x_1) &= 2L_0 x_1 - G_{-\frac{1}{2}}G_{\frac{1}{2}} x_1 \\
&= - G_{-\frac{1}{2}}G_{\frac{1}{2}} x_1 \in N^0 \\
G_{\frac{3}{2}}G_{-\frac{1}{2}} x_1 &= 2L_1 x_1 - G_{-\frac{1}{2}}G_{\frac{3}{2}}x_1 \\
L_1 G_{\frac{1}{2}} G_{-\frac{1}{2}} x_1 &= - L_1 G_{-\frac{1}{2}}G_{\frac{1}{2}}x_1 \\
&= G_{-\frac{1}{2}}(...) - G_{\frac{1}{2}}G_{\frac{1}{2}}x_1
\end{align*}
and we see that $(G_{\frac{3}{2}} + 2 L_1 G_{\frac{1}{2}})(G_{-\frac{1}{2}}x_1) \in N^0$ because $2 G_{\frac{1}{2}}G_{\frac{1}{2}} = 2L_1$.  We also compute
\begin{align*}
G_{\frac{1}{2}}G_{-\frac{3}{2}}x_2 &= 2 L_{-1} x_2 - G_{-\frac{3}{2}}G_{\frac{1}{2}}x_2 \in N^0 \\
G_{\frac{3}{2}}G_{-\frac{3}{2}}x_2 &= 2 L_0 x_2 + \frac{2c}{3} x_2 - G_{-\frac{3}{2}} G_{\frac{3}{2}} x_2 \\
&= G_{-\frac{3}{2}}(...) + \frac{2c}{3} x_2 - 2 x_2 \\
L_1 G_{\frac{1}{2}} G_{-\frac{3}{2}} x_2 &= 2 L_1 L_{-1} x_2 - L_1 G_{-\frac{3}{2}} G_{\frac{1}{2}} x_2 \\
&= 4 L_0 x_2 + L_{-1}(...) + G_{-\frac{1}{2}}(...) + G_{-\frac{3}{2}}(...)
\end{align*}
From these computations, we see that $(G_{\frac{3}{2}} + 2 L_1 G_{\frac{1}{2}})(G_{-\frac{3}{2}}x_2) \in N^{-1}$ if and only if $\frac{2c}{3} - 10 = 0$ or equivalently, $c = 15$.
\end{proof}

\begin{theorem} (No-ghost theorem)
Suppose $u \in P^{1/2}(\alpha)$ and $L$ is a Lorentzian even lattice of rank $10$.  Then $u = v + w$ where $v \in T(\alpha,c)$ and $w \in P^{1/2}$ is a linear combination of terms $G_{-n}x$ for some $x \in V_{NS}$.  
\end{theorem}

\begin{proof} 
The proof is now very similar to the proof in the non-super case, with only a few additions.  Replace $K^s$ with the subspace of $V_{s}(\alpha)$ composed of elements of the form
$$A = \sum\limits_{i \in I}a_i  \tilde{c} \left(- \frac{1}{2}\right)^{\delta_1}...\tilde{c} \left(-l + \frac{1}{2}\right)^{\delta_l} c(-1)^{\mu^i_1}...c(-m)^{\mu_m^i} B_i$$
for some $B_i \in T(\alpha,c)$ and the condition $L_0(A) = sA$.  Using Lemma \ref{bosonfermionbridge}, $G_{\frac{1}{2}}$ and $G_{\frac{3}{2}}$ will be maps $K^s \to K^{s - \frac{1}{2}}$ and $K^s \to K^{s - \frac{3}{2}}$, respectively.

The remainder of the proof proceeds basically the same as in Theorem \ref{noghost}.  We observe that $v,w \in P^{1/2}$ as well by independence of the spaces.  Then one remarks using $(c,\alpha) = 1$ that $A$ can only be in $P^{1/2}$ if all the exponents $\delta_1,...,\delta_l$ and $\mu_1,...,\mu_m$ are zero.
\end{proof}

\subsection{The Lie algebra of physical states}
The Lie algebra in this subsection was first described in \cite{borcherds0} and then studied further in \cite{scheit2}.  It is described here in a slightly different way.

Having defined a new space $P^i$ in the $N=1$ model, one would be tempted to define a Lie bracket as we have before, by
$$[u,v] = u_0v$$
for $u,v \in P^{1/2}$.  However, one should observe that the weight of our elements has changed from $1$ to $1/2$.  Furthermore, the space $P^i$ is defined by extra conditions : indeed, elements must now be lowest weight vectors for the entirety of the super-Virasoro algebra.  Because of these issues, this bracket is not in general closed, or even well-defined on the quotient $P^{1/2}/G_{- \frac{1}{2}}P^0$ anymore.  To resolve these issues, we replace the old Lie bracket by a new Lie bracket :
$$[u,v] = (G_{-\frac{1}{2}}u)_0 v.$$

\begin{theorem}
The space $P^{1/2}/G_{- \frac{1}{2}}P^0$ with the bracket defined above is a Lie algebra.
\end{theorem}

\begin{proof}
We must prove that the bracket is well-defined.  Suppose $u,v \in P^{1/2}$.  Lemma \ref{highestweightcommutation} does not have an obvious generalization, because the commutator formula between $Y(\tau,z_1)$ and $Y(u,z)$ yields
\begin{align*}
\{ Y(\tau,z_1), Y(u,z) \} &= z^{-1} Y(G_{- \frac{1}{2}}u,z) \delta(z_1/z) - z^{-1} Y(G_{\frac{1}{2}}u,z) \frac{\partial}{\partial z_1} \delta(z_1/z) \\
&+ \res_{z_0} z^{-1} \sum\limits_{n > 0} Y(G_{n + \frac{1}{2}} v,z) z_0^{-n-2} \left( e^{-z_0 \frac{\partial}{\partial z_1}} \delta(z_1/z) \right).
\end{align*}
and while all the terms except for the first term disappear, there is nothing we can do with said first term.  Instead, if we replace $u$ with $G_{- \frac{1}{2}}u$, an even element, we see that only the first two terms remain, and we obtain
$$[ Y(\tau,z_1), Y(G_{- \frac{1}{2}}u,z) ] = z^{-1}Y(G_{- \frac{1}{2}}^2u,z) \delta(z_1/z) - z^{-1}Y(G_{\frac{1}{2}}G_{- \frac{1}{2}}u,z) \frac{\partial}{\partial z} \delta(z_1/z)$$
which simplifies to
$$[ Y(\tau,z_1), Y( G_{- \frac{1}{2}} u,z) ] = z^{-1}Y(L_{-1}u,z) \delta(z_1/z) - z^{-1}Y(2 L_0 u,z) \frac{\partial}{\partial z} \delta(z_1/z).$$
From there, the proof of Lemma \ref{highestweightcommutation} can be repeated, and we obtain that for $u \in P^{i}$,
$$[ G_{m + \frac{1}{2}},  (G_{- \frac{1}{2}}u)_0 ] = ((2i)(m+1) - m - 1)u_m.$$
In particular, for $i = \frac{1}{2}$, we obtain
$$[ G_{m + \frac{1}{2}}, (G_{- \frac{1}{2}}u)_0 ] = 0.$$
This implies that $[ u,v ] \in P^{1/2}$.  Next, suppose $u \in G_{- \frac{1}{2}}P^0$.  Then write $u = G_{- \frac{1}{2}} u'$ and then
$$[ u,v ] = (G_{- \frac{1}{2}}^2 u')_0 v = (L_{-1}u')_0 v = 0.$$
again because $(L_{-1}u')_0 = \res_{z} \frac{d}{dz} Y(u',z) = 0$.

Therefore, we see that the bracket is well-defined.  It remains to prove anticommutativity and the Jacobi identity.

For anticommutativity we once again use skew-symmetry, which results in
$$Y(G_{- \frac{1}{2}}u,z)v = e^{z L_{-1}} Y(v,-z) G_{- \frac{1}{2}} u.$$
Recall that for $N=1$ SVOA's, we have the property
$$\{G_{-\frac{1}{2}}, Y(a,z) \} = Y(G_{- \frac{1}{2}}a,z),$$
which allows us to rewrite
$$Y(G_{- \frac{1}{2}}u,z)v = e^{z L_{-1}} Y(G_{-\frac{1}{2}}v,-z)u - e^{z L_{-1}}G_{-\frac{1}{2}} Y(v,-z)u.$$
Taking residue at $z = 0$ we see that
$$[u,v] = (G_{-\frac{1}{2}}u)_0 v = -[v,u] + G_{-\frac{1}{2}}x$$
for some $x$ (here we also use the fact that $L_{-1} = G_{-\frac{1}{2}}^2$).  The proof that $x \in P^0$ is essentially the same as in the proof of Theorem \ref{liealgebraofphysicalstates}.

Finally, we must prove Jacobi identity, i.e.
\begin{equation} \label{jacobiidentity}
[u,[v,w]] - [v,[u,w]] = [[u,v],w]
\end{equation}
for all $u,v,w \in P^{1/2}$.
Once again, we have by Corollary \ref{commutatorcorollary} that
$$[u'_0,v'_0] = (u'_0 v')_0$$
Here we use $u' = G_{-\frac{1}{2}}u$ and $v' = G_{-\frac{1}{2}}v$ to find
\begin{align*}
[u,[v,w]] - [v,[u,w]] &= ((G_{-\frac{1}{2}} u)_0 G_{-\frac{1}{2}}v)_0 w \\
&= (G_{-\frac{1}{2}} (G_{-\frac{1}{2}} u)_0 v)w + ((G_{-\frac{1}{2}}^2u)_0 v)_0 w \\ &= (G_{-\frac{1}{2}} (G_{-\frac{1}{2}} u)_0 v)w + ((L_{-1}u)_0 v)_0 w
\end{align*}
however $(L_{-1}u)_0 = 0$ hence
$$[u,[v,w]]-[v,[u,w]] = [[u,v],w]$$
as desired.
\end{proof}

We have defined a bilinear form $V_L$ which we now denote $(.,.)_L$ and a bilinear form on $F_d$ which we now denote $(.,.)_F$.  Together they define a bilinear form on $V_{NS}$ by
$$(x_1 \otimes y_1, x_2 \otimes y_2) = (x_1,x_2)_L (y_1,y_2)_F$$
where $x_1,x_2 \in V_L$ and $y_1,y_2 \in F_d$.

We also extend $\theta$ to all of $V_{NS}$ by specifying $\theta(a(m)) = -a(m)$.  Then we may repeat the steps taken in Section \ref{subsectionphysicalstate} to obtain
\begin{theorem}
Suppose $u,v,w \in V_{NS}$.  Then we have
$$(Y(v,z)u,w) = (u, Y(e^{z L_1}(-1)^{\overline{L_0}} z^{-2L_0} \theta(v),z^{-1})w).$$
\end{theorem}
\begin{proof}
There is no substantial difference between the proof of this theorem and the analogous proofs in Section \ref{subsectionbilinearform}.  The only additional thing to check is that, in the dual module $(V_{NS})'$, we have for $a \in A$,
\begin{align*}
\langle Y'_{\theta}(a ( - \frac{1}{2} ) ,z)u,w \rangle &= \langle u, Y(e^{z L_1}(-1)^{\overline{L_0}} z^{-2L_0} \theta(a ( -\frac{1}{2})),z^{-1})w \rangle \\ &= \langle u, z^{-1} a(z^{-1})w \rangle.
\end{align*}
and then use Lemma \ref{adjointfermion}.
\end{proof}

\begin{lemma}
With respect to $(.,.)$, the adjoint of operators $G_{r}$ for $r \in \Z + \frac{1}{2}$ on $V_{NS}$ is $G_{-r}$.
\end{lemma}
\begin{proof}
See the proof of Lemma \ref{adjointvirasoro}, replacing only the adjoint of $h_i(k)$ with the adjoint of $a_i(k)$.
\end{proof}

Most importantly, the following analogue of Corollary \ref{adjointcorollary}.
\begin{corollary}
Suppose $v \in P^{1/2}$.  Then $(G_{-\frac{1}{2}}v)^*_0 = -(G_{- \frac{1}{2}}\theta(v))_0$.
\end{corollary}

Then it is clear that once again the radical of the bilinear form $(.,.)$ is an ideal.  In particular, we now have an analogous theorem (\cite{scheit2}) :
\begin{theorem} \label{multiplicitiesN=1}
The quotient $g_{NS}$ of $P^{1/2}$ by the radical of the bilinear form $(.,.)$ is a Lie algebra.  Furthermore, the multiplicities of $(g_{NS})_{\alpha}$ are given by the coefficient $c(\frac{1}{2} - \frac{(\alpha,\alpha)}{2})$ in $$\sum\limits_{r \in \frac{1}{2} \Z} c(r)q^r = \frac{\prod\limits_{i \in \Z_{\geq 0} + \frac{1}{2}}(1+q^i)^{8}}{\prod\limits_{i = 1}^\infty (1-q^i)^8}.$$
\end{theorem}
This theorem was proven by counting the dimensions of $T(\alpha,c) \cap P^{1/2}$ and by manipulation of the characters for $V_{NS}$.

\begin{remark}
Observe that if $L$ is even, the numbers $\frac{1}{2} - \frac{(\alpha,\alpha)}{2}$ are always an integer plus $1/2$.  Therefore we do not need the full numerator, indeed we may replace it with
$$\frac{1}{2} \left(\prod\limits_{i \in \Z_{\geq 0} + \frac{1}{2}} (1+q^i)^{8} - \prod\limits_{i \in \Z_{\geq 0} + \frac{1}{2}} (1-q^i)^{8} \right) = 8 q^{1/2} \prod\limits_{i = 1}^\infty (1+q^i)^8,$$
using the \emph{aequatio identica satis abstrusa} of Jacobi, used here as in \cite{gliozziolivescherk}.  Inserting this in the numerator we then obtain
$$8 q^{1/2} \frac{\prod\limits_{i = 1}^\infty (1+q^i)^{8}}{\prod\limits_{i = 1}^\infty (1-q^i)^8}.$$
\end{remark}

\subsection{DDF construction in the N=1 model} \label{subsectionDDFN=1}
We have successfully recovered the no-ghost theorem in the $N=1$ Neveu-Schwarz model, however we have yet to exhibit an explicit basis for the states that remain as a result of the quotient of $P^{1/2}$ by the radical of the bilinear form $(.,.)$.  In this section, we describe how to construct the transverse DDF operators in the $N=1$ model as introduced in \cite{browerfriedman}, but in the mathematical framework of Vertex Operator Algebras.

Fix $\alpha \in L$ and again choose $c \in L \otimes \R$ such that $(c,\alpha) = 1$ and $(c,c) = 0$.  Then assume $a \in L \otimes \R$ satisfies $(a,\alpha) = (a,c) = 0$.  We may consider
\begin{align*}
A^{a}_m &= \res_z Y(G_{- \frac{1}{2}} \tilde{a} (-1/2) e^{mc} ,z) \\ &= \res_z Y(a(-1) - m \tilde{a}(-1/2) \wedge \tilde{c}(-1/2) e^{mc},z)
\end{align*}
as an operator $S(\alpha + Nc) \to S(\alpha + (N+m)c)$.  By construction, it is easy to check that
$$[G_{r},A^a_m] = 0$$
because $\tilde{a}(-1/2)e^{mc} \in P^{1/2}$.  We also have :
\begin{proposition}
Suppose $a_1,a_2$ satisfy $(a_i,\alpha) = (a_i,c) = 0$.  Then,
$$[A^{a_1}_{m_1},A^{a_2}_{m_2}] = m_1 (a_1,a_2) \delta_{m_1+m_2} c(0)$$
\end{proposition}
\begin{proof}
We know from Proposition \ref{ddfproperties} that
$$[\res_{z_1} Y(a_1(-1)e^{m_1 c},z_1), \res_{z_2} Y(a_2(-1) e^{m_2 c},z_2)] = m_1(a_1,a_2) \delta_{m_1+m_2} c(0),$$
so we need only prove
$$[\res_{z_1} Y(a_1(-1)e^{m_1 c},z_1), \res_{z_2} Y(-m_2 \tilde{a}_2(-1/2) \wedge \tilde{c} (-1/2) e^{m_2 c},z_2)] = 0$$
which is easy, and that the commutator of
$$\res_{z_1} Y(-m_1 \tilde{a}_1(-1/2) \wedge \tilde{c} (-1/2) e^{m_1 c},z_2))$$
with
$$\res_{z_2} Y(-m_2 \tilde{a}_2(-1/2) \wedge \tilde{c} (-1/2) e^{m_2 c},z_2)]$$
is always zero.  Following the proof of Proposition \ref{ddfproperties}, and using Lemma \ref{commutatorfermionicfields} we end up with having to evaluate
$$m_1 m_2 \res_{z_1} \res_{z_2} (a_1,a_2) z_2^{-1} \delta(z_1/z_2) \tilde{c}(z_1) \wedge \tilde{c} (z_2) Y(e^{m_1 c},z_1) Y(e^{m_2 c},z_2)$$
which creates a term of the form $\tilde{c}(z) \wedge \tilde{c}(z) = 0$ by an easy application of Lemma \ref{commutatorfermionicfields}.
\end{proof}

Clearly these operators are not sufficient for the entire spectrum-generating algebra since they only add up to $8$ "bosons".  We need to consider other operators to find the remaining 8 "fermions", which will require explanation.  Define
$$B^{a}_r = \res_z Y(G_{-\frac{1}{2}} \tilde{c}(-1/2) \wedge \tilde{a}(-1/2) c(-1)^{-1/2} e^{rc} ,z),$$
for $r \in \frac{1}{2} + \Z$.
The exponent $-1/2$ on $c(-1)$ is peculiar.  In section \ref{sectionlongitudinal}, we defined an operator
$$\log (z c(z)),$$
hence we now define
$$(z c(z))^{-1/2} = \exp \left( \frac{1}{2} \log(z c(z)) \right)$$
and then, in turn,
$$c(z)^{-1/2} = z^{1/2} (zc(z))^{-1/2}.$$

Because the element $c(-1)^{-1/2}$ does not live in $S(\widehat{h}^-)$, let us consider a larger vector space $R(\widehat{h}^-)$ generated by symmetric polynomials in $$h(-1)^{\pm \frac{1}{2}},h(-2)^{\pm \frac{1}{2}},...$$
where we quotient out by the ideal generated by elements of the form $$h(-m)^{1/2} h(-m)^{-1/2}-1$$ so that we may treat any exponent of $0$ as a simplification replacing $h(-m)^0$ with $1$.
On this space, we still have the action for $n < 0$ :
$$\alpha(n) \cdot v = \alpha(n)v,$$
but now also for $n > 0$ :
$$\alpha(n) \cdot \beta(m)^{\pm 1/2} = \pm \frac{n}{2} (\alpha,\beta) \beta(n)^{\pm 1/2 - 1} \delta_{m+n}$$
which we extend to all of $R(\widehat{h}^-)$ by the product rule.

It is clear by construction that $S(\widehat{h}^-) \subset R(\widehat{h}^-)$. Furthermore, we have :
\begin{proposition}
Given $n \in \Z$, the actions of the operators $\alpha(n)$ on $S(\widehat{h}^-)$ and $R(\widehat{h}^-)$ are compatible.
\end{proposition} 
\begin{proof}
We only need to verify that for $n > 0$,
$$\alpha(n) \cdot \beta(-n) = n (\alpha,\beta)$$
and then the proposition follows from the product rule.  Indeed,
\begin{align*}
\alpha(n) \cdot \beta(-n) &= \alpha(n) \cdot \beta(-n)^{1/2} \beta(-n)^{1/2} \\
&= \frac{n}{2} (\alpha,\beta) \beta(-n)^{-1/2} \beta(-n)^{1/2} + \beta(-n)^{-1/2} (\alpha,\beta) \frac{n}{2} \beta(-n)^{-1/2} \\
&= n (\alpha,\beta)
\end{align*}
as desired.
\end{proof}

This proposition allows us to extend the definitions of $L_m$ and $G_r$ to the entirety of $R(\widehat{h}^-)$ and implies that we may use the chain rule in some sense.  We can then add to Lemma \ref{bosonfermionbridge} :
\begin{lemma}
Given $n \in \Z_{\geq 0}$, $r,s \in \frac{1}{2} + \Z$ and $x \in L$.  Then
$$[G_{r},x(-n)^{s}] = ns x(-n)^{s-1} \tilde{x}(-n+r).$$
\end{lemma}

From these results and the facts that $(c,c) = (c,a) = 0$, we can now see that $ c(-1)^{-1/2} \tilde{a}(-1/2) \wedge \tilde{c}(-1/2) e^{rc}$ is a lowest weight vector for the Super Virasoro Algebra of weight $1/2$, hence also that
$$[G_r, B_s^a] = 0$$
with a direct computation.
We now prove the following lemma :

\begin{lemma}
Suppose $u,v \in P^{1/2}$.  Then
$$\res_{z_1,z_2} \{ Y( G_{- \frac{1}{2}} u,z_1),Y( G_{- \frac{1}{2}} v,z_2) \} = \res_{z_1,z_2} \{ G_{- \frac{1}{2}} , \{ Y(u,z_1),Y( G_{- \frac{1}{2}} v,z_2) \} \}$$
\end{lemma}
\begin{proof}
Use the fact that $Y(G_{-\frac{1}{2}} u,z_1) = \{G_{- \frac{1}{2}}, Y(u,z_1)\}$ and $G_{- \frac{1}{2}}^2 = L_{-1}$ and rewrite the left-hand side of the equation.
\end{proof}

With this lemma, we can now more easily evaluate the supercommutator of the operators $B^{a}_r$.
\begin{proposition}
Suppose $a_1,a_2$ satisfy $(a_i,\alpha) = (a_i,c) = 0$.  Then,
$$\{ B^{a_1}_{r_1}, B^{a_2}_{r_2}\} = (a_1,a_2) \delta_{r_1+r_2} c(0)$$
and
$$[B_{r}^{a_1},A_{m}^{a_2}] = 0.$$
\end{proposition}
\begin{proof}
First let us rewrite
\begin{align*}
 & G_{- \frac{1}{2}} \tilde{c}(-1/2) \wedge \tilde{a}(-1/2) c(-1)^{-1/2} e^{rc} \\
&= \tilde{a}(-1/2) c(-1)^{1/2} e^{rc} \\
&- \tilde{c}(-1/2) a(-1) c(-1)^{-1/2} e^{rc} \\
&- \frac{1}{2} \tilde{c}(-1/2) \wedge \tilde{a}(-1/2) \wedge \tilde{c}(-3/2) c(-1)^{-3/2} e^{rc}.
\end{align*}
Let us write $T_1$, $T_2$ and $T_3$ for each of these three terms in order with $a = a_2, r = r_2$.  Using the previous Lemma, we first evaluate
$$\res_{z_1,z_2} [ Y(\tilde{c}(-1/2) \wedge \tilde{a_1}(-1/2) c(-1)^{-1/2} e^{r_1 c},z_1) , Y(T_1 + T_2 + T_3,z_2) ]$$
Dealing with each term one at a time :
\begin{align*}
& \res_{z_1,z_2} [ Y(\tilde{c}(-1/2) \wedge \tilde{a_1}(-1/2) c(-1)^{-1/2} e^{r_1 c},z_1) , Y(T_1,z_2) ] \\
&= \res_{z_1,z_2} \tilde{c}(z_1) \{ \tilde{a}_1(z_1), \tilde{a}_2(z_2) \} c(z_1)^{-1/2} c(z_2)^{1/2} Y(e^{r_1 c},z_1) Y(e^{r_2 c},z_2) \\
&= \res_{z_1,z_2} (a_1,a_2) z_2^{-1} \delta(z_1/z_2) \tilde{c}(z_1) c(z_1)^{-1/2} c(z_2)^{1/2} Y(e^{r_1 c},z_1) Y(e^{r_2 c},z_2) \\
&= \res_{z} (a_1,a_2) \tilde{c}(z) Y(e^{(r_1+r_2) c},z) \\
&= (a_1,a_2) \res_z Y(\tilde{c}(-1/2) e^{(r_1+r_2)c},z) \\
& \res_{z_1,z_2} [ Y(\tilde{c}(-1/2) \wedge \tilde{a_1}(-1/2) c(-1)^{-1/2} e^{r_1 c},z_1) , Y(T_2,z_2)] \\
&= 0 \\
& \res_{z_1,z_2} [ Y(\tilde{c}(-1/2) \wedge \tilde{a_1}(-1/2) c(-1)^{-1/2} e^{r_1 c},z_1) , Y(T_3,z_2)] \\
&= \frac{1}{2} \res_{z_1,z_2} \tilde{c}(z_1) \tilde{c}(z_2) \{ \tilde{a}_1 (z_1) , \tilde{a}_2 (z_2) \} \frac{\partial}{\partial z_2} \tilde{c}(z_2)  c(z_1)^{-1/2} c(z_2)^{-3/2} \\ & \hspace{0.5cm} \centerdot Y(e^{r_1 c},z_1) Y(e^{r_2 c},z_2) \\
&= \frac{1}{2} (a_1,a_2) \res_{z_1,z_2} \tilde{c}(z_1) \tilde{c}(z_2) \delta(z_1/z_2) \frac{\partial}{\partial z_2} \tilde{c}(z_2) c(z_1)^{-1/2} c(z_2)^{-3/2} \\ & \hspace{0.5cm} \centerdot Y(e^{r_1 c},z_1) Y(e^{r_2 c},z_2) \\
&= 0
\end{align*}
This proves the first assertion.  To prove the second assertion, we evaluate
\begin{align*}
& \res_{z_1,z_2} [Y(\tilde{c}(-1/2) \wedge \tilde{a_1}(-1/2) c(-1)^{-1/2} e^{rc},z_1), Y(a_2(-1)e^{mc}  ,z_2) ] \\
&= 0
\end{align*}
and
\begin{align*}
& \res_{z_1,z_2} [Y(\tilde{c}(-1/2) \tilde{a}_1(-1/2) c(-1)^{-1/2} e^{rc},z_1), Y(-m \tilde{a}_2(-1/2) \tilde{c}(-1/2)  ,z_2) ] \\
&= m \res_{z_1,z_2} \{ \tilde{a}_1 (z), \tilde{a}_2 (z) \}\tilde{c}(z_1) \tilde{c}(z_2) c(z_1)^{-1/2} Y(e^{rc},z_1) Y(e^{mc},z_2) \\
&= m (a_1,a_2) \res_{z_1,z_2} \delta(z_1/z_2) \tilde{c}(z_1) \tilde{c}(z_2) c(z_1)^{-1/2} Y(e^{rc},z_1) Y(e^{mc},z_2) \\
&= 0.
\end{align*}
Together, these two evaluations show that $[B^{a_1}_r, A^{a_2}_m] = 0$.
\end{proof}

What we have done is define two sets of $d-2$ operators $A^a_m$ and $B^b_r$, which commute with the super Virasoro algebra and satisfy very simple commutation relations, where the $A$'s are "bosons" and the $B$'s are "fermions".  In particular, we can evaluate their adjoints directly, because most of the terms commute, to find
$$(A^a_m)^* = A^a_{-m}$$
$$(B^b_r)^* = B^b_{-r}$$

Together with the commutation relations, which implies linear independence, and the multiplicities of $g_{NS}$ (Theorem \ref{multiplicitiesN=1}), we have now proven :
\begin{theorem} \label{DDFbasis}
Let $\alpha$ be a root of $g_{NS}$.  Then
$$(g_{NS})_{\alpha} = \spn \{ A^{a_{i_1}}_{-m_1} ... A^{a_{i_k}}_{-m_k} B^{a_{j_1}}_{-r_1}...B^{a_{j_l}}_{-r_l} e^{\alpha + Nc} \}$$
for $i_1,...,i_k, j_1,...,j_l \in \{1,...,d-2\}$ and $N = m_1 + ... + m_k + r_1 + ... + r_l = \frac{1}{2} - \frac{(\alpha,\alpha)}{2}$.
\end{theorem}
In particular, these operators provide an explicit basis for the elements of $g_{NS}$.

\begin{example}
We specialize to the lattice $L = E_{10} = E_8 \oplus II_{1,1}$.  We may write an element $\lambda$ of $L$ as $\lambda = \lambda' + m \rho + n \rho'$ where $(\rho,\rho') = -1$ and $\lambda' \in E_8$.  Then it is clear that $(\lambda,\lambda) = (\lambda',\lambda') - 2mn$.  In particular, the number $n$ induces a grading on $L$ which is inherited by $g_{NS}$ :
$$g_{NS} = \bigoplus\limits_{n \in \Z} (g_{NS})_n.$$
We can use this grading to study $g_{NS}$ to some extent.

If $n = 0$ then $(\lambda,\lambda) = (\lambda',\lambda') \geq 0$ because $E_8$ is positive-definite and equals $0$ only if $\lambda = 0$.  The only elements of $P^{1/2}$ in this case are linear combinations of $a_1(-1/2),...,a_{10}(-1/2)$ and nothing more.

If $n = \pm 1$ then $(\lambda,\lambda) = (\lambda',\lambda') \mp 2m$.  In this case we use the previous theorem to describe the structure of $(g_{NS})_{\lambda'}$.  For any $\lambda' \in E_8$ and any $m_1,..,m_k \in \Z_{\geq 1}, r_1,...,r_l \in \frac{1}{2} + \Z_{\geq 0}$, there exists $m \in \Z$ such that
$$m_1 + ... + m_k + r_1 + ... + r_l = \frac{1}{2} - \frac{(\lambda',\lambda')}{2} \mp m$$
as long as $l$ is an odd number.  Therefore, summing over all $m \in \Z$ and $\lambda' \in E_8$, we then have
$$(g_{NS})_{\pm 1} \simeq \bigoplus\limits_{i \in \Z} (V_{E_8} \otimes F_8)^{i + \frac{1}{2}},$$
linearly.
\end{example}

\begin{example}
We modify the previous example.  Let $L = E_8 \oplus \Z \rho \oplus \Z \rho'$ where $(\rho,\rho') = -\frac{1}{2}$.  Note that $L$ is not an even lattice anymore, yet we may still consider the space $V_{NS}$, no longer with a vertex operator superalgebra structure.  The no-ghost theorem still applies, and the operators $A^a_m$'s and $B^b_r$'s still provide an explicit basis for $V_{NS}$ via Theorem \ref{DDFbasis}.

Again write an element $\lambda$ of $L$ as $\lambda' + m \rho + n \rho'$, $\lambda' \in E_8$.  We grade $g_{NS}$ in the same way.  In the cases $n = \pm 1$, the main difference is that now, $(\lambda,\lambda) = (\lambda',\lambda') \mp m$.  Therefore, the relation
$$m_1 + ... + m_k + r_1 + ... + r_l = \frac{1}{2} - \frac{(\lambda',\lambda')}{2} \mp \frac{m}{2}$$
is satisfied for some $m \in \Z$, regardless of whether or not $l$ is odd or even.  Therefore, summing over all $m \in \Z$ and $\lambda' \in E_8$ we then have
$$(g_{NS})_{\pm 1} \simeq V_{E_8} \otimes F_8,$$
linearly.

\end{example}

\subsection{Semi-infinite cohomology in $N=1$, continued}
We continue what we first developed in section \ref{BRSTSVir}, once again a review of \cite{lianzuckerman1}.  For $\kappa = \frac{1}{2}$ we take $V_\frac{1}{2} = V_{NS}$.  For $\kappa = 0$, replace $\widehat{A}$ with
$$\widehat{A} = A \otimes \C[t,t^{-1}] \oplus \C c.$$
The main difference is that we now consider integer powers of $t$. Once again let
$$F_d = \Lambda(\widehat{A}^-).$$  However, this time let $R = Spin(9,1)$ be the spinor representation for $$a_1(0),...,a_5(0),a^*_1(0),...,a^*_5(0).$$  Then we define
$$V_{Ram} = V_L \otimes R \otimes F_d.$$  For $\kappa = 0$ we then have $V_0 = V_{Ram}$.
Next, decompose the chain complex $C^{\frac{\infty}{2}+*}(\mf{g}_\kappa, V_\kappa)$ further as
$$C^{r,s} = \{ v \in C^{\frac{\infty}{2}+*}(\mf{g}_\kappa, V_\kappa), Uv = rv, \theta(L_0)v = -sv\}.$$
It is clear that $d : C^{r,s} \to C^{r+1,s}$.
In \cite{lianzuckerman1}, they went through a great deal of effort to obtain the following proposition :
\begin{proposition}
There exists a bilinear form $\langle .,. \rangle_C$ on $C^{\frac{\infty}{2}+*}(\mf{g}_\kappa, V_\kappa)$ which is non-degenerate when restricted to $C^{r,s} \oplus C^{-r,s}$, and in particular non-degenerate for a fixed eigenvalue of $\theta(L_0)$.
\end{proposition}
We also write
$$C^r = \bigoplus\limits_{s \in \frac{1}{2} \Z } C^{r,s}.$$
We now define
\begin{align*}
& &B^r &= \{ v \in C^r, b(0) v = 0 \} & \text{ for } \kappa = \frac{1}{2} & &\\
& & &= \{ v \in C^r, b(0) v = \beta(0) v = 0 \} & \text{ for } \kappa = 0 & & \\
& & C^r_{\text{rel}} &= \{ v \in B^r, \theta(L_0)v = 0 \} & \text{ for } \kappa = \frac{1}{2} & & \\
& & &= \{ v \in B^r, \theta(L_0)v = \theta(G_0)v = 0 \} & \text{ for } \kappa = 0. & & 
\end{align*}
This now defines a relative chain complex, which also results in cohomology groups $H^n_{\text{rel}}$.  Clearly these groups are also graded by $L$, hence we may decompose them as
$$H^n_{\text{rel}} = \bigoplus\limits_{\lambda \in L} H^n_{\text{rel}}(\lambda).$$
Then, \cite{lianzuckerman1} dedicates en entire section to prove the following theorem, which we adjust here to our different choice of $q_F$, :
\begin{theorem}
Suppose $\lambda \neq 0$.  Then $H^n_{\text{rel}} = 0$ unless $n = 1$.
\end{theorem}
With this theorem, as in \cite{frenkelgarlandzuckerman}, they are able to compute multiplicities of these relative cohomology groups and they found :
\begin{proposition}
For $\kappa = \frac{1}{2}$,
$$\dim H^1_{\text{rel}} (\lambda) = c\left( \frac{1}{2} - \frac{\langle \lambda, \lambda \rangle}{2} \right).$$
where $c(r)$ is the coefficient as in Theorem \ref{multiplicitiesN=1}.
For $\kappa = 0$,
$$\dim H^1_{\text{rel}} (\lambda) = \frac{\dim R}{2} c \left( - \frac{\langle \lambda,\lambda \rangle}{2}\right)$$
where $c(n)$ is the coefficient of $q^n$ in
$$\frac{\prod\limits_{i \in \Z_{\geq 0}} (1-q^{2i})^8}{\prod\limits_{i \in \Z_{\geq 0}} (1 - q)^{16}}.$$
\end{proposition}
Furthermore, they also recover the no-ghost theorem in this case, i.e. that $H^1_{\text{rel}}(\lambda)$ is positive definite.  Just as in the previous section, this indicates that, for $\kappa = \frac{1}{2}$, the relative cohomology groups should match with $g_{NS}$.  We may once again define a map
$$T : P^{1/2} \to H^1_{rel}, T(v) = v \otimes c(1)$$
In \cite{lianzuckerman1}, Theorem 2.20, they show that this map is a unitary isomorphism.  Therefore, we see that BRST cohomology provides yet another way to study $g_{NS}$ and also reveals an additional space for $\kappa = 0$.

\begin{remark}
A natural question whether or not the Lie algebra structure in $N=1$ can be extended to include $\kappa = 0$?  The answer to this question is yes.  A Lie superalgebra was obtained by Scheithauer in \cite{scheit}.  A Lie superalgebra structure on $H^1$ (not relative cohomology) was also obtained by Lian, Moore and Zuckerman in the unpublished paper \cite{lianmoorezuckerman}.  Both of these results were achieved by "bosonization" of the fermions and the inclusion of an additional operator known as the "picture changing" operator.  Because of this bosonization procedure, the structure presented in these papers appears very different from the one that appears in this document and therefore we do not include it here.  Note that the same techniques can also be used in $N=2$ as done in \cite{kugel}, although the more traditional techniques that we use here in the next chapter will provide us with plenty to examine.
\end{remark}

\section{The N=2 Lie algebra of physical states}

\subsection{No-ghost theorem in the N=2 model}
So far we have only considered cases where $L$ is a Lorentzian lattice.  Due to the unique structure of the $N=2$ superconformal algebra, we will now consider cases where the signature of $L$ is $(l,l)$, i.e.,
$$L \otimes \R = L^+ \oplus L^-$$
where $L^+$ and $L^-$ are isotropic subspaces of dimension $l$.  Then we may choose basis elements $h_i^+, h_i^-$ of $L^+, L^-$ as before. 
We replace the definition of $P^i$ with
$$P^{i,j} = \{v \in V_{NS} | L_0 v = iv, J_0 v = jv, G^{\pm}_{n - \frac{1}{2}}v = J_n v = 0 \text{ for all } n > 0 \}.$$
We may refer to the eigenvalues of $J_0$ as \emph{charge}.  Again fix $\alpha \neq 0 \in L$.  In this section, we will attempt to describe $P^{0,0}(\alpha)$ for $l=2$.  If $(\alpha,\alpha) = 0$ then it is easy to check that $P^{0,0}(\alpha) = \C e^{\alpha}$ therefore we may assume $(\alpha,\alpha) < 0$.  In this case, write $\alpha = \alpha^+ - \alpha^-$ where $(\alpha^+,\alpha^+) = (\alpha^-,\alpha^-) = 0$ and $(\alpha^+,\alpha^-) > 0$.  Choose $\beta^+ \in L^+$ and $\beta^- \in L^-$ such that $(\beta^+,\alpha) = (\beta^-,\alpha) = 0$ and $(\beta^+,\beta^-) = (\alpha^+,\alpha^-)$.  

In this case, we may write
$$J_n = k \sum\limits_{r \in \Z + \frac{1}{2}} : \widetilde{\alpha^+}(r) \wedge \widetilde{\alpha^-}(n-r) : + :\widetilde{\beta^+}(r) \wedge \widetilde{\beta^-}(n-r):$$
for $k$ such that $k (\alpha^+,\alpha^-)=1$.  This description of $J_n$ is valid because $J_n$ does not depend on the choice of dual bases for $L$.

We now adapt Lemma \ref{independentlemma} in this setting, following a similar approach to \cite{bienkowska} : 
\begin{lemma}
Suppose $l=2$ and $\lambda_1,...,\lambda_n, \lambda^j_1,...,\lambda^j_{n_j}, \mu^{\pm}_1,...,\mu^{\pm}_{m_{\pm}} \in \Z_{\geq 0}$ and \break $\epsilon^{\pm}_1,...,\epsilon^{\pm}_{k_{\pm}}$, $ \delta^{\pm}_1,...,\delta^{\pm}_{l_{\pm}} \in \{0,1\}$.  Then for different values of $\lambda,\mu,\epsilon, \delta$, the elements
\begin{align} \label{basiselements}
G_{- \frac{1}{2}}^{\pm \epsilon^{\pm}_1}...G_{-k + \frac{1}{2}}^{\pm \epsilon_{k_{\pm}}}L_{-1}^{\lambda_1}...L_{-n}^{\lambda_n} J_{-1}^{\lambda^j_1}...J_{-n}^{\lambda^j_{n_j}} & \widetilde{\beta^{\pm}} \left(- \frac{1}{2}\right)^{\delta_1}...\widetilde{\beta^{\pm}} \left(-l_{\pm} + \frac{1}{2}\right)^{\delta_{l_{\pm}}} \\ \cdot &\beta^{\pm}(-1)^{\mu_1}...\beta^{\pm}(-m_{\pm})^{\mu_{m_{\pm}}} e^{\alpha} \nonumber
\end{align}
where all $+$'ses appear before $-$'ses, is a basis for $(V_{NS})_{\alpha}$.
\end{lemma}

\begin{proof}
Write $v = p e^{\alpha} \in (V_{NS})_{\alpha}$.  We first prove that $v$ can be written as a linear combination of elements of the form in (\ref{basiselements}).  Because $\alpha^{\pm},\beta^{\pm}$ is a basis for $L \otimes \R$, we may assume that $p$ is a monomial in terms of the form in (\ref{basiselements}), but with some additional powers of $\alpha^{\pm}(-n), \widetilde{\alpha^{\pm}}(-s)$ at the front.  Let $M$ be the $L_0$-degree of $p$ and $N$ be the number of non-zero exponents in this monomial.  The proof proceeds by double induction on $M$ first and $N$.  The case $M=0$ is trivial.  We explain how the inductive step proceeds before we deal with the base case $N=1$.

If $p = \alpha^{\pm}(-n) q$ or $\widetilde{\alpha^{\pm}}(-s) q$ for some $q$ then by induction hypothesis (on $M$) $q$ is a linear combination of terms of the form in (\ref{basiselements}).  We may then assume $q$ is again a monomial of the form in (\ref{basiselements}).  Commuting $\alpha^{\pm}(-n)$ or $\widetilde{\alpha^{\pm}}(-s)$ with $q$ will lead in new terms in two ways : some where a power of $\alpha^{\pm}(-n)$ or $\widetilde{\alpha^{\pm}}(-s)$ no longer leads the monomial, and some where they still lead the monomial but $N$ decreases.  In both cases we may still apply induction on $N$ to complete the proof.  The only scenario in which we do not have an inductive step is when $N=1$.

In other words, we need to show that $\alpha^{\pm}(-n) e^{\alpha}$ or $\widetilde{\alpha^{\pm}}(-s) e^{\alpha}$ can be written as a linear combination of elements in (\ref{basiselements}).  Observe that
$$\widetilde{\alpha^{\pm}}(-s) e^{\alpha} = G^{\pm}_{-s} e^{\alpha},$$
therefore terms of this type are accounted for.  For $\alpha^{\pm}(-n) e^{\alpha}$ however, the situation is slightly more complicated.  We perform the proof for $\alpha^+(-n) e^{\alpha}$ for simplicity, and the proof for the other term is more or less the same.  By a simple computation, we see :
$$G^+_{-n+r} G^-_{-r} e^{\alpha} = \alpha^+(-n) e^{\alpha} +  \widetilde{\alpha^+}(-r) \wedge \widetilde{\alpha^-}(-n+r) e^{\alpha}$$
and therefore,
$$\sum\limits_{1/2 \leq r < n} G^+_{-n+r} G^-_{-r} e^{\alpha} = n \alpha^+(-n) e^{\alpha} + \sum\limits_{1/2 \leq r < n} \widetilde{\alpha^+}(-r) \wedge \widetilde{\alpha^-}(-n+r) e^{\alpha}.$$
However, we have
$$J_{-n} e^{\alpha} = \sum\limits_{1/2 \leq r < n} \widetilde{\alpha^+}(-r) \wedge \widetilde{\alpha^-}(-n+r) e^{\alpha} + \widetilde{\beta^+}(-r) \wedge \widetilde{\beta^-}(-n+r) e^{\alpha}.$$
Consequently, we can solve for $\alpha^+(-n) e^{\alpha}$ in terms of other terms of the form in (\ref{basiselements}).  This completes the case $N=1$ and in general proves that the terms (\ref{basiselements}) form a spanning set for $(V_{NS})_{\alpha}$.

To see why they form a basis, count the dimension of $(V_{NS})_{\alpha}$ at a fixed weight $M$ using the definition and observe that it matches the number of elements of the form in (\ref{basiselements}), and since they form a spanning set they must also form a basis.
\end{proof}

\begin{lemma}
Let $l = 2$ and fix $\alpha,\beta$. Let $N^{r,s}$ denote the span of all elements $v$ of $V_{NS}(\alpha)$ in the form (\ref{basiselements}) such that $L_0 v = rv$ and $J_0v = sv$, and any of the $\lambda_i, \lambda^j_i, \epsilon^{\pm}_i$ is non-zero.  Then the operators $G^{+}_{\frac{1}{2}}$, $G^-_{\frac{1}{2}}$ and $J_1 + G^+_{\frac{1}{2}} G^-_{\frac{1}{2}} - G^-_{\frac{1}{2}}G^+_{\frac{1}{2}}$ are well-defined maps $N^{0,0} \to N^{-1/2,1}$, $N^{0,0} \to N^{-1/2,-1}$ and $N^{0,0} \to N^{1,0}$.
\end{lemma}

\begin{proof}
Let $u$ be an element of $N^{r,s}$.  Then we may write
$$ u = G^+_{-\frac{1}{2}}x + G^-_{-\frac{1}{2}}y + J_{-1} z$$
for some $x,y,z$, because $G^+_{-\frac{1}{2}},G^-_{-\frac{1}{2}},J_{-1}$ generate the lower part of the $N=2$ superconformal algebra.  We may evaluate $L_0x = L_0y = -1/2 x, L_0z = -z$ and $J_0x = -x, J_0y = y, J_0z = 0$.  The proof is essentially multiple computations as in the $N=0$ and $N=1$ cases.

Starting with the element $G^+_{-\frac{1}{2}} x \in N^{0,0}$,
\begin{align*}
G^{+}_{\frac{1}{2}}G^+_{-\frac{1}{2}}x &= G^{+}_{-\frac{1}{2}}G^+_{\frac{1}{2}}x \\
G^-_{\frac{1}{2}} G^+_{-\frac{1}{2}}x &= - G^+_{-\frac{1}{2}} G^-_{\frac{1}{2}}x + (L_0 - \frac{1}{2}J_0) x
\\&= - G^+_{-\frac{1}{2}} G^-_{\frac{1}{2}}x \in N^{-1/2,-1}
\\ J_1 G^+_{-\frac{1}{2}} &= G^+_{-\frac{1}{2}} J_1 x + G^+_{\frac{1}{2}}x
\\ G^-_{\frac{1}{2}}G^+_{\frac{1}{2}}G^+_{-\frac{1}{2}}x &= -G^-_{\frac{1}{2}}G^+_{-\frac{1}{2}}G^+_{\frac{1}{2}}x \\
&= G^+_{-\frac{1}{2}}G^-_{\frac{1}{2}}G^+_{\frac{1}{2}}x - (L_0 - \frac{1}{2}J_0) G^+_{\frac{1}{2}}x \\
&= G^+_{-\frac{1}{2}}G^-_{\frac{1}{2}}G^+_{\frac{1}{2}}x + G^+_{\frac{1}{2}}x
\\ G^+_{\frac{1}{2}}G^-_{\frac{1}{2}}G^+_{-\frac{1}{2}}x &= - G^+_{\frac{1}{2}} G^+_{-\frac{1}{2}} G^-_{\frac{1}{2}} x = G^+_{-\frac{1}{2}} G^+_{-\frac{1}{2}} G^-_{\frac{1}{2}} \in N^{1,0}
\end{align*}
From there we see that $J_1 + G^+_{\frac{1}{2}} G^-_{\frac{1}{2}} - G^-_{\frac{1}{2}}G^+_{\frac{1}{2}}$ indeed maps $G^+_{-\frac{1}{2}} x$ to an element of $N^{1,0}$.  Observe that these computations required specific values of $L_0$ and $J_0$.  The computation for $G^-_{-\frac{1}{2}} y$ is essentially the same with some minor sign differences.  On $J_{-1}z$, we have,
\begin{align*}
G^+_{\frac{1}{2}} J_{-1}z &= -G^+_{-\frac{1}{2}} z + J_{-1} G^+_{\frac{1}{2}} z \in N^{-1/2,1} \\
G^-_{\frac{1}{2}} J_{-1}z &= G^-_{-\frac{1}{2}} z + J_{-1} G^-_{\frac{1}{2}} z \in N^{-1/2,-1} \\
J_1 J_{-1} z &= J_{-1} J_1 z + \frac{c}{3} z
\end{align*}
\begin{align*}
G^+_{\frac{1}{2}} G^-_{\frac{1}{2}} J_{-1} z &= G^+_{\frac{1}{2}} J_{-1} G^-_{\frac{1}{2}} z + G^+_{\frac{1}{2}} G^-_{-\frac{1}{2}} z \\
&= -z + \text{ other terms in } N^{1,0} \\
G^-_{\frac{1}{2}} G^+_{\frac{1}{2}} J_{-1} z &= G^-_{\frac{1}{2}} J_{-1} G^+_{\frac{1}{2}} z - G^-_{\frac{1}{2}} G^+_{-\frac{1}{2}} z \\
&= z + \text{ other terms in } N^{1,0}
\end{align*}
Adding the last three terms we find that $J_1 + G^+_{\frac{1}{2}} G^-_{\frac{1}{2}} - G^-_{\frac{1}{2}}G^+_{\frac{1}{2}}$ maps $J_{-1} z$ into $N^{1,0}$ if and only if $\frac{c}{3} - 2 = 0$, or equivalently $c = 6$.  This happens for $l = 2$, as assumed.
\end{proof}

We are now ready for the $N=2$ version of no-ghost theorem.

\begin{theorem}
Suppose $u \in P^{0,0}(\alpha)$ and $L$ is an even lattice of signature $(2,2)$.  Then $u = v + w$ where $v \in \spn({e^\alpha}) \cap P^{0,0}(\alpha)$ and $w \in P^{0,0}$ is a linear combination of terms $G^{\pm}_{-n} x$ and $J_{-n} x$ for some $x$'s in $V_{NS}$.
\end{theorem}

\begin{proof}
The proof proceeds as in Theorem \ref{noghost}, where we replace $K^s$ with the subspace of $V_s (\alpha)$ spanned by elements of the form
\begin{equation} \label{basiselementsK}
 \widetilde{\beta^{\pm}} \left(- \frac{1}{2}\right)^{\delta_1}...\widetilde{\beta^{\pm}} \left(-l_{\pm} + \frac{1}{2}\right)^{\delta_{l_{\pm}}} \beta^{\pm}(-1)^{\mu_1}...\beta^{\pm}(-m_{\pm})^{\mu_{m_{\pm}}} e^{\alpha}
\end{equation}
for some exponents $\delta_1,...,\delta_{l_{\pm}} \in \{ 0,1 \}$ and $\mu_1,...\mu_{m_{\pm}} \in \Z_+$.  As before, the operators $G^+_{\frac{1}{2}}$, $G^-_{\frac{1}{2}}$ and $J_1 + G^+_{\frac{1}{2}} G^-_{\frac{1}{2}} - G^-_{\frac{1}{2}}G^+_{\frac{1}{2}}$ map $K^s \to K^{s - 1/2}$ and $K^s \to K^{s-1}$.

Therefore, we can write $u = v + w$ where $v \in K^0$ and $w \in N^{0,0}$.  Then using the previous lemma we find that if $u \in P^{0,0}$ then so are $v,w$.  It remains only to prove that $v \in \spn({e^\alpha}) \cap P^{0,0}(\alpha)$.  To do so let
$$V = \spn \{  \widetilde{\beta^{\pm}} \left(- \frac{1}{2}\right)^{\delta_1}...\widetilde{\beta^{\pm}} \left(-l_{\pm} + \frac{1}{2}\right)^{\delta_{l_{\pm}}} \beta^{\pm}(-1)^{\mu_1}...\beta^{\pm}(-m_{\pm})^{\mu_{m_{\pm}}} e^{\alpha} \}.$$

Then given $v$ an element of $V$, it is clear that each monomial in its linear combination has the same $L_0$-weight.  Given a list of these terms, one may order them in such a way that a term is "greater" than another if the exponent on the highest number $k \in \frac{1}{2}\Z$ for which $\beta^{\pm}(-k)$ or $\widetilde{\beta^{\pm}}(-k)$ have different exponents is greater, giving preference to $+$ over $-$.

Then by looking at a few cases, excluding $e^{\alpha}$ itself, and applying operators $G^+_{\frac{1}{2}}$, $G^-_{\frac{1}{2}}$, and $J_1$, we see that from the maximal term we can create a non-zero term that cannot be attained by any lower ones.  However there are two special cases where this is not possible : first the one with maximal terms which contain $x = \widetilde{\beta^+}(-r) \wedge \widetilde{\beta^-} (-1/2)$.  In this case there is no way to create a term that cannot be attained by any other term.  By applying $J_1$, there may be other terms, which are no longer maximal, which will contain $y = \widetilde{\beta^-}(-3/2) \wedge \widetilde{\beta^+} (-r-1)$ which may produce the same result.  In this case observe that $L_1 x = -L_1 y + ...$ but $J_1 x = J_1 y + ...$, which induces a small system of equation on the resulting non-zero term where the only solution is the trivial solution.  Therefore, this first special case is not possible.  The second special cases is the same but with $+$ and $-$ swapped, and can be dealt with similarly.

Therefore, we conclude that the only possible term is $e^{\alpha}$ itself hence $K^0 = \spn{(e^{\alpha})} \cap P^{0,0}(\alpha)$.
\end{proof}

The bilinear form $(.,.)$ is defined in the exact same way and its radical is still an ideal.  Let $g^{(2)}_{NS}$ be the quotient of $P^{0,0}$ by the radical of the bilinear form $(.,.)$.  For now, we can say the following :

\begin{corollary} \label{corollaryn=2structure}
The vector space $g^{(2)}_{NS}$ is isomorphic to $\spn \{ e^{\alpha} | (\alpha,\alpha) = 0 \}$.
\end{corollary}
\begin{proof}
From the previous theorem, every $u = v+w$ where $w$ lies in the radical and the options for $v$ are very limited.
\end{proof}

\begin{example} \label{matrixrealization}
Consider the case where $L = II_{2,2}$.  All even indefinite unimodular lattices are unique up to isomorphism, therefore $L \simeq II_{1,1} \oplus II_{1,1}$.  We may then describe $L$ as the lattice generated by elements $\rho_1,\rho'_1,\rho_2,\rho'_2$ where $(\rho_i,\rho_j) = (\rho'_i,\rho'_j) = 0$ and $(\rho_i,\rho'_j) = \delta_{i,j}$.  Suppose then that $\lambda \in L$ satisfy $(\lambda,\lambda) = 0$.  Then write
$$\lambda = m_1 \rho_1 - m_2 \rho_2 + n_1 \rho'_1 + n_2 \rho'_2.$$
From $(\lambda,\lambda) = 0$ we then find the condition $m_1 n_1 - m_2 n_2 = 0$.  We may then think of $\lambda$ as a matrix of determinant $0$ :
$$\lambda = \begin{bmatrix}
m_1 & n_2 \\
m_2 & n_1
\end{bmatrix}.$$
Given two $2 \times 2$ matrices $A=[a_{ij}], B=[b_{ij}]$ of determinant $0$ representing two elements $\lambda_1,\lambda_2$, we recover $(\lambda_1,\lambda_2)$ from the quadratic form $\det(\lambda)$ with the formula
$$a_{11}b_{22} + b_{11} a_{22} - a_{12} b_{21} - a_{21} b_{12} = \det(A+B) - \det(A) - \det(B).$$
\end{example}

\subsection{The N=2 Lie algebra of physical states} \label{subsectionN=2}
Once again we consider the space $P^{0,0}$ of lowest weight vectors for the $N=2$ superconformal algebra with energy and charge $0$.  Define a Lie bracket on $P^{0,0}$ by
$$[u,v] = \res_{z} Y(G^+_{-\frac{1}{2}} G^-_{-\frac{1}{2}}u,z)v$$

\begin{remark}
There is another option for the Lie bracket, but it does not provide anything new, indeed, we may define
$$[u,v]' = \res_{z} Y(G^-_{-\frac{1}{2}} G^+_{-\frac{1}{2}}u,z)v$$
instead, but we have
$$\{ G^+_{-\frac{1}{2}}, G^-_{-\frac{1}{2}} \} = L_{-1}$$
and then by taking residues we have $[u,v]' = -[u,v]$.
\end{remark}

\begin{lemma}
Suppose $u \in P^{i,j}$.  Then we have the following commutation relations :
\begin{align*}
\{ G^+_{m+\frac{1}{2}}, (G^+_{-\frac{1}{2}} G^-_{-\frac{1}{2}}u)_0 \} &= -(i + \frac{j}{2})(m+1) (G^+_{-\frac{1}{2}}u)_m \\
\{ G^-_{m+\frac{1}{2}}, (G^+_{-\frac{1}{2}} G^-_{-\frac{1}{2}}u)_0 \} &= ((i - \frac{j}{2}+1)(m+1) - m - 1)(G^-_{-\frac{1}{2}} u)_m \\
\{ J_m,(G^+_{-\frac{1}{2}} G^-_{-\frac{1}{2}}u)_0 \} &= j (G^+_{-\frac{1}{2}} G^-_{-\frac{1}{2}}u)_m - m(i + \frac{j}{2}) u_{m-1}
\end{align*}
In particular, $(G^+_{-\frac{1}{2}} G^-_{-\frac{1}{2}}u)_0$ commutes with the action of the $N=2$ superconformal algebra if and only if $i=j = 0$.
\end{lemma}

\begin{proof}
Recall the commutator formula for an SVOA : 
\begin{equation}
\{Y(x,z_1),Y(y,z_2)\} = \res_{z_0} z_2^{-1} \left( e^{-z_0 \frac{\partial}{\partial z_1}} \delta(z_1/z_2) \right) Y(Y(x,z_0)y,z_2).
\end{equation}
Here we use it three times, for $x = \tau^+, \tau^-, j$ and $y = G^+_{-\frac{1}{2}} G^-_{-\frac{1}{2}}u$.  It is enough to verify the following table :
\begin{align*}
\tau^+_{0} y &= G^+_{-\frac{1}{2}} G^+_{-\frac{1}{2}} G^-_{-\frac{1}{2}}u = 0 \\
\tau^+_{1} y &= G^+_{\frac{1}{2}}G^+_{-\frac{1}{2}} G^-_{-\frac{1}{2}}u = - (i + \frac{j}{2}) G^+_{-\frac{1}{2}}u \\
\tau^+_{n} y &= 0  \text{ for } n \geq 2\\
\tau^-_0 y &= G^-_{-\frac{1}{2}} G^+_{-\frac{1}{2}} G^-_{-\frac{1}{2}}u = L_{-1} G^-_{-\frac{1}{2}}u\\
\tau^-_{1} y &= G^-_{\frac{1}{2}}G^+_{-\frac{1}{2}} G^-_{-\frac{1}{2}}u = (i - \frac{j}{2} + 1) G^-_{-\frac{1}{2}}u \\
\tau^-_{n} y &= 0  \text{ for } n \geq 2\\
j_0 y &= J_0 G^+_{-\frac{1}{2}} G^-_{-\frac{1}{2}}u = j G^+_{-\frac{1}{2}} G^-_{-\frac{1}{2}}u\\
j_{1} y &= J_1 G^+_{-\frac{1}{2}} G^-_{-\frac{1}{2}}u = G^+_{\frac{1}{2}} G^-_{-\frac{1}{2}} u + ... =  (i + \frac{j}{2}) u\\
j_{n} y &= 0  \text{ for } n \geq 2
\end{align*}
and then apply the standard techniques presented before to find the coefficients of the commutator formula.
\end{proof}

\begin{theorem}
The space $$\widetilde{P}^{0,0}=P^{0,0}/(G^+_{-\frac{1}{2}}V_{-\frac{1}{2},-1} + G^-_{-\frac{1}{2}}V_{-\frac{1}{2},1}) \cap P^{0,0}$$ with bracket defined above is a Lie algebra.
\end{theorem}
\begin{proof}
From the previous Lemma, we see that the Lie bracket of two elements of $P^{0,0}$ is once again an element of $P^{0,0}$, because $i=j=0$.  Furthermore, because $(G^{\pm}_{-\frac{1}{2}})^2 = 0$, and replacing $[u,v]$ with $[u,v]'$ where needed, we see that brackets with elements of the quotient are all trivial.  Therefore, the Lie bracket is well-defined.

As in $N=1$, Jacobi identity for Lie algebra follows from Jacobi identity for SVOAs, here because all elements are even elements.

Therefore, the only non-trivial thing to check is anticommutativity.  By skew-symmetry, we have
$$Y(G^+_{-\frac{1}{2}} G^-_{-\frac{1}{2}}u,z)v = e^{z L_{-1}} Y(v,-z) G^+_{-\frac{1}{2}} G^-_{- \frac{1}{2}} u.$$
Observe that we still have in vertex operator superalgebras the properties
$$\{ G^+_{-\frac{1}{2}}, Y(u,z) \} = Y(G^+_{-\frac{1}{2}}u,z)$$
and
$$\{ G^-_{-\frac{1}{2}}, Y(u,z) \} = Y(G^-_{-\frac{1}{2}}u,z),$$
so we may rewrite
\begin{align*}
&e^{z L_{-1}} Y(v,-z) G^+_{-\frac{1}{2}} G^-_{- \frac{1}{2}} u = e^{z L_{-1}} (G^+_{-\frac{1}{2}} Y(v,-z) G^-_{-\frac{1}{2}} u - Y(G^+_{-\frac{1}{2}}v,-z) G^-_{-\frac{1}{2}} u ) \\
&= e^{zL_{-1}}(G^+_{-\frac{1}{2}} Y(v,-z) G^-_{-\frac{1}{2}} u - G^-_{-\frac{1}{2}} Y(G^+_{-\frac{1}{2}} v,-z)u - Y(G^-_{-\frac{1}{2}} G^+_{- \frac{1}{2}} v,-z)u)
\end{align*}
Taking residues, we find
$$[u,v] = [v,u]' + G^+_{-\frac{1}{2}} x + G^-_{- \frac{1}{2}} y$$
for some $x,y$.  Therefore, we have
$$[u,v] = [v,u]' = -[v,u]$$
in the quotient.
\end{proof}

This Lie algebra structure is also inherited by $g^{(2)}_{NS}$, for which we have the following more concrete description (continuing from Corollary \ref{corollaryn=2structure}) :

\begin{corollary} \label{bracketn=2}
The vector space $g^{(2)}_{NS}$ is a Lie algebra with Lie bracket given by

\begin{equation*}
[e^{\alpha},e^{\beta}]=\begin{cases}
     (\alpha^+,\beta^-) \epsilon(\alpha,\beta) e^{\alpha+\beta}      \quad &\text{if } \, (\alpha,\beta) = 0 \\
          0 &\text{if } \, (\alpha,\beta) \neq 0 \\
     \end{cases}
\end{equation*}
where $\alpha=\alpha^+ + \alpha^-,\beta=\beta^+ + \beta^- \in L = L^+ \oplus L^-$ and $(\alpha,\alpha)=(\beta,\beta)=0$.  
\end{corollary}

\begin{proof}
If $(\alpha,\beta) \neq 0$ then $[e^{\alpha},e^{\beta}] \in (V_{NS})_{\alpha+\beta} \cap P^{0,0}$ which are all null states by no-ghost theorem because then $(\alpha+\beta,\alpha+\beta) \neq 0$.  Therefore we can assume $(\alpha,\beta) = 0$ and the result follows from a straight computation using residues of vertex operators.
\end{proof}

We specialize to $L = II_{2,2} = \Z \rho_1 \oplus \Z \rho_2 \oplus \Z \rho'_1 \oplus \Z \rho'_2 = L^+ \oplus L^-$ with $(\rho_i,\rho_j) = (\rho'_i,\rho'_j) = 0$ and $(\rho_i,\rho'_j) = \delta_{i,j}$.  For $\alpha \in L$, write
$$\alpha = (a,b,c,d) = a \rho_1 + b \rho_2 + c \rho'_1 + d \rho'_2.$$
For each non-zero $r = \frac{p}{q} \in \Q$, where $p,q$ are relatively prime and $q > 0$, define
$$u_r = (p,0,0,q), v_r = (0,-p,q,0).$$
It is clear that $(u_r,u_r) = (v_r,v_r) = (u_r,v_r) = 0$ hence also
$$(m u_r + n v_r, k u_r + l v_r) = 0.$$
However, $(u_r^+,v_r^-) = pq$ and $(v_r^+,u_r^-) = -pq$ therefore
$$((m u_r + n v_r)^+, (k u_r + l v_r)^-) = (ml-nk)pq.$$
which results in the Lie bracket with a proper choice of cocycle $\epsilon(u_r,v_r)$,
\begin{equation} \label{bracketn=2a}
[e^{m u_r + n v_r}, e^{k u_r + l v_r}] = (ml - nk) pq e^{(m+k) u_q + (n+l) v_q}.
\end{equation}

Before we proceed further, let us clarify how this cocycle is chosen.  Denote the elements $h_1 = (1,0,0,0), h_2 = (0,0,0,1), h_3 = (0,1,0,0), h_4 = (0,0,1,0)$.  Then define $\epsilon(h_i,h_j) = 1$ for $i \leq j$ and $(-1)^{(h_i,h_j)}$ for $j < i$.  Extend this definition bilinearly (with multiplication in $\Z / 2 \Z$).  By remark 5.1.1 in \cite{voa}, $\epsilon(.,.)$ is a 2-cocycle and it clearly satisfies $\epsilon(u_r,v_r) = 1$.

Next, note that if $r \neq r'$ then $(u_r,u_{r'}) = (u_r^+,u_{r'}^-) = (v_r,v_{r'}) = (v_r^+,v_{r'}^-) = 0$ however $(u_r,v_{r'}) = p q' - qp' \neq 0$ and $(v_r,u_{r'}) = -p q' + q p'$.  In particular,
$$(m u_r + n v_r, k u_{r'} + l v_{r'}) = (ml-nk)(pq' - qp').$$
Therefore, if $ml-nk = 0$, we consider the case
$$(m u_r^+ + n v_r^+, k u_{r'}^- + l v_{r'}^-) = ml (pq') - nk (pq') = ml(pq') - ml (pq') = 0,$$
which results in the Lie bracket
\begin{equation}\label{bracketn=2b}
[e^{m u_r + n v_r}, e^{k u_{r'} + l v_{r'}}] = 0.
\end{equation}

Now suppose that $\alpha = (a,b,c,d) \in L$ satisfies $(\alpha,\alpha) = 0$.  Then there are three possibilities.  Either $a=b=0$, $c=d=0$ or neither of these happen.  In the first two cases, we then have that either $\alpha^+ = 0$ or $\alpha^- = 0$ and then it is clear that $e^{\alpha}$ lies in the center of $g^{(2)}_{NS}$.  Therefore, let us consider the third possibility.  In this case, to have $(\alpha,\alpha) = 0$ we either have $ad \neq 0$ or $bc \neq 0$.  If both $ad$ and $bc$ are non-zero then the condition $ac + bd = 0$ implies that $a/d = -b/c$.  Therefore, we find that $\alpha$ is a linear combination with integer coefficients of $u_r,v_r$ for some unique $r$ given by $a/d$ or $-b/c$, whichever is defined.

We have now proven the following :
\begin{theorem} \label{corollarybracketn=2}
If $L = II_{2,2}$, the structure of $g^{(2)}_{NS}$ is given by
$$g^{(2)}_{NS} = \C e^0 \oplus A_0 \oplus A_{\infty} \oplus \bigoplus\limits_{r \in \Q^\times} A_r$$
where $A_r = \spn\{e^{m u_r + n v_r}\}$ for $r \in \Q^\times$ is the Lie algebra generated by $e^{\pm u_r}, e^{\pm v_r}$ with Lie bracket 
$$[e^{m u_r + n v_r}, e^{k u_r + l v_r}] = (ml - nk) pq e^{(m+k) u_q + (n+l) v_q},$$
each of these subalgebras commute with each other and
$$A_0 = \spn \{ e^\lambda , \lambda^+ = 0, \lambda^- \neq 0 \}, A_{\infty} = \spn \{ e^{\lambda}, \lambda^- = 0, \lambda^+ \neq 0 \}.$$
\end{theorem}

We return to Example \ref{matrixrealization}, where we realized $\lambda \in II_{2,2}$ with $(\lambda,\lambda) = 0$ as a matrix of determinant $0$ :
$$\lambda = \begin{bmatrix}
m_1 & n_2 \\
m_2 & n_1
\end{bmatrix}, \lambda^+ = \begin{bmatrix}
m_1 & 0 \\ m_2 & 0
\end{bmatrix}, \lambda^- = \begin{bmatrix}
0 & n_2 \\ 0 & n_1
\end{bmatrix}.$$
With this realization, we now have
$$u_r = \begin{bmatrix}
p & q \\
0 & 0
\end{bmatrix}, v_r = \begin{bmatrix}
0 & 0 \\
p & q 
\end{bmatrix}.$$

We discuss two different actions of the group $SL(2,\Z)$ on these matrices.  Given a matrix $M  =\begin{bmatrix}
a & b \\ c & d
\end{bmatrix} \in SL(2,\Z)$ and matrices $\lambda,\lambda_1,\lambda_2$ of determinant $0$, it is clear that $M \lambda$ and $\lambda M^t$ are once again matrices of determinant $0$.  
Therefore, from the formula for the bilinear form in Example \ref{matrixrealization}, we have
$$(M \lambda_1, M \lambda_2) = (\lambda_1 M^t, \lambda_2 M^t) = (\lambda_1, \lambda_2).$$
Now write
$$\lambda_1 = \begin{bmatrix}
a_1 & a_2 \\ a_3 & a_4
\end{bmatrix}, \lambda_2 = \begin{bmatrix}
b_1 & b_2 \\ b_3 & b_4
\end{bmatrix}.$$
We will examine the left action first.
$$(M \lambda_1)^+ = \begin{bmatrix}
a a_1 + b a_3 & 0 \\
c a_1 + d a_3 & 0
\end{bmatrix}, (M \lambda_2)^- = \begin{bmatrix}
0 & a b_2 + b b_4 \\
0 & c b_2 + d b_4
\end{bmatrix}.$$
From there, we see that
$$((M \lambda_1)^+, (M \lambda_2)^-) = (ad - bc) a_1 b_4 - (ad - bc) (b_2 a_3) = a_1 b_4 - b_2 a_3 = (\lambda_1^+, \lambda_2^-).$$
Therefore, the left action of $SL(2,\Z)$ on $II_{2,2}$ lifts to an isomorphism in $g^{(2)}_{NS}$.  Furthermore, using generators $S,T$ of $SL(2,\Z)$, we see that each preserve the subspaces $A_r, r \in \Q \cup \infty$.

Next we consider the right action.  Given $r = p/q$, observe that
$$u_r M^t = \begin{bmatrix}
p & q \\ 0 & 0
\end{bmatrix} \begin{bmatrix}
a & c \\ b & d
\end{bmatrix} = \begin{bmatrix}
ap+bq & cp+dq \\ 0 & 0
\end{bmatrix} = u_{r'}$$
for $r' = \frac{ap+bq}{cp+dq}$.  Similarly, $v_r M^t = v_{r'}$.  In this case, it is not true that \break $((\lambda_1 M^t)^+, (\lambda_2 M^t)^-) = (\lambda_1^+, \lambda_2^-)$.  As a counter-example, take $r=1$ and $S = \begin{bmatrix}
0 & 1 \\
- 1 & 0 \\
\end{bmatrix}$.  Then $r' = -1$, however $(u_1^+,u_1^-) = 1$ but $(u_{-1}^+, u_{-1}^-) = -1$.

Recall the well-known action of $SL(2,\Z)$ on $\C$ given by
$$\begin{bmatrix}
a & b \\ c & d
\end{bmatrix} \cdot z = \frac{a z + b}{cz + d}.$$
If $z = r = \frac{p}{q}$ we may rewrite this action as
$$\frac{ap/q + b}{cp/q + d} = \frac{ap + bq}{cp + dq} = r'.$$
We then see that these two actions match and we know it to be transitive on $\Q$.  

We have therefore proven :

\begin{proposition}
Each left action of $SL(2,\Z)$ on $II_{2,2}$ lifts to an isomorphism in $g^{(2)}_{NS}$ which preserves each $A_r, r \in \Q \cup \infty$.  Each right action of $SL(2,\Z)$ on $II_{2,2}$ induces vector space isomorphisms $A_r \simeq A_{r'}$ for $r,r' \in \Q$.
\end{proposition}

One may note however that $A_r \simeq A_{r'}$ whenever $r,r' \in \Q^{\times}$ as Lie algebras, however this isomorphism requires rescaling of the basis elements and therefore does not appear to be the natural extension of the right action of $SL(2,\Z)$ on $II_{2,2}$.

Because $g^{(2)}_{NS}$ is defined as a quotient of $\widetilde{P}^{0,0}$, it is not a subalgebra of $\widetilde{P}^{0,0}$.  However, we still have the following :
\begin{proposition}
Suppose $r \in \Q \cup \infty$.  Then each $A_r$ also exists as a subalgebra of $\widetilde{P}^{0,0}$ in the obvious way.
\end{proposition}
\begin{proof}
Examine the proof of Corollary \ref{corollarybracketn=2} and observe that the quotient is not needed to study each $A_r$.
\end{proof}

\subsection{Lower bound for multiplicities of the Lie algebra of physical states in $N=2$}
We may attempt to imitate the DDF construction from the $N=0$ and $N=1$ cases in $N=2$.  Unfortunately we do not obtain a perfect analogy and there is no spectrum-generating algebra for $g^{(2)}_{NS}$.  However, we may still say something about $\widetilde{P}^{0,0}$.  Fix $\alpha \neq 0 \in L$ and assume $(\alpha,\alpha) < 0$.  Then we write
$$\alpha = \alpha^+ + \alpha^-$$
where $\alpha^+ \in L^+$ and $\alpha^- \in L^-$.  As before, we have the Lemma :
\begin{lemma}
Suppose $(\alpha,\alpha) < 0$.
There exists $c = c^{+} + c^- \in L^+ \oplus L^-$ such that $(c,c) = 0$ and $(c^+,\alpha) = (c^-,\alpha) = -1/2$.  In particular, $(c,\alpha) = -1$.
\end{lemma}
\begin{proof}
Let $k = -(\alpha,\alpha)$.  Choose $h^+,h^- \in L^+,L^-$ such that $(h^+,\alpha) = (h^-,\alpha) = 0$ and $(h^+,h^-) = 1/4k$.  Then the element
$$c = \frac{\alpha^+}{2k} + \frac{\alpha^-}{2k} + h^+ + h^-$$
satisfies all the desired properties.
\end{proof}
Although $e^c$ is not an element of $V_{NS}$ in general (because $c$ is not an element of $L$), we may think of it as an element of some larger space $W$ which contains $V_{NS}$, and vertex operators of the form $Y(e^c,z)$ are still well-defined on $\bigoplus\limits_{M \in \Z} (W)_{\alpha + Mc}.$

Therefore, let
$$\widetilde{G} = G^+_{-\frac{1}{2}}G^-_{-\frac{1}{2}} - G^-_{-\frac{1}{2}}G^+_{-\frac{1}{2}}$$
and define
$$A_m = (\widetilde{G} e^{mc})_0.$$
Because $L$ is even, there exists a number $M$ such that $(\alpha - Mc, \alpha - Mc) = 0$.  It is easy to see that $M = -(\alpha,\alpha)/2$.  Then we may consider elements of the form
$$A_{m_1}^{\lambda_1} ... A_{m_k}^{\lambda_k}  e^{\alpha - Mc}$$
These elements lie in $\widetilde{P}^{0,0}(\alpha)$ for $\lambda_1 m_1 + ... + \lambda_k m_k = M$.  There will then be
$$p(-(\alpha,\alpha)/2)$$
such elements, and we prove that they are linearly independent.

\begin{proposition}
The elements
$$A_{m_1}^{\lambda_1} ... A_{m_k}^{\lambda_k} e^{\alpha - Mc}$$
are linearly independent.
\end{proposition}
\begin{proof}
First, observe that all $A_{m}$'s commute with each other, but also that they commute with any (super)polynomial in $c^+,c^-$ that would be generated by applying $A_m$ to $e^{\alpha - Mc}$.  Therefore, each of these operators $A_{m}$'s is essentially multiplication by a specific (super)polynomial on $e^{\alpha - Mc}$, with the added effect of replacing $e^{\alpha-Mc}$ with $e^{\alpha-Mc + mc}$.  Let us examine these polynomials for a moment.

Fix $A_m$ for now.  Then
$$\widetilde{G} e^{mc} = m (c^+(-1) - c^-(-1)) + 2m^2 (G^+_{\frac{1}{2}} c(-1)) (G^-_{\frac{1}{2}} c(-1)) e^{mc}$$
Using definition of vertex operators, we see that $(mc,\alpha - Mc) = -m$ induces a power $z^{-m}$ in the computation and there are no power of $z$ originating from $c^+(0) - c^-(0)$ because $(c^+ - c^-,\alpha) = 0$.  In particular, this results in a term of the form $c^+(-m) - c^-(-m)$ in the corresponding polynomial, among other terms.

Therefore the polynomial obtained from
$$A_{m_1}^{\lambda_1} ... A_{m_k}^{\lambda_k} e^{\alpha - Mc}$$
will contain a term of the form
$$(c^+(-m_1) - c^-(m_1))^{\lambda_1}...(c^+(-m_k) - c^-(m_k))^{\lambda_k}.$$
and other complicated terms, which we ignore.  By a maximality argument, we then see that terms of the form $A_{m_1}^{\lambda_1} ... A_{m_k}^{\lambda_k} e^{\alpha - Mc}$ are linearly independent in $V_{NS}$.  Furthermore, they are also linearly independent in $\widetilde{P}^{0,0}$.
\end{proof}
From this proposition, we then obtain the following corollary :
\begin{corollary}
Suppose $\alpha \in L$.  Then $\dim \widetilde{P}^{0,0}(\alpha) \geq p(-(\alpha,\alpha)/2)$.
\end{corollary}
Unfortunately, due to the complexity of the ideal it appears difficult to find exact multiplicities for $\widetilde{P}^{0,0}$.

\section{Jacobi forms and Borcherds products}

The main goal of this section is to describe the technique developed by Borcherds in \cite{borcherds2} and \cite{borcherds3} known as \emph{Borcherds products.}  We will use these techniques to study the characters of the fake monster Lie algebra that we described in section \ref{sectionoghost} from a different perspective.  We simplify much of the statements and theorems from \cite{borcherds3} in the process, because we do not need to work with lattices that are not self-dual.

\subsection{Niemeier lattices}

\begin{definition} Let $L$ be an integer-valued lattice and $x_1,...,x_n$ be a $\Z$-basis for $L$.  Then we say $L$ is \emph{unimodular} if the determinant of the matrix $(a_{ij})_{i,j=1}^n$ where $a_{ij} = (x_i,x_j)$ is $1$ or $-1$.
\end{definition}

\begin{lemma}
An integer-valued lattice is unimodular if and only if it is self-dual.
\end{lemma}
\begin{proof} Because the lattice is integer-valued, it is clear that $L \subseteq L^*$.  Let $G = (a_{ij})_{i,j=1}^n$ as in the definition.  Suppose $x_1,...,x_n$ is a basis for $L$ and $y_1,...,y_n$ is a dual basis for $L^*$.  Then there exists an invertible $n \times n$ matrix $S$ with integer entries such that $x_i = \sum\limits_{k=1}^n s_{ik} y_k$.  We also have $y_i = \sum\limits_{k=1}^n t_{ik} x_k$, but the $t_{ik}$'s may not be integers.  Therefore,
$$\det(S^{-1})\det(G) = \det(S^{-1}G) = \det((x_i,y_j))_{i,j=1}^n = 1.$$
If $L$ is unimodular then $\det(G) = \pm 1$ by assumption and then $\det(S^{-1}) = \pm 1$ as well, hence the entries of $S^{-1}$ are also integers and $L$ is self-dual.  If $L$ is self-dual then the entries of $S^{-1}$ are integers hence $\det(S^{-1}) = \pm 1$ hence $\det(G) = \pm 1$.
\end{proof}

Consider the lattice $$II_{m,n} = \left\lbrace (a_1,...,a_{m+n} | a_i \in \Z \right\rbrace \cup \left\lbrace a_i,...,a_{m+n} | a_i \in \Z + 1/2 \right\rbrace$$ with inner product inherited from $\Z^{m,n}$.  It is a simple exercise to show that $II_{m,n}$ is unimodular if and only if $m-n$ is divisible by $8$.

\begin{theorem}
Suppose $L$ is an even unimodular lattice of signature $(m,n)$ with $m,n \geq 1$.  Then $(m-n)$ is divisible by $8$ and $L \simeq II_{m,n}$.
\end{theorem}

We will need the following classification theorem from \cite{niemeier} :

\begin{theorem}
There are exactly 24 positive-definite unimodular even lattice of rank $24$.  They are known as Niemeier lattices.
\end{theorem}

For a lattice $L$, denote by $R(L)$ the set of all elements of $L$ of square length $2$.  We call $R(L)$ the root system of $L$.  Niemeier lattices have been classified by their root systems.  Indeed, we have :
\begin{flalign*}
& \emptyset : \text{The Leech lattice } \Lambda,\\
& A_1^{24}, A_2^{12}, A_3^8, A_4^6, A_8^3, A_{12}^2, A_{24}, \\
& D_4^6, D_6^4, D_8^3, D_{12}^2, D_{24}, \\
& E_6^4, E_8^3, \\
& A_5^4D_4, A_7^2 D_5^2, A_9^2 D_6, A_{15} D_9, E_8 D_{16}, E_7^2 D_{10}, E_7 A_{17}, \\
& E_6 D_7 A_{11}
\end{flalign*}

If $L$ is a Niemeier lattice, then $L \oplus II_{1,1}$ is a unimodular lattice of signature $(25,1)$ and therefore isomorphic to $II_{25,1}$.  Conversely, suppose $\rho$ is an isotropic element of $II_{25,1}$.  Because $II_{25,1}$ is self-dual, there exists an element $\rho' \in II_{25,1}$ such that $(\rho,\rho') = 1$.  We may choose $\rho'$ to be isotropic by adding a suitable multiple of $\rho$ to $\rho'$.

Then the Gram matrix of $II_{25,1}$ can be written as
$\begin{bmatrix}
A & 0 \\
0 & J
\end{bmatrix}$
where $J = \begin{bmatrix}
0 & 1 \\
1 & 0
\end{bmatrix}$ and $A$ is the Gram matrix of the lattice $\rho^\perp / \rho$.  In particular, $\rho^\perp/\rho$ is a positive-definite unimodular even lattice of rank $24$, hence a Niemeier lattice.
We conclude :

\begin{proposition} \label{niemeiersublattice}
Up to automorphisms of $II_{25,1}$, there are $24$ orbits of isotropic vectors $\rho \in II_{25,1}$.  In particular, these orbits are in a bijective correspondence with the Niemeier lattices via the map $\rho \rightarrow \rho^{\perp} / \rho$.
\end{proposition}

\subsection{Jacobi forms and Hecke operators}

Let $\He$ be the upper-half space and $L$ a positive definite unimodular lattice.  From now on, we will use the notation $e(x) = \exp(2 \pi i x)$ for brevity when applicable.

\begin{definition}
Let $k \in \frac{1}{2} \Z$.  A function $f : \He \to \C$ is called a modular form of weight $k$ for $SL_{2}(\Z)$ if
\begin{enumerate}
\item $f((a \tau + b)/(c \tau + d)) = (ct + d)^k F(\tau)$ for all $\begin{bmatrix}
a & b \\ c & d
\end{bmatrix} \in SL_2(\Z)$.
\item $f$ is holomorphic on $\He$.
\item $f$ is holomorphic at the cusp $\infty$.
\end{enumerate}
We say that $f$ is nearly holomorphic if it has a pole of finite order at $\infty$.
\end{definition}

The main example of a nearly holomorphic modular form that will interest us is the negative 24th power of the Dedekind eta function.  Write
$$\Delta(\tau) = \eta^{-24}(\tau) = e^{-2 \pi i \tau} \prod\limits_{n = 1}^{\infty} (1 - e^{2 n \pi i \tau})^{-24} = \sum\limits_{n \geq -1} c(n) e(n \tau),$$
where $c(n) \in \Z$.  Then $\Delta(\tau)$ is a nearly holomorphic modular form of weight $12$.

\begin{definition}
A Jacobi form of weight $k \in \Z$ and index $m \in \N$ for $L$ is a function
$$\phi : \He \times (L \otimes \C) \rightarrow \C$$
such that
\begin{align*}
\phi \left( \frac{a \tau + b}{c \tau + d}, \frac{z}{c \tau + d} \right) &= (c \tau + d)^k \exp \left(\pi i \frac{c m(z,z)}{c \tau + d} \right) \phi(\tau,z) \\
\phi(\tau,z+\lambda \tau + \mu) &= \exp(-\pi i m((\lambda,\lambda) \tau + 2 (\lambda,z))) \phi(\tau,z)
\end{align*}
for any $A = \begin{bmatrix}
a & b \\ c & d
\end{bmatrix} \in SL_2(\Z)$ and any $\lambda,\mu \in L$.  Furthermore, $\phi(\tau,z)$ has a Fourier expansion
$$\phi(\tau,z) = \sum\limits_{n \in \Z, l \in L} f(n,l) \exp(2 \pi i (n \tau + (l,z))).$$
If $n \geq 0$ then $\phi$ is said to be weak.  If $2mn - (l,l) \geq 0$ then $\phi$ is said to be holomorphic and if $2mn - (l,l) > 0$ then $\phi$ said to be cusp.  Furthermore, we say $\phi$ is weakly holomorphic if it is weak and holomorphic.
\end{definition}

The following was proven in \cite{gritsenko2}, which provides a way of defining a modular form from a Jacobi form and vice-versa.

\begin{lemma} \cite{gritsenko2}
Suppose $\phi$ is a holomorphic Jacobi form of weight $k \in \Z$ and index $1$ and with Fourier expansion as in the definition.  Then the coefficients $f(n,l)$ of $\phi$ depend only on the norm $2n - (l,l)$.  Furthermore, the function
$$\phi(\tau,z) = \sum\limits_{\substack{n \in \Z, l \in L \\ 2n - (l,l) \geq 0}} f(n,l) \exp(2 \pi i (n \tau + (l,z)) = \phi(\tau) \Theta_{L} (\tau,z)$$
where
$$\Theta_{L} (\tau, z)
= \sum\limits_{\lambda \in L} e(\tau (\lambda,\lambda)/2 + (\lambda,z)),$$
defines a modular form $\phi(\tau)$ by
$$\phi(\tau) = \sum\limits_{r \in \Z} f(r) q^r$$
with $f(r) = f(r,0)$.
\end{lemma}

\begin{proof}
We only prove the first claim.  The remaining claims are fairly lengthy, but straightforward computations.  Using the property
$$\phi(\tau,z+\lambda \tau) = \exp(-\pi i m((\lambda,\lambda) \tau + 2 (\lambda,z))) \phi(\tau,z),$$
we find with $m=1$ and some minor simplifications,
$$\phi(\tau,z) = \sum\limits_{n,l} f(n,l) e(\tau(n + (\lambda,l) + (\lambda,\lambda)/2) + (\lambda+l,z)).$$
Therefore, by comparing the terms of the series we find
$$f(n,l) = f(n+(\lambda,l) + (\lambda,\lambda)/2, l + \lambda).$$
If we specialize to $\lambda = -l$ we then obtain
$$f(n,l) = f(n - (\lambda,\lambda)/2, 0),$$
which proves the first claim.
\end{proof}

We will need to know about Hecke operators in this context.  The idea of using Hecke operators on characters of affine Lie algebras first appeared in \cite{feingoldfrenkel} but we will follow the definitions from \cite{zagier}.  If $\phi$ is a Jacobi form of weight $k$ and index $m$, we define a Hecke operator
\begin{align*}
&\phi|T_l (\tau, z) =  \\ &\sum\limits_{\substack{ \begin{pmatrix}
a & b \\ c & d
\end{pmatrix}  \in SL_2(\Z) / M_2(\Z) \\ ad-bc = l}} l^{k-1} (c \tau + d)^{-k} \exp \left( ml \frac{-cz^2}{c \tau + d} \right) \phi \left( \frac{a \tau + b}{c \tau + d} , \frac{lz}{c \tau + d} \right)
\end{align*}

where $M_2(\Z)$ is the set of $2 \times 2$ matrices with coefficients in $\Z$.  $T_l$ is a Hecke operator of type $V_l$ from $\cite{zagier}$, with which they have shown the following (originally from \cite{feingoldfrenkel} in a different form) : 
\begin{proposition}
Suppose $\phi$ is a Jacobi form of weight $k$ and index $m$.  Then $\phi|T_{l}(\tau,z)$ is a Jacobi form of weight $k$ and index $ml$.
\end{proposition}
and the theorem :
\begin{theorem} \label{hecketransformation}
Suppose $L = \Z$ and $\phi$ is a Jacobi form of weight $k$ and index $m$ such that $\phi = \sum\limits_{n,r} c(n,r) q^n \xi^r$ where $q^n = exp(2 \pi in \tau)$ and $\xi^r = \exp(2 \pi i r z)$.  Then
$$\phi|T_{l} = \sum\limits_{n,r} \sum\limits_{a|(n,r,l)} a^{k-1} c\left(\frac{nl}{a^2},\frac{r}{a}\right) q^n \xi^r$$
\end{theorem}
We will use these results later.

\begin{remark}
Up to this point we have stated all our results for positive-definite lattices or lattices of signature $(d-1,1)$.  However from now on we will use the equivalent statements for negative definite lattices and lattices of signature $(1,d-1)$.  The author believes that it is more natural to work with positive-definite lattices in the more physical and algebraic setting.  However the literature in the context of Jacobi form and Borcherds products is more consistent with negative-definite lattices.  This is a trivial change however, because every appearance of the bilinear form $(.,.)$ may be replaced with $-(.,.)$ to switch from one setting to the other.
\end{remark}

Now, we denote $q(\lambda) = \lambda^2/2 = (\lambda,\lambda)/2$ for brevity.  Let $M = II_{1,25} \oplus II_{1,1}$.  Write an element of $M$ in the form $(\lambda,m,n)$ with square norm $\lambda^2 - 2mn$.  Then we let $z = (0,0,1)$ and $z' = (0,-1,0)$ (so that $(z,z') = 1$).  Also, we write $K = II_{1,25}$.

The complex upper-half space can be generalized in the following way, based on pages 46-47 from \cite{borcherds3} :  let $D_n = \{Z = X + iY \in K \otimes \C, q(Y) > 0 \}$.  Consider the set
$$N= \{[Z_M] \in P(M \otimes \C)); (Z_M,Z_M) = 0 \}$$
where $Z_M \in M \otimes \C$ and $P(M \otimes \C)$ is the projective space of $M \otimes \C$.  Then we have the subset
$$N' = \{[Z_M] \in N, (Z_M, \overline{Z_M}) > 0 \}$$
Let $P$ be the connected component of $N'$ of the identity under the action of $O(M)$.  In particular, if $Z \in D_n$, we have a biholomorphic map
$$Z \to Z_M = (Z, 1, -q(Z)) \in N'$$
We define $\He_n$ to be the component of $D_n$ mapped into $P$.
In particular, given an automorphic form $\psi_M$ on $P$ we can define automorphic form $\psi_z$ on $\He_n$ in the following way :
$$\psi_Z(Z) = \psi_M(Z_M) = \psi_M((Z,1,-q(Z))).$$  Finally, if $Z = X + iY \in \He_l$, we write
\begin{align*}
&X_M = (X,1,q(Y) - q(X)) \\
&Y_M = (Y,0,-(X,Y))
\end{align*}
We think of $X_M, Y_M$ as an oriented base for $v$ a $2$-dimensional positive subspace of $M$.

Because $M$ has signature $(2,26)$, choose $v$ to be some $2$-dimensional positive definite subspace of $M$.  We denote by $v^\perp$ the orthogonal complement of $v$.  Given $\lambda \in M \otimes \R$ we now have a decomposition $\lambda = \lambda_{v} + \lambda_{v^{\perp}}$.  Let $w = v \cap (z_v)^{\perp}$ and $w^{\perp} = v^{\perp} \cap (z_{v^\perp})^{\perp}$.  We also have a decomposition
$$M \otimes \R = w \oplus \R z_{v} \oplus w^{\perp} \oplus \R z_{v^\perp}.$$
We have a special element
$$\mu = -z' + \frac{z_v}{2 z_v^2} + \frac{z_v^\perp}{2 z_{v^\perp}^2}.$$
For an arbitrary lattice $Q$ of signature $(b^+,b^-)$, define
\begin{align*}
\Theta_{Q} (\tau, u)
= \sum\limits_{\lambda \in Q} e(\tau \lambda^2_{u}/2 + \overline{\tau} \lambda^2_{u^\perp} / 2)
\end{align*}
and 
\begin{align*}
\Theta_{Q} (\tau, u;r,t)
= \sum\limits_{\lambda \in Q} e(\tau q((\lambda + t)_u) + \overline{\tau} q((\lambda+t)_{u^\perp}) - (\lambda + t/2,r))
\end{align*}
where $r,t \in Q \otimes \R$ and $u$ is a $b^+$-dimensional positive-definite subspace of $Q$.  If $r = t = 0$, we will simply write $\Theta_Q(\tau,u)$.  We will consider $\Theta_M(\tau,v)$ where $v$ is as before and $\Theta_K(\tau,w,r,t)$ where $w$ is a one-dimensional positive definite subspace of $K$.

  Next, let $F$ be a modular form of weight $1-b^-/2$ and we consider the integral
$$\Phi_M (v, F) = \int\limits_{SL_2(\Z) / \He} \overline{\Theta}_M(\tau;v) F (\tau) y^{- 1} dx dy$$
This integral does not in general converge and must be regularized.  An examination of the singularities of $\Phi_M(v,F)$ reveals that they are all logarithmic.  Borcherds proved the following theorem in \cite{borcherds3}, which we simplify here for our purposes :

\begin{theorem}\label{borcherdsproduct}  With all the notation we have set up, there exists a meromorphic function $\Phi_M(Z_M,F)$, $Z_M \in P$ with the following properties :
\begin{enumerate}
\item $\Phi_M$ is an automorphic form of weight $1/2$ for $O(M)$.
\item $\log| \Psi_Z(Z,F)| = \frac{- \Phi_M (Z_M,F)}{4} - \frac{1}{2} (\log |Y_M| + \Gamma'(1)/2 + \log \sqrt{2 \pi})$
where $\Gamma'(1) = \lim\limits_{n \to \infty} (\log(n) - 1/1 - 1/2 - ... - 1/n)$ is also known as Euler's constant.
\item $\Psi_Z(Z,F)$ has an infinite product expansion 
$$e(Z, \rho) \prod\limits_{\substack{\lambda \in K \\ (\lambda,W) > 0}} (1 - e((\lambda, Z)))^{c(q(\lambda))}$$
\end{enumerate}
\end{theorem}

The proof of this theorem in \cite{borcherds3} is done in a more general setting, which introduces a lot of complexity.  However, in our setting much of the complexity disappears and the lemmas and computations become easier to follow.  The remainder of this section is dedicated to the proof of this theorem in this simplified setting, which may be skipped, as well as important examples.

\subsection{Reduction to smaller lattice}

We first state an important theorem about theta functions.
\begin{theorem} \label{thetatransform} \cite{borcherds3}, Theorem 4.1
Let $Q$ be any lattice of signature $(b^+, b^-)$ and $u$ a $b^+$ dimensional positive definite subspace of $N$.  Then
\begin{align*}
& \Theta_Q((a \tau + b) / (c \tau + d),u; a \alpha + b \beta, c \alpha + d \beta) \\
&= (c \tau + d)^{b^+/2} (c \overline{\tau} + d)^{b^-/2} \Theta_Q(\tau,u; \alpha, \beta).
\end{align*}
\end{theorem}
\begin{proof}
It is enough to check this relation for the two generators $T,S$ of $SL_2(\Z)$, but it is trivial for $T$.  Therefore, we prove that
\begin{align*}
\Theta_Q(-1 / \tau,u; - \beta, \alpha) = \sqrt{\tau / i}^{b^+} \sqrt{i \overline{\tau}}^{b^-} \Theta_Q(\tau,u; \alpha, \beta)
\end{align*}
The Fourier transform of
$$\sqrt{\tau / i}^{b^+} \sqrt{i \overline{\tau}}^{b^-} e( (-1 / \tau) q((\lambda + \alpha)_{u}) + (-1 / \overline{\tau}) q((\lambda + \alpha)_{u^\perp}) - (\lambda + \alpha/2, \beta))$$ is
$$e( \tau q((\lambda + \beta)_{u}) + \overline{\tau} q((\lambda + \beta)_{u^\perp}) - (x + \beta/2, \alpha))$$
and the proof follows by the Poisson summation formula.
\end{proof}

As stated before, $\Phi$ must be regularized.  We achieve this in the following way.  Let
$$F_{\infty} = \{ \tau \in C | |\tau| \geq 1, |\re \tau | \leq 1/2 \}$$
be the usual fundamental domain for $SL_2(\Z)$.   Then define
$$F_{w} = \{ \tau \in F_{\infty}, \im(\tau) \leq w \}.$$
It is clear that $\lim\limits_{w \to \infty} F_w = F_{\infty}$.  Then let $s$ be an arbitrary complex number and we consider
$$\lim\limits_{w \to \infty} \int\limits_{F_w} \overline{\Theta}_M(\tau,v) F(\tau) dx dy / y^{1 + s}$$
as a function of $s$, which can be analytically continued to a function for all complex numbers $s$.  Then we define $\Phi_M(v,F)$ to be the constant term of the Laurent expansion of this analytic continuation at $s = 0$.

Similarly, we also define $\Phi_K(v,F)$ to be the constant term of the Laurent expansion of the analytic continuation at $s= 0$ of
$$\lim\limits_{w \to \infty} \int\limits_{F_w} \overline{\Theta}_K(\tau,v) F(\tau) dx dy / y^{3/2 + s}.$$

Because we are integrating over a fundamental domain, we should check that the integral does not depend on the choice of fundamental domain.  We verify this for $\Phi_M$ and the proof for $\Phi_K$ is similar.

\begin{proposition}
Suppose $\gamma \in \Gamma$.  Then
$$\overline{\Theta}_M (\gamma \tau,v) F(\gamma \tau) \Im(\gamma \tau) = \overline{\Theta}_M (\tau,v) F(\tau) y $$
In particular,
$$\int\limits_{F_{\infty}} \overline{\Theta}_K (\gamma \tau,v) F(\gamma \tau) \Im(\gamma \tau) dx dy/y^2 = \int\limits_{F_\infty} \overline{\Theta}_M(\tau,v) F(\tau) dx dy / y$$
\end{proposition}

\begin{proof}
One could check independently for $T,S \in SL_2(\Z)$ or directly for an arbitrary matrix $\gamma = \begin{bmatrix}
a & b \\ c & d
\end{bmatrix} \in SL_2(\Z)$.  We do the latter.   We know that $$F(\gamma \tau) = (c \tau + d)^{-12} F(\tau)$$ and that $\Im(\gamma \tau) = \frac{y}{|c \tau + d|^2}$ and we also know how $\Theta_M$ transforms by Theorem \ref{thetatransform}.  Therefore, we obtain
\begin{align*}
&\overline{\Theta}_M (\gamma \tau,v) F(\gamma \tau) \Im(\gamma \tau)   \\
&= (c \overline{\tau} + d) (c \tau + d)^{13} \overline{\Theta}_M (\tau,v) (c \tau + d)^{-12} F(\tau) y / |c \tau + d|^2
\\ &= \overline{\Theta}_M (\tau,v) F(\tau) y
\end{align*}
and the second statement follows from a standard fact on hyperbolic measure.
\end{proof}

We next proceed with important lemmas and theorems.

\begin{lemma} (\cite{Bruinier}, Lemma 2.3, \cite{borcherds3}, Lemma 5.1)
\begin{align*}
\Theta_M (\tau, v) &= \frac{1}{\sqrt{2 y z^2_v}} \sum\limits_{\lambda \in M / z} \sum\limits_{d \in \Z} e(\tau q(\lambda_w) + \overline{\tau} q(\lambda_{w^\perp})) \cdot \\ &\cdot e \left( - \frac{|(\lambda,z) \tau + d|^2}{4 i y z^2_v} - \frac{d(\lambda,z_v - z_{v^\perp})}{2 z^2_v} \right)
\end{align*}
\end{lemma}

\begin{proof}
Write 
$$g(\lambda,v;d) = e(\tau q((\lambda + dz)_v) + \overline{\tau}q((\lambda+dz)_{v^\perp}))$$
Then
$$\Theta_M(\tau,v) = \sum\limits_{\lambda \in M/z} \sum\limits_{d \in \Z} g(\lambda,v;d)$$
We apply the Poisson summation formula to obtain
$$\Theta_M(\tau,v) = \sum\limits_{\lambda \in M/z} \sum\limits_{d \in \Z} \hat{g}(\lambda,v;d).$$

So it is enough to compute the fourier transform of the functions $g(\lambda,v;d)$.  Since they are very close to Gaussian functions, this is straightforward and results in
$$\hat{g}(\lambda,v;d) = \frac{1}{\sqrt{2 y z_v^2}} e \left( \tau q(\lambda_w) + \overline{\tau} q(\lambda_{w^\perp}) - \frac{d(\lambda,z_v - z_{v^\perp})}{2 z_v^2} - \frac{|d + (\lambda,z) \tau|^2}{4 i y z_v^2} \right)$$
Which proves the result once we substitute $\hat{g}$ back in.
\end{proof}

We use this lemma to prove the following crucial theorem :

\begin{theorem} (\cite{Bruinier}, Theorem 2.4, \cite{borcherds3}, Theorem 5.2)\label{thetaK}
$$\Theta_M(\tau,v) = \frac{1}{\sqrt{2 y z_v^2}} \sum\limits_{c,d \in \Z} e \left( - \frac{|c \tau + d|^2}{4 i y z^2_v} \right) \Theta_K (\tau,w;d \mu, - c \mu)                                                                                                                                                                                                                                                                                                                                                                                                                                                                                                                                                                                                                                                                                                                                                                                                                                                                                                                                                                                                                                                                                                                                                                                                                                                                                                                                                                                                                                                                                                                                                  $$
\end{theorem}

\begin{proof}
Recall that
$$\mu = -z' + \frac{z_v}{2 z_v^2} + \frac{z_v^\perp}{2 z_{v^\perp}^2}.$$
For $\lambda \in M/z$, write $\lambda = \lambda_K + c z'$ for $c \in \Z$.  We can rewrite
\begin{align*}
& \Theta_M(\tau,v) = \frac{1}{\sqrt{2 y z^2_v}} \sum\limits_{\lambda \in M / z} \sum\limits_{d \in \Z} e(\tau q((\lambda_K+cz')_w) + \overline{\tau} q((\lambda_K+cz')_{w^\perp})) \times \\
& e \left( - \frac{|(\lambda_K + cz',z) \tau + d|^2}{4 i y z^2_v} - \frac{d(\lambda_K + cz',z_v - z_{v^\perp})}{2 z^2_v} \right).
\end{align*}
Because $(\lambda_K - c \mu)_w = (\lambda_K + cz')_w$ and $(\lambda_K - c \mu)_{w^\perp} = (\lambda_K + cz')_{w^\perp}$, it is enough to check that
$$- \frac{|(\lambda_K + cz',z) \tau + d|^2}{4 i y z^2_v} - \frac{d(\lambda_K + cz',z_v - z_{v^\perp})}{2 z^2_v} = - \frac{|c \tau + d|^2}{4 i y z^2_v} - (\lambda_K - c \mu/2, d \mu)$$
We can easily see that the first term on each sides are equal, because $(\lambda_K,z) = 0$ and $(z',z) = 1$.  So it remains to check that
$$- \frac{d (\lambda_K + cz', z_v - z_{v^\perp})}{2 z_v^2} = - (\lambda_K - c \mu/2, d \mu)$$
which is a straightforward check from the definition of $\mu$ and  the fact that $z_v^2 = - z_{v^\perp}^2$, i.e.
\begin{align*}
&- (\lambda_K - c \mu/2, d \mu) \\ &= -d (\lambda_K +cz'/2 - c (z_v/4 z_v^2 + z_{v^\perp}/ 4 z_{v^\perp}^2), - z' + z_v/2 z_v^2 + z_{v^\perp} / 2 z_{v^\perp}^2)
\\ &=  -d (\lambda_K + cz'/2, z_v/2 z_v^2 + z_{v^\perp} / 2 z_{v^\perp}^2 + cd/2 (z_v/2 z_v^2 + z_{v^\perp} / 2 z_{v^\perp}^2, -z')
\\ &= -d (\lambda_K + cz'/2, z_v/2 z_v^2 + z_{v^\perp} / 2 z_{v^\perp}^2 - d (c/2 z',z_v/2 z_v^2 + z_{v^\perp} / 2 z_{v^\perp}^2)
\\ &= - \frac{d (\lambda_K + cz', z_v - z_{v^\perp})}{2 z_v^2}
\end{align*}
as desired.
\end{proof}

\begin{lemma} \label{unfoldinglemma} With the notation established in this section,
\begin{align*}
&\frac{1}{\sqrt{2 z_v^2}} \int_{F_{\infty}} \sum\limits_{\substack{c,d \in \Z \\ |c| + |d| \neq 0}} F(\tau) e \left( - \frac{|c \tau + d|^2}{4 i y z_v^2} \right) \overline{\Theta}_K(\tau,w; d\mu, - c\mu) \frac{dx dy}{y^{3/2 + s}}\\
&= 
\frac{\sqrt{2}}{|z_v|} \int\limits_0^\infty \int\limits_0^1 \sum\limits_{n > 0} \exp \left( \frac{- \pi n^2}{2 z^2_v y} \right) F(\tau) \overline{\Theta}_K (\tau, w; - n \mu, 0) \frac{dx dy}{y^{3/2+s}}
\end{align*}
\end{lemma}

\begin{proof}
\begin{align*}
&\frac{1}{\sqrt{2 z_v^2}} \int_{F_{\infty}} \sum\limits_{\substack{c,d \in \Z \\ |c| + |d| \neq 0}} F(\tau) e \left( - \frac{|c \tau + d|^2}{4 i y z_v^2} \right) \overline{\Theta}_K(\tau,w ; d \mu, - c \mu) \frac{dx dy}{y^{3/2 + s}}\\
&= \frac{1}{\sqrt{2 z_v^2}} \int_{F_{\infty}} \sum\limits_{(c,d) = 1 } \sum\limits_{n \geq 1} F(\tau) e \left( - \frac{n^2 |c \tau + d|^2}{4 i y z_v^2} \right) \overline{\Theta}_K(\tau,w, n d \mu, - n c \mu) \frac{dx dy}{y^{3/2 + s}} \\
&= \frac{1}{\sqrt{2 z_v^2}} \int_{F_{\infty}} \sum\limits_{A \in \Gamma_{\infty} \backslash \Gamma_1} \sum\limits_{n \geq 1} F(A \tau) e \left(\frac{- \pi n^2}{2 z_v^2 \Im( A \tau)} \right) \cdot \\ & \cdot \overline{\Theta}_K(A \tau,w; n \mu, 0) \Im (A \tau)^{1/2} \frac{dx dy}{y^{2+s}}
\end{align*}
The last equality follows from transformation relations of $\Im (\tau), \Theta_K$ and $F(\tau) = \eta^{-24}$.  We now use the so-called "unfolding trick", replacing the fundamental domain with the set $[0,1] \times [0,\infty)$ to obtain : 
\begin{align*}
\frac{\sqrt{2}}{|z_v|} \int\limits_0^\infty \int\limits_0^1 \sum\limits_{n > 0} \exp \left( \frac{- \pi n^2}{2 z^2_v y} \right) F(\tau) \overline{\Theta}_K (\tau, w; - n \mu, 0) \frac{dx dy}{y^{3/2+s}}
\end{align*}
which is the desired result.
\end{proof}
With the help of Theorem \ref{thetaK} we can now relate $\Phi_M$ to $\Phi_K$ :
\begin{theorem} (\cite{Bruinier}, Theorem 2.15, \cite{borcherds3}, Theorem 7.1) \label{maintheorem}
With the notation of this section, $\Phi_M(v,F)$ is given by the constant term of the Laurent expansion at $s=0$ of the analytic continuation of
\begin{align*}
&\frac{1}{\sqrt{2 z^2_{v}}} \Phi_K (w,F) + \\
& \frac{\sqrt{2}}{|z_{v}|} \sum\limits_{n > 0} \sum\limits_{\lambda \in K} e((n \lambda, \mu))  \times \int_{y > 0} c(\lambda^2 / 2) \exp(\pi n^2 / (2 y z_v^2) - 2 \pi y \lambda^2_w) y^{-s-3/2} dy
\end{align*}
\end{theorem}
\begin{proof}
By Theorem \ref{thetaK},
$$\int\limits_{F_\infty} \overline{\Theta}(\tau;v) F(\tau) dx dx / y^{1 + s}$$
is equal to
\begin{align*}
\frac{1}{\sqrt{2 z_v^2}} \int_{F_{\infty}} \sum\limits_{c,d \in \Z} F(\tau) e \left( - \frac{|c \tau + d|^2}{4 i y z_v^2} \right) \overline{\Theta}_K(\tau,w, d \mu, - c \mu) \frac{dx dy}{y^{3/2 + s}}.
\end{align*}
The term when $c = d = 0$ corresponds exactly to $\frac{1}{\sqrt{2 z_v^2}} \Phi_K(w,F)$.  By Lemma \ref{unfoldinglemma} on the other terms we obtain
$$\frac{\sqrt{2}}{|z_v|} \int\limits_0^\infty \int\limits_0^1 \sum\limits_{n > 0} \exp \left( \frac{- \pi n^2}{2 z^2_v y} \right) F(\tau) \overline{\Theta}_K (\tau, w;n \mu, 0) \frac{dx dy}{y^{3/2+s}}$$

And the result of the theorem is obtained by inserting the definition of $F(\tau)$ and $\Theta_K (\tau, w)$ and carrying the integral over $dx$.  Indeed, recall that if $\tau = x + iy$,
\begin{align*}
\Theta_K (\tau, w ; n \mu, 0) &= \sum\limits_{\lambda \in K} e(\tau  q(\lambda_w) + \overline{\tau} q(\lambda_{w^\perp}) - (\lambda, n \mu) )
\\ &= \sum\limits_{\lambda \in K} e(iy q(\lambda_w) - iy q(\lambda_{w^\perp}) - (\lambda, n \mu) ) e(q(\lambda) x)
\end{align*}
and
$$F(\tau) = \sum\limits_{k \geq -1} c(k) e(k \tau) = \sum\limits_{k \geq -1} c(k) e(k x) e(ik y)$$
Inserting these sums into the expression and integrating over $x$ we see that the only term that survives is when $k = q(\lambda)$ (equality because we are inserting the complex conjugate $\overline{\Theta}_K$) which results after simplifications and term rearrangements in
\begin{align*}
& \frac{\sqrt{2}}{|z_{v}|} \sum\limits_{n > 0} \sum\limits_{\lambda \in K} e((n \lambda, \mu))  \times
\\ & \int_{y > 0} c(\lambda^2 / 2) \exp(- \pi n^2 / (2 y z_v^2) - 2 \pi y \lambda^2_w) y^{-s-3/2} dy
\end{align*}
as desired.  For example, one such simplification is $e(iy q(\lambda_w) - iy q(\lambda_{w^\perp})$ with $e(i k y)$ to result in $\exp(-2 \pi y \lambda^2_w)$.

\end{proof}

\subsection{Lorentzian lattices}

The previous section shows how to reduce the computations to a smaller lattice.  In our cases of interest, the theorem relates $\Phi_M$ to $\Phi_K$, and $K$ is a Lorentzian lattice.  However, it does not say anything about how to compute $\Phi_K$ itself.  This section is dedicated to this solving this problem.

The main theorem is :

\begin{theorem}
Suppose $K$ is Lorentzian.  Then there exists a vector $\rho(K,W,F)$ (which can be constructed) such that
$$\Phi_K(w,F) = 8 \sqrt{2} \pi (\overline{w},\rho(K,W,F))$$
where $\overline{w}$ is a unit vector representing $w$ and $W$ is a Weyl chamber (i.e. connected components of points over which $\Phi_K(w,F)$ is real analytic).
\end{theorem}

One needs to prove Theorem \ref{maintheorem} again for $\Phi_K$ to relate it to $\Phi_L$ where $L = N^-$ is the negative-definite version of a Niemeier lattice (i.e. the same underlying abelian group but replacing the bilinear form $(.,.)$ with $-(.,.)$).  Now, $N^{-}$ can be obtained from $x^\perp / x$ for some $x \in II_{1,25}$.  So we shall use this vector as our new $z$ for this section (where the previous $z$ does not appear).  We shall also choose $z' \in II_{1,25}$ to be an isotropic vector such that $(z,z')=1$.  We write elements of $II_{1,25}$ as $(\lambda,m,n)$ with norm $\lambda^2 - 2mn$ where $\lambda \in N^{-}$.  As before, we have a decomposition,
$$K \otimes \R = u \otimes \R z_{w} \oplus u^\perp \oplus \R z_{v^\perp}$$
however $u = w \cap (z_w)^{\perp} = 0$.  As a result, $\lambda_u = 0$ for all $\lambda \in L$ which simplifies everything dramatically ; indeed, because $N$ is negative-definite, $\Phi_L(u,F)$ can no longer depend on $u$ hence is constant and thus there is no need for further reduction.  We obtain the following equivalent result :

\begin{theorem} \label{lorentziantheorem}
With the notation of this section, $\Phi_K(w,F)$ is given by the constant term of the Laurent expansion at $s=0$ of the analytic continuation of
\begin{align*}
&\frac{1}{\sqrt{2 z^2_{w}}} \Phi_L (F) + \\
& \frac{\sqrt{2}}{|z_{w}|} \sum\limits_{n > 0} \sum\limits_{\lambda \in L} e((n \lambda, \mu))  \times \int_{y > 0} c(\lambda^2 / 2) \exp(- \pi n^2 / (2 y z_w^2) ) y^{-2-s} dy
\end{align*}
\end{theorem}

To go further in these computations, we require the following lemma :

\begin{lemma} \cite{borcherds3}, Lemma 7.3
\begin{align}
& \int_{y > 0} c(\lambda^2 / 2) \exp(- \pi n^2 / (2 y z_w^2)) y^{-2-s} dy \\
&= c(\lambda^2/2) \left( \frac{\pi n^2}{2 z_{w}^2} \right)^{-s - 1} \Gamma(s - 1)
\end{align}
Furthermore, the constant term of the Laurent expansion at $s = 0$ of this expression is
$$c(\lambda^2/2) \frac{2 z_w^2}{\pi n^2}$$
\end{lemma}
As a consequence, we can now write $\Phi_K(w,F)$ as
$$\frac{1}{\sqrt{2 z^2_{w}}} \Phi_L (F) +
 \frac{\sqrt{2}}{|z_{w}|} \sum\limits_{n > 0} \sum\limits_{\lambda \in L} e((n \lambda, \mu))  c(\lambda^2/2) \frac{2 z_w^2}{\pi n^2}$$

And we use the following identity (from Bernoulli polynomials) :
$$\sum\limits_{n \neq 0} \frac{e(n (\lambda,\mu))}{n^2} = 2 \pi^2 ((\lambda,\mu)^2 - (\lambda,\mu) + \frac{1}{6})$$
Using the fact that both $\lambda$ and $-\lambda$ appear in $L$ we can rewrite $\Phi_K(w,F)$ as
$$\frac{1}{\sqrt{2} |z_{w}|} \Phi_L (F) +
 2 \sqrt{2} \pi |z_w| \sum\limits_{\lambda \in L} c(\lambda^2/2) ((\lambda,\mu)^2 - (\lambda,\mu) + \frac{1}{6})$$
The sum is finite because only finitely many $\lambda \in L$ satisfy $c(\lambda^2/2) \neq 0$, since $L$ is negative definite.

Before we proceed further, recall that
$$\mu = - z' + \frac{z_v}{2 z_v^2} + \frac{z_{v^\perp}}{2 z^2_{v^\perp}}$$
and we can recover $\overline{w}$ from $\mu$ by
$$\overline{w} = (z, \overline{w}) \mu + (z, \overline{w}) z' + \frac{z}{2 (z, \overline{w})} = |z_w| \mu + |z_w| z' + \frac{z}{2 |z_w|}$$

This shows that $\Phi_K$ is in fact a rational function in $\overline{w}$ with denominator some power of $(z, \overline{w})$.  Now observe that the same reasoning can be followed by exchanging the roles of $z$ and $z'$, which results in a similar formula, but with denominator some power of $(z', \overline{w})$.  If this power turns out to be non-zero, we then obtain different singularities for $\Phi_K$, which is not possible by the Wall-Crossing formula (\cite{borcherds3}, Corollary 6.3).  Therefore, most of these terms must cancel (in fact, they form a so-called vector system).  In particular, we can then obtain the result of Theorem \ref{lorentziantheorem} by reading the linear terms of $\Phi_K$, which results in $\rho(M,W,F) = \rho + \rho_z z + \rho_{z'} z'$ where,
$$\rho = - \frac{1}{2} \sum\limits_{\substack{\lambda \in K \\ (\lambda, W) > 0}} c(\lambda^2/2) \lambda$$
$$\rho_{z'} = \text{ constant term of } \overline{\Theta}_L (\tau) F(\tau) E_2 (\tau) /24$$
$$\rho_{z} = \frac{1}{12} \sum\limits_{\substack{\lambda \in K \\ (\lambda, W) > 0}} c(\lambda^2/2)$$
where $E_n$ are Eisenstein series and $W$ is a Weyl chamber whose closure contains $z$.

Alternatively, because the sum is finite, it is also possible to simply write it out and observe that we indeed obtain the result of the theorem.

\subsection{Completing the proof}

The last thing we need to do is compute the integrals that appear in the statement of Theorem \ref{maintheorem}.  These depend on whether or not $\lambda_w = 0$.  We obtain the following two results.
\begin{lemma}
Suppose $\lambda_w \neq 0$.  Then
\begin{align*}
& \int_{y > 0} c(\lambda^2 / 2) \exp(- \pi n^2 / (2 y z_v^2) - 2 \pi y \lambda^2_w) y^{-2-s} dy \\
&=  c(\lambda^2/2)  \frac{|z_v| \sqrt{2}}{n} \exp \left( - \frac{2 \pi n |\lambda_w|}{|z_v|} \right)
\end{align*}
\end{lemma}

\begin{lemma}
Suppose $\lambda_w = 0$.  Then
\begin{align*}
& \int_{y > 0} c(\lambda^2 / 2) \exp(- \pi n^2 / (2 y z_v^2) ) y^{-2-s} dy \\
&= c(\lambda^2/2) \left(\frac{\pi n^2}{2 z_{v}^2} \right)^{-5/2} \Gamma(1/2)
\end{align*}
\end{lemma}

Once these integrals are computed, one can replace the terms in Theorem \ref{maintheorem} with some additional simplification to obtain
\begin{align*}
& 8 \pi (Y, \rho) +
\\ & 2 c(0) \Gamma(s+1/2) \pi^{-s-1/2} (2 z^{2}_v)^s \sum\limits_{n > 0} \frac{1}{n^{2s+1}} +
\\ & 2 \sum\limits_{\substack{\lambda \in K \\ \lambda \neq 0}} e((n \lambda X))  \exp(- 2 \pi n |(n,Y)|) \frac{1}{n} c (q(\lambda))
\end{align*}

Taking the constant term at $s=0$ with the help of \cite{borcherds3}, Lemma 13.2 we obtain
\begin{align*}
& 8 \pi (Y, \rho) +
\\ & c(0) (\log(z^2_{v}) - \Gamma'(1) - \log(2\pi)) +
\\& 4 \sum\limits_{\substack{\lambda \in K \\ \lambda \neq 0}}
- c(q(\lambda)) \log(1 - e((\lambda,Z))
\end{align*}
This proves statement 2-3 of the theorem by examination of the terms.

\subsection{Examples}
In this section, we present two examples of Borcherds products, and an extra example obtained by twisting the first example.

\paragraph{Fake Monster Lie algebra :} For the lattice $II_{1,25}$ and a choice of Weyl vector $\rho$, \cite{borcherds1}
$$e^\rho \prod\limits_{r \in \Pi^+} (1 - e^r)^{p_{24} (1 - r^2/2)} = \sum\limits_{\substack{ w \in W \\ n \in \Z}} \det(w) \tau(n) e^{w(n \rho)}$$
where $\rho$ is a choice of Weyl vector, $\Pi^+$ is the set of elements of $II_{25,1}$ that have negative inner product with a element of $II_{25,1}$ of negative norm, and $\tau(n)$ is Ramanujan's Tau function.  We will elaborate on this character in the next chapter.

\paragraph{Fake Monster Superalgebra :}  \cite{scheit}, Corollary 5.27, originally from Borcherds products in \cite{borcherds3}, Example 13.7 (level 2 cusp).
$$\prod\limits_{\alpha \in \Delta_+} \frac{(1 - e(\alpha))^{c(- \frac{1}{2} \alpha^2)}}{(1 + e(\alpha))^{c(- \frac{1}{2} \alpha^2)}} = 1 + \sum \alpha(\lambda) e(\lambda)$$
where $\alpha(\lambda)$ is the coefficient of $q^n$ of
$$\prod\limits_{m > 0} \left( \frac{1 - q^m}{1 + q^m} \right)^8$$
if $\lambda$ is $n$ times a primitive norm zero vector in the closure of the positive cone and zero otherwise, and $c(n)$ are given by
\begin{align*}
\sum c(n) q^n &= 8 \prod\limits_{m > 0} \left( \frac{1 + q^m}{1 - q^m} \right)^8 = 8 \eta( 2 \tau)^8 / \eta(\tau)^{16}
\end{align*}

\paragraph{Niemann's algebra $G_{23}$ :} \cite{niemann}  By twisting the first example, a BKM was constructed with the following character :  
\begin{align*}
e^{\rho} \prod\limits_{r \in L^+} (1 - e^r)^{p_{\sigma}(1 - r^2/2)} \prod\limits_{r \in 23 L^{*+}} (1-e^r)^{p_{\sigma} (1 - r^2/46)} \\ = \sum\limits_{w \in W^\sigma} \det(w) w \left( e^{\rho} \prod\limits_{ i > 0} (1 - e^{i \rho}) (1 - e^{23 i \rho} \right)
\end{align*}
where $L^+$ are positive roots, $\rho = (0,0,1) \in \Lambda \oplus II_{1,1}$, $\sigma \in \aut(\Lambda)$ of order $23$, $W^\sigma$ are elements of $W$ that commute with $\sigma$ and $p_\sigma(q) = q \eta(q)^{-1} \eta(23 q)^{-1}$.
There are many other similar examples in \cite{niemann}.

\section{Characters of the fake monster Lie algebra}
\subsection{From the Borcherds-Weyl-Kac denominator formula} \label{subsectioncharacter1}

In Section \ref{sectionoghost}, we constructed a Lie algebra $g_L$ based on an even lattice $L$.  We will now show following \cite{borcherds1} that a special choice of $g_L$ is a BKM-algebra, using Theorem \ref{generalizedkacmoodycharacterization}.  Indeed, let $L = II_{25,1}$, the unimodular rank $26$ lattice of signature $25,1$.  Then we may write
$$L = II_{25,1} = N \oplus II_{1,1} = N \oplus \Z \rho \oplus \Z \rho'$$
such that $\rho$ and $\rho'$ are isotropic and $(\rho,\rho') = 1$ and $N$ is a Niemeier lattice.  To show that $g_L$ is a BKM algebra, we fix $N = \Lambda$, the Leech lattice.  Write $U = \Z \rho \oplus \Z \rho'$.  For $\lambda \in II_{25,1}$, define
$$\deg \lambda = (\lambda,-2\rho+\rho').$$
Then $\deg \lambda$ defines a $\Z$-grading on $g_{L}$, i.e.
$$g_{L} = \bigoplus\limits_{i \in \Z} \bigoplus\limits_{\deg \lambda = i} (g_{L})_{\lambda}.$$
Next, suppose $\lambda = \lambda' - m \rho + n \rho'$, $\lambda \in \Lambda, m \in \Z$ is a root of $(g_L)_0$.  Then by definition we must have $m + 2n = 0$.  But then either $m = n = 0$ or $m \rho + n \rho ' = n(2\rho + \rho')$.  In the second case,
$$\langle \lambda' + n (2 \rho + \rho'), \lambda' + n(2 \rho + \rho') \rangle = \langle \lambda', \lambda' \rangle + 4n^2 \geq 4.$$
Assume $v \in (g_L)_0 \subseteq P^1$.  Then write
$$v = \prod\limits_{i=1}^N s_i(-n_i) \otimes e^{\lambda'}$$
and observe that $L_0v = v$ is only possible if $\langle \lambda', \lambda' \rangle = 0$ or $\langle \lambda',\lambda' \rangle = 2$ hence in the case of the Leech lattice, $\lambda' = \lambda = 0$ because the Leech lattice has no real roots.  Therefore, all elements of $(g_L)_0$ are in fact of the form $h_i(-1)$ for an orthogonal basis $h_1,...,h_{26}$ of $II_{25,1} \otimes \C$.  Furthermore, because $g_L$ is defined to be the quotient of $II_{25,1}$ by the radical of the bilinear form $(.,.)$, each of these elements survive in the quotient.  Therefore, we conclude that
$$(g_L)_0 \simeq II_{25,1} \otimes \C.$$

From there it is easy to see that the first two conditions of Theorem \ref{generalizedkacmoodycharacterization} are satisfied.  Let $\theta$ be the Cartan automorphism defined in Section \ref{sectionoghost}.  It is clear that $\theta$ acts as $-1$ on $(g_{L})_0$.  The remaining conditions of Theorem \ref{generalizedkacmoodycharacterization} are a technicality : indeed we may define a new bilinear form $(.,.)$ by way of $\theta$, i.e.
$$(x,y)_{new} = -(x,\theta(y))_{old}$$
and the condition follows from the fact that $(.,.)_{old}$ is positive-definite.  Therefore, we have proven :
\begin{corollary}
The Lie algebra $g_{L}$ is a generalized Kac-Moody algebra.
\end{corollary}

Because $g_L$ is graded by $II_{25,1}= \Lambda \oplus U$, it is also graded by $U$.  In particular, we are interested in the following character
$$\ch_U (g_L) = \sum\limits_{m > 0} \sum\limits_{n \in \Z} \dim (g_L)_{m,n} p^m q^n$$
where $(g_L)_{m,n}$ is the space of all elements of $g_L$ with root $(\lambda,m,n)$ for some $\lambda$.  This character was described explicitly in the proof of \cite{borcherds1}, Section 4, Lemma 2.  We repeat the proof here.

\begin{lemma} \label{leechgrading}
$\ch_U (g_L) = \sum\limits_{m > 0} \sum\limits_{n \in \Z} c(mn) p^m q^n$
where $c(mn)$ is the coefficient of $$\Theta_{\Lambda} (q) \sum\limits_{l \geq -1} p_{24}(1+l) q^l$$ with
$$\Theta_{\Lambda}(q) = \sum\limits_{\lambda \in \Lambda} q^{\lambda^2/2}.$$
\end{lemma}

\begin{proof}
Our goal is to count all elements of $g_L$ with root which projects on $m \rho' + n \rho$.  For a fixed root $(\lambda,m,n)$, we know there are $p_{24}(1 - \lambda^2/2 + mn)$ such linearly independent elements.  Summing over all $\lambda \in \Lambda$, we find that
$$c(mn) = \sum\limits_{\lambda \in \Lambda} p_{24}(1 - \lambda^2/2 + mn)$$
which is the same as the coefficient of $q^{mn}$ in the sum
$$\sum\limits_{\lambda \in \Lambda} q^{\lambda^2/2} p_{24}(1 - \lambda^2/2 + mn) q^{- \lambda^2/2 + mn}$$
but also in the sum
$$\sum\limits_{\lambda \in \Lambda} q^{\lambda^2/2} \sum\limits_{l \in \Z} p_{24}(1 - \lambda^2/2 + l) q^{- \lambda^2/2 + l} = \sum\limits_{\lambda \in \Lambda} q^{\lambda^2/2} \sum\limits_{l \in \Z} p_{24}(1 + l) q^{l}.$$
Clearly, $\sum\limits_{\lambda \in \Lambda} q^{\lambda^2}/2 = \Theta_{\Lambda}(q)$ and the fact that $l \geq -1$ is merely a consequence of the fact that $p_{24}(1+l) = 0$ for all $l < -1$.  This completes the proof.
\end{proof}

As it turns out, we may relate this character with the elliptic $j$ function : $$\Theta_{\Lambda}(q) \sum\limits_{l \geq -1} p_{24}(1+n)q^n = j(q) - 720 = q^{-1} + 24 + 196884q + ...$$

Since we have found that $g_{L}$ is a generalized Kac-Moody algebra, it is natural to ask about its simple roots and its BKM matrix.  These results can be found in \cite{borcherds1}, which we summarize here.

\begin{lemma}
The real simple roots of of $g_L$ are the norm $2$ vectors $r$ of $II_{25,1}$ such that $(r,\rho) = -1$.
\end{lemma}

\begin{lemma}
The imaginary simple roots of $g_L$ are the positive multiples of $\rho$, with multiplicity $24$.
\end{lemma}

\begin{remark}
An explicit description of the imaginary simple roots can be seen with no-ghost theorem.  It is enough to describe the spaces $T(m \rho,\rho')$ :
$$T(m \rho,\rho') = \spn \{ h_i(-1) e^{m \rho} | 1 \leq i \leq 24 \}$$
where the $h_i$'s form an orthogonal basis for $\Lambda \otimes \C$.
\end{remark}

From these two lemmas, \cite{borcherds2} then uses the Borcherds-Weyl-Kac denominator formula to find :

\begin{theorem}
The Weyl-Kac denominator formula for $g_L$ is
$$\Phi = e^{\rho} \prod\limits_{r \in \Pi^+} (1-e^r)^{p_{24}(1 - r^2/2)} = \sum\limits_{\substack{w \in W \\ n \in \Z}} \det(w) \tau(n) e^{w(n \rho)}$$
where $\tau(n)$ is Ramanujan's tau function.
\end{theorem}

$\Phi$ has been widely studied \cite{borcherds2}, \cite{borcherds3}, \cite{gritsenko} and many more.

So far we have extracted a lot of information from $g_L$ by setting $N = \Lambda$, the Leech lattice.  However we may still examine $g_L$ from the perspective of different Niemeier lattices.  Recall that
$$R(N) = \{ \lambda \in N, (\lambda,\lambda) = 2 \}$$
is the set of real roots of $N$.  Since $II_{25,1} = N \oplus II_{1,1}$, again let $\rho,\rho' \in II_{25,1}$ such that $II_{25,1} = N \oplus \Z \rho \oplus \Z \rho'$ and $(\rho,\rho') = 1$. Then write $L_n = \{ \alpha \in II_{25,1} | (\alpha,\rho) = n\}$ and $g_{L_n} = \bigoplus\limits_{\alpha \in L_n} (g_L)_{\alpha}$.  We obtain from this grading the following theorem :

\begin{theorem}
Suppose $N$ is a Niemeier lattice other than the Leech lattice and $\rho$ an isotropic vector of $II_{25,1}$ such that $\rho^\perp / \rho = N$.  Write $\widehat{g}(N) = \widehat{g}(R(N))$.  Then 
$$g_{L_0} \simeq \widehat{g}(N),$$
as Lie algebras, $g_L = \bigoplus\limits_{n \in \Z} g_{L_n}$ is a graded $\widehat{g}(N)$-module and
$$g_{L_1} \simeq g_{L_{-1}} \simeq V_N$$
linearly.  Furthermore, $(g_{L_{-1}})^* = g_{L_1}$ with respect to the bilinear form $(.,.)$ on $V_L$.  
\end{theorem}

\begin{proof}
Let $\theta$ be a highest root in $R(N)$.  If $\Pi$ is a root basis for $R(N)$ then the set $\widehat{\Pi} = \{ \rho - \theta \} \cup \Pi$ is a root basis for $\widehat{R}(N)$, which we may enumerate as $\alpha_0,...,\alpha_d$.  Furthermore, the Dynkin diagram for $\widehat{\Pi}$ satisfies the conditions of Theorem \ref{vertexrepresentation}.  Fix $\alpha \in II_{25,1}$ and $v \in (g_L)_{\alpha}$.  Then write $\alpha = \alpha' + m \rho + n \rho'$, $\alpha \in L$.  We first consider the case $n=0$ :

All elements of $g_L$ satisfying $n = 0$ must be a linear combination of multiples of elements of the form
$$1 \otimes e^{\alpha}, \beta_i(-1) \otimes e^{m \rho}$$
where $\beta_i \in \widehat{\Pi} \otimes \C$, and depending on whether or not $(\alpha,\alpha) = 2$ or $(\alpha,\alpha) = 0$, respectively.  Note that we include in $m\rho$ the possibility that $\alpha = 0$.  If we define a map on generators of $\widehat{g}(N)$ into $P^1$ by
\begin{align*}
e_i &\rightarrow 1 \otimes e^{\alpha_i} \\
f_i &\rightarrow 1 \otimes e^{-\alpha_i} \\
h_i &\rightarrow \alpha_i(-1) \otimes 1
\end{align*}
then we see that by Theorem \ref{vertexrepresentation} and the Lie bracket defined on $P^1/ L_{-1} P^0$, we have $g_{L_0} \simeq \widehat{g}(N)$ as Lie algebras.

To see why $g_L$ is a graded $\widehat{g}(N)$-module, it is enough to observe that $\widehat{g}(N)$ is a subalgebra of $g_L$ and that the action by Lie bracket of the elements $e_i,f_i,h_i$ on an element $v \in g_{L_n}$ with $n \neq 0$ does not modify the coefficient of $\rho'$ hence they each preserve $g_{L_n}$.

Finally, we examine the cases $n = \pm 1$.  From the spectrum-generating algebra, with $c = \pm \rho$, $$(g_L)_{\alpha} = \spn \left\{ \prod\limits_{i=1}^k A^{a_{j_i}}_{-n_i} \otimes e^{\alpha'+ M\rho \pm \rho'} \right\}$$
 and $\left( \sum\limits_{i=1}^k n_i \right) + \frac{(\alpha',\alpha')}{2} \pm M = 1$. Then for any element $\prod\limits_{i=1}^k s_i(-n_i) e^{\alpha'} \in V_{N}$ we may define an isomorphism via $s_i(-n_i) \to A_{-n_i}^i$, $e^{\alpha'} \to e^{\alpha' + M \rho \pm \rho'}$, with $M$ being uniquely determined.
\end{proof}

\subsection{From Borcherds products}

The previous section provides an algebraic perspective to the following theorem by \break Gritsenko \cite{gritsenko}.
\begin{theorem}
Let $N$ be a Niemeier lattice other than the Leech lattice.  Then the first non-zero Fourier-Jacobi coefficient of $\Phi$ in this cusp, is, up to a sign, the Weyl-Kac denominator function of the affine Lie algebra $\widehat{g}(N)$ : 
$$\Phi(\tau,z,\omega) = \pm \eta(\tau)^{24} \prod\limits_{v \in \Delta_{+}} \frac{\theta(\tau,(v,z))}{\eta(\tau)} e^{2 \pi i h(R(N)) \omega}$$
\end{theorem}

Despite our algebraic description, let us repeat the proof of this theorem here as it provides an example of how we may reach the same conclusion, and slightly more, using the theory of Borcherds products.  Our first step is to use Theorem \ref{borcherdsproduct}, to obtain the following corollary :

\begin{corollary}
Let $N$ be a Niemeier lattice and consider $N \oplus II_{1,1} \oplus II_{1,1} \simeq II_{26,2}$. $F = \Delta(\tau)^{-1}$.  Replacing the bilinear form $(.,.)$ on $II_{26,2}$ by $-(.,.)$, the meromorphic function $\Psi_Z(Z,F)$ of Theorem \ref{borcherdsproduct} simplifies to
$$\Psi_Z(Z,\Delta^{-1}) = \exp(Z,\rho) \prod\limits_{\lambda \in \Pi^+} (1 - \exp((\lambda,Z)))^{p_{24}\left( 1 - \frac{(\lambda,\lambda)}{2} \right)}$$
where a Weyl chamber can be chosen such that $\Pi^+$ is described as
$$\Pi^+ = \{\lambda' + m \rho + n \rho' | m \geq 0 \text { or } m = 0, n > 0 \text { or } m = n = 0, \lambda' < 0  \}.$$
where $\lambda' > 0$ signifies that $(\lambda',\rho_f) > 0$ for a Weyl vector $\rho_f$ of $R(N)$.
\end{corollary}

\begin{proof}
Note that the only dependence on $N$ is on the choice of Weyl chamber.  We need to find a Weyl chamber $W$ such that $(\lambda,W) > 0$ correspond to the set $\Pi^+$.  It is enough to find a vector $r$ in $II_{25,1}$ such that the set $\Pi^+$ is obtained as the vectors $x$ such that $(x,r) > 0$.

Note that if $\lambda' \neq 0$ and $m,n < 0$, we must have $(\lambda,\lambda) \leq 2$ otherwise \break $p_{24}\left( 1 - \frac{(\lambda,\lambda)}{2}\right) = 0$, because $N$ is positive definite.  If instead $\lambda' = 0$ the only possibilities for $m,n$ are $m = -1, n = 1$ or $m = 1, n = -1$.  This will be important when we define $r$.

Let $\{x_i | i = 1,...,24 \}$ be an orthonormal basis $N \otimes \C$. We consider the numbers $a_i = (\rho_f,x_i)$.  It is clear that $|a_i|_{i=1}^{24} \leq C$ for some $C \geq 1$.  Now suppose $x = \sum\limits_{i=1}^{24} b_i x_i$ and write $B = \max\{ |b_i| | i=1,...,24 \}$.  Then $$|(\rho_i,x)| \leq \sum\limits_{i=1}^{24} |a_i b_i| \leq \sum\limits_{i=1}^{24} |b_i| C \leq 24BC$$
and
$$\sqrt{(x,x)} \geq \sqrt{\sum\limits_{i=1}^{24} b_i^2} \geq B.$$  Now observe once again that we must have $$(x,x) - 2mn \leq 2$$ otherwise $p_{24}\left( 1 - \frac{(\lambda,\lambda)}{2}\right) = 0$.  Assume that $m,n > 0$.  Then there exists a constant $D > 0$ such that either $m$ or $n$ is greater than $DB$.  On the contrary if $m,n < 0$ there exists a constant $CD>0$ such that either $m$ or $n$ is less than $-DB$.  We let $D$ be the smaller of these two constants.

Now define $r = \rho_f - \frac{24C}{D} \rho - (\frac{24C}{D} + 1) \rho'$.  It is clear that if $\lambda' = 0$ and $m > 0$ or $m = 0$, $n > 0$, $(\lambda,r) > 0$ and also that $(\rho - \rho',r) > 0$.  Now we compute for $\lambda' \neq 0$ :
$$(\lambda' + m \rho + n \rho', r) = (\rho_f,x) + \frac{24Cn}{D} + \left( \frac{24C}{D}+1 \right)m$$
and we see that by the choice of $B,C,D$ that the sign of this expression is determined by the signs of $m,n$ whenever $m$ and $n$ are non-zero.  Furthermore, if $m,n = 0$ then the sign of this equation is determined by $(\rho_f,x)$.  This proves the assertion.
\end{proof}

Write $Z = (\tau,z,\omega)$, $q = \exp(2 \pi i \tau)$, $r^l = \exp(2 \pi i (l,z))$, $s = \exp(2 \pi i \omega)$.  In \cite{gritsenko}, they use the explicit Weyl vector $\rho$ to be $\rho = (A,B,C)$ where
$$A = 1 + \frac{1}{24} \sum\limits_{l \in L} 1, B = \frac{1}{2} \sum\limits_{l > 0 \in L} l, C = \frac{1}{24} \sum\limits_{l \in L} 1 \\
$$
As a result, we can rewrite $\Psi_Z$ again to
$$\Psi_Z(Z,\Delta^{-1}) = q^A r^B s^C \prod\limits_{\substack{m,n \in \Z, l \in L \\(n,l,m) > 0}} \left(1 - q^n r^l s^m \right)^{p_{24}\left(1 - \frac{l^2 - 2mn}{2}\right)}.$$
Write
$$\psi_{N;C}(Z) = \eta(\tau)^{24} \prod\limits_{l > 0} \left(\frac{\theta(\tau,(l,z))}{\eta(\tau)} \right)^{p_{24}\left(1 - \frac{l^2}{2} \right)}.$$
This term correspond to the product in $\Psi_Z$ where $m = 0$ but also :
\begin{proposition}
The function $\psi_{N;C}(Z)$ corresponds to the right side of the Weyl-Kac denominator formula of $\widehat{R}(N)$ with specialization $p = r^{-l}$.
\end{proposition}
\begin{proof}
Recall that the Weyl-Kac denominator formula for $\widehat{R}(N)$ is given by 
$$e^\rho \prod\limits_{m \geq 1}(1 - q^m)^d \prod\limits_{\alpha \in \Delta_{fin}^+} \prod\limits_{n \geq 1} (1 - q^{n-1} p) (1 - q^n p^{-1}).$$

This is our objective.  When $m = 0$ we have the product
$$\prod\limits_{\substack{n \in \Z, l \in N \\(n,l) > 0}} \left(1 - q^n r^l \right)^{p_{24}\left(1 - \frac{l^2}{2}\right)}.$$
From there we have two cases.  If $l = 0$, then $p_{24}(1) = 24$ hence we can collect the term $\prod\limits_{m \geq 1} (1 - q^m)^d$.

In the second case, $l \neq 0$ is a real root and then $p_{24}(0) = 1$.  For each real root $l$, we know that $-l$ is also a real root.  Hence we obtain terms $1 - q^n r^l = 1 - q^n p^{-1}$ and $1 - q^n r^{-l} = 1 - q^n p$.  However $(n,l,0) > 0$ if $n \geq 1$ or $n = 0$ and $l < 0$.  This leads to the presence of the term $1-p$, but not $1 - p^{-1}$ in the product.

Collecting all the terms when $m=0$ we then recover the specialized Weyl-Kac denominator formula for $\widehat{R}(N)$.
\end{proof}

With these methods we may also acquire additional information for situations where $m \neq 0, \pm 1$.  It is difficult to approach the problem algebraically, but Hecke operators provide us with at least some formulas.

Write $N(n,l,m) = 1 - \frac{l^2 - 2mn}{2}$.  For the terms where $m > 0$ we take $\log$ of $\psi_Z$ to find
\begin{align*}
& \log \left( \prod\limits_{\substack{(n,l,m)\\ m > 0}} (1 - q^n r^l s^m)^{p_{24} (N(n,l,m))} \right) \\ &= - \sum\limits_{(n,l,m) > 0} p_{24}(N(n,l,m)) \sum\limits_{i \geq 1} \frac{1}{i} q^{in} r^{il} s^{im}
\\ &= - \sum\limits_{(a,b,c) > 0} \left(\sum\limits_{d | (a,b,c)} d^{-1} p_{24}(N(a/d, b/d, c/d))\right)q^a r^b s^c \\
&= -\sum\limits_{m \geq 1} m^{-1} \phi|T_{m} e^{2 m \pi i w} 
\end{align*}
where $\phi = \Delta(\tau)^{-1} \Theta_N (\tau,z)$.  Note that the last equality makes use of Theorem \ref{hecketransformation}.  Taking exponentials again, we find
$$\Psi_Z = \psi_{N;C}(Z) \exp \left( - \sum\limits_{m \geq 1} m^{-1} \phi |T_{m} e^{2 m \pi i w} \right).$$
In particular, we can use this description to compute the Fourier coefficients of $\Psi_Z$ with respect to $\omega$.  It is enough to expand the exponential and use properties of Hecke operators to evaluate how the coefficient of $w$ changes.

\begin{corollary}
The first few terms of the Fourier-Jacobi expansion of $\Psi_Z(Z,F)$ defined by $N$ are
\begin{align*}
\psi_{L;C}(\tau,z) e^{2 \pi i C w} \left( 1 - \phi(\tau,z) e^{2 \pi i w}  + \frac{1}{2} (\phi^2(\tau,z) - \phi(\tau,z)|T_{2} e^{4 \pi i w} + ... \right).
\end{align*}
\end{corollary}
These methods lead to explicit characters for $g_{L_m}$ where $|m| \geq 2$.
\newpage

\begin{singlespace}

\end{singlespace}

\end{document}